\documentclass[12pt,draftcls,onecolumn]{IEEEtran}
\usepackage{cite}
\usepackage{calrsfs}

\usepackage{dutchcal}
\usepackage{wrapfig}
\usepackage{amsmath,amssymb,amsfonts}
\usepackage{algorithmic}
\usepackage{graphicx}
\usepackage{algorithm,algorithmic}
\usepackage{hyperref}
\hypersetup{hidelinks=true}
\usepackage{textcomp}
\usepackage{dsfont}
\algsetup{indent=1.5em}
\usepackage{bm}
\usepackage{mathrsfs}

\let\labelindent\relax
\usepackage[shortlabels]{enumitem}
\usepackage{graphicx}
\usepackage{tikz}
\usepackage{nicefrac}
\newtheorem{thm}{Theorem}

\newtheorem{cor}{Corollary}
\newtheorem{lem} {Lemma}

\newtheorem {Def}{{\bf Definition}}
\usepackage{bbm}
 
 \newcommand{\Sscr}{{\mathscr S}}
 \newcommand{\X}{{\cal X}}

\newcommand{\y}{{\bm y}}
\newcommand{\h}{{\bm h}}
\newcommand{\g}{{\bm g}}
\newcommand{\gam}{{\bm \gamma}}
\newcommand{\bsig}{{\bm \sigma }}
\renewcommand{\u}{{ {\bm u}}} 
\newcommand{\U}{{\mathcal U}}
\newcommand{\bdot}[1]{{\bm \dot{#1}}}

\newcommand{\x}{{\bar {\bf x}}}
\renewcommand{\u}{{\bar {\bf u}}}
\newcommand{\eop}{{\hfill $\square$}}

\newcommand{\ignore}[1]{}

\renewcommand{\y}{{\bf y}}

 \newcommand{\Remove}[1]{}
\newcommand{\G}{{\mathcal{G}}}
\newcommand{\hash}{{\mbox{\tiny $\#$}}}
\newcommand{\R}{{\cal R}}
\newcommand{\muL}{\mu_{leb}}

\newcommand{\N}{{\mathscr N}}
\newcommand{\A}{{\mathscr A}}

\newcommand{\vp}{v_p}
\newcommand{\rwc}{\stackrel{rw}{\longrightarrow
} }

\DeclareMathAlphabet{\mathdutchbcal}{U}{dutchcal}{b}{n}

\newcommand{\vr}{v_r}

\newcommand{\Dw}{{\mathcal{d}}}
\newcommand{\Up}{{\mathcal{u}}}
\newcommand{\I}{{\mathcal I}}
\newcommand{\T}{{\mathcal T}}
 \renewcommand{\u}{{ {\bm u}}} 
  \renewcommand{\x}{{ {\bm x}}} 
  \renewcommand{\P}{{\cal P}}

\newcommand{\dt}[1]{ \stackrel{\mbox{\Large{$.$}} }{#1} }

\usepackage{ulem}
\def\BibTeX{{\rm B\kern-.05em{\sc i\kern-.025em b}\kern-.08em
    T\kern-.1667em\lower.7ex\hbox{E}\kern-.125emX}}
\begin{document}
\title{ Optimal Control with $L^\infty$ and Integral Cost Functionals}
 \author{Madhu Dhiman, Veeraruna Kavitha, Nandyala Hemachandra
 \thanks{Department of
Industrial Engineering  and Operations Research, 
IIT Bombay, India. 
  Emails: \{madhu.dhiman, vkavitha, nh\}@iitb.ac.in.}}

\maketitle
 \thispagestyle{empty}
 \pagestyle{empty}
 
\begin{abstract} 
Many control problems are classically formulated using integral costs  that capture the cumulative performance of a system. 
Peak  or worst-case behavior is captured via supremum or $L^\infty$-costs in another variety of control problems. 
When both cumulative and peak performance are important, it is natural to consider objective functions that combine the two costs.  
 Although each criterion is well studied, their combination has not been explored extensively and we precisely work on  such control problems.

Towards establishing the existence, we first consider the relaxed    framework, where control is considered using probability distributions.  Using the well-known compactness and convexity properties of such control spaces, we establish the existence of an optimal relaxed control---we eventually establish the existence of an $\epsilon$-optimal pure (or classical) control, for every $\epsilon > 0.$

Despite these existence results, computing optimal policies remains challenging due to the non-smoothness introduced by the supremum term, and it is not clear whether the dynamic programming principle holds for our combined problem. To address this, we introduce a family of smooth approximations that yield standard control problems with well-defined optimal (pure) solutions. Using Maximum theorem,   we establish that  the solutions of the  smooth problems
among pure controls form $\epsilon$-optimal for the original problem, with $\epsilon$ tending to zero as   the smoothing parameter converges to zero.

Finally using the  methods proposed in this paper, we study a queueing  problem to illustrate (among others) that the required  trade-off between peak congestion levels and cumulative performance can be achieved.
\end{abstract}


\section{Introduction}
\label{sec_intro}
Optimal control theory studies the problem of determining a control policy for a dynamical system so as to optimize a prescribed performance criterion over a given time horizon. Classical formulations of optimal control problems typically involve an objective functional that consists of an integral of a running reward together with a terminal reward (see \cite{ fleming2012deterministic, fleming2006controlled, warga2014optimal}). More precisely, for a system governed by the dynamics $\dot{x}(s) = f(s,x(s),u(s))$,  with initial condition $ x(0)=x_0,$
the objective is usually of the form
$J^r(x_0;u) = \int_0^{t_1} L_r(s,x(s),u(s))\,ds + \Psi(x(t_1))$, where $u(s)$ is the control at time $s$, $\Psi(\cdot)$ represents the terminal  cost (or reward) function and $L_r(\cdot, \cdot, \cdot)$  represents the instantaneous cost (or reward) function. 
Such formulations arise naturally in engineering, economics, operations research, and applied mathematics, where the running reward represents the instantaneous performance of the system and the terminal reward captures the reward  associated with the state at the end of the horizon.

While the classical formulation with integral rewards captures the cumulative   performance of the system, in many practical applications it is equally important to account for extreme or peak behavior over the given time horizon. Such peak performance criteria have also been considered in the literature (see \cite{barron1990pontryagin, barron1999viscosity, barron1989bellman, di1999minimax,fialho1999worst, lygeros2004reachability, serea2002discontinuous, vinter2005minimax})---by minimizing the objective functions of the form $J^\infty( x_0; u) := \sup_{s \in [0, t_1]} L_{mx} (s, x(s))$---where $L_{mx} (s, x(s))$, at any time~$s$,     represents the instantaneous value of a certain quantity whose peak value over the horizon needs to be controlled---and solution $x(\cdot)$ is controlled using $u (\cdot)$  as before.

    Optimal control problems of both of the above-mentioned types have been extensively studied in the literature
 (see e.g., \cite{barron1989bellman, fleming2012deterministic,fleming2006controlled}). 
The value function of these control problems satisfies a Hamilton--Jacobi--Bellman (HJB) partial differential equation that arises from Bellman's dynamic programming principle. The HJB equation provides necessary and sufficient conditions for optimality and characterizes  the value function of the control problem. Alternatively, Pontryagin's maximum principle  provides necessary optimality conditions through a Hamiltonian associated with HJB equation, involving the state and adjoint variables (see e.g., \cite{barron1990pontryagin, fleming2006controlled}). These two techniques  are used to numerically solve both types of control problems mentioned above: (a) problems with integral (or running) rewards; and  (b) problems with $L^\infty$ rewards.

In several applications, both aspects, the cumulative running performance and the peak levels, play a crucial role and must be controlled simultaneously.
For example, in an inventory control problem, if there is a requirement and flexibility to design the inventory capacity, one needs to optimize a combined objective function formed using typical  holding/shortage running costs as well as the peak inventory levels reached during the planning horizon (e.g., see \cite{dhiman2024optimal}). 
%
In queueing systems, in a similar way,  the peak congestion levels reached during the operating period impacts an important design aspect,  the capacity of the waiting room (in fact here the peak levels also represent an important quality of service, important from customer standpoint). The  cumulative occupancy, the  cumulative server capacity utilized, etc., represent the typical running costs (e.g., \cite{bauerle2002optimal, malhotra2009feedback}). Our aim again is to provide solutions to such problems  that cater to both the types of costs. 
Similarly, in epidemic control models, the peak number of infected individuals is often of primary concern, as it reflects the maximum burden on the healthcare systems. Again there is a need to control a combined objective function that includes peak-infection levels  and other running costs, like cumulative vaccination costs (see, e.g., \cite{dhiman2023integrative}). 

Motivated by such applications, in this paper, we study a control problem that maximizes a combined  objective function of  the following  form, 
\begin{eqnarray}
    J( x_0;u) &=&   \int_{0}^{t_1} L_r(s, x(s), u(s))ds  + \Psi(x(t_1))  -  \sup_{s \in [0,t_1]} \{ L_{mx} (s, x (s)) \}.
    \label{eqn_objective}
\end{eqnarray}
The above  formulation  can capture a weighted combination of two or more objective functions,  with one of them being  the peak-level reached by a certain quantity during the operating horizon. 
In \cite{barron1999viscosity},  Barron  considered a class of optimal control problems where the objective function combines the two criterion in the following manner:
\begin{equation}
\sup_{t\in[0,t_1]}
\left\{
\int_t^{t_1} L_r(s,x(s),u(s))\,ds - L_{mx}(t,x(t))
\right\}
+ \Psi(x(t_1)). 
\label{Eqn_barron_obj_fun}
\end{equation}
Such formulations possess a structure that is particularly amenable to the dynamic programming approach and can be solved and analyzed  using a particular form  of HJB equations.
However, the objective in \eqref{Eqn_barron_obj_fun} differs fundamentally from that in \eqref{eqn_objective}. The former optimizes the running maximum of a quantity that, at each intermediate time~$t$, combines the   instantaneous penalty at~$t$ with the cumulative reward from time~$t$ to the end of the horizon $t_1$. 
In contrast, the objective function in \eqref{eqn_objective} combines the running maximum with accumulated rewards in a manner that naturally accommodates multi-objective optimization problems.

\ignore{
In contrast, the problem considered in this work involves study a general  control problem that involves a combination of a classical integral cost and an $L^\infty$-type peak cost over the entire time horizon. The objective function is given by
\begin{eqnarray}
    J( x_0;u) &=&   \int_{0}^{t_1} L_r(s, x(s), u(s))ds  + \Psi(x(t_1)) 
    -  \sup_{s \in [0,t_1]} \{ L_{mx} (s, x (s)) \}.
    \label{eqn_objective}
\end{eqnarray}
 This formulation generalizes both the classical optimal control problem with integral cost and the supremum-type formulations studied previously in the literature.

It is important to note that the above problem is different from a standard constrained optimal control problem. In a constrained setting, one typically imposes a condition such as 
$$
\sup_{s \in [0,t_1]} L_{mx}(s,x(s)) \leq \kappa, \mbox{ for some } \kappa< \infty, 
$$
and introduces a Lagrange multiplier term to handle this constraint, thereby converting the problem into an unconstrained one. In that case, the multiplier is chosen (or adjusted) so as to enforce the constraint indirectly.  In contrast, in the present formulation \eqref{eqn_objective} the supremum term appears directly in the objective with a fixed weight, and there is no parameter that can be tuned to recover an equivalent constrained problem. Thus, the problem is not a Lagrangian relaxation of a constraint, but instead reflects a genuine trade-off between cumulative performance and peak values, both of which are optimized simultaneously.}

Our first contribution is to establish the existence of suitable solutions  for problems that optimize \eqref{eqn_objective}. To this end,  we adopt the framework of relaxed controls. By exploiting   the well-known compactness and convexity properties of the relaxed control space, we establish the existence of an optimal relaxed control (see e.g.,~\cite{warga2014optimal}). Then  using the  denseness of pure controls in the relaxed control space (see again~\cite{warga2014optimal}), we prove the existence of  a pure control that is $\epsilon$-optimal for the original problem,   for every $\epsilon > 0$.

While these results guarantee existence, they do not directly lead to a tractable method for computing ($\epsilon$) optimal policies. The presence of the $L^\infty$ term introduces significant analytical challenges: the objective becomes non-smooth 
and it is not clear whether the classical dynamic programming principle applies to this setting.

To address this difficulty, we adopt an approach based on smooth approximations. We first augment the state by introducing an auxiliary variable, $y(t) := \sup_{s \le t}  L_{mx}(s,x(s)) $, that tracks the running maximum along the trajectory. Although this reformulation captures the path-dependent feature, the resulting problem remains non-smooth. We therefore consider a family of smooth problems that approximate the original formulation, leading to standard optimal control problems for which optimal solutions exist and can be computed numerically.
We then show that, as the smoothing parameter tends to zero, the solutions of the smooth problems converge to those of the original problem. This is established in two main steps: (i) the smooth approximate trajectory  converges to the instantaneous maximum trajectory $y(\cdot)$, and (ii) using the Maximum Theorem, the value functions of the smooth control problems converge to that of the original problem.



\ignore{
The presence of the supremum term introduces significant analytical challenges. Unlike classical optimal control problems, the $L^\infty$-term depends on the entire trajectory of the state process, and it is not a priori clear whether the classical dynamic programming framework applies.  Thus, we adopt an alternative approach based on relaxed controls. By using the compactness and convexity properties of the relaxed control space, we establish the existence of optimal relaxed controls. Furthermore, using the well-known denseness of pure controls in the relaxed control space, we show that for every $\epsilon>0$, there exists a classical control that is $\epsilon$-optimal for the relaxed problem.

While these results ensure existence and approximation, the non-smoothness induced by the $L^\infty$-term makes the problem computationally challenging.
To overcome this difficulty, we first introduce a maximum trajectory that tracks the running maximum of the $L^\infty$-term along the trajectory.  However, even in this augmented formulation, the objective remains non-smooth.

For computation tractability, we introduce a family of smooth approximations of the maximum trajectory, parameterized by a regularization parameter. These approximations replace the non-smooth maximum trajectory with a smooth surrogate 
 that enables analytical tractability, 
 allowing us to employ tools from classical optimal control theory. 
Within this framework, we prove the existence of  optimal control  for the smooth control problems and that the corresponding value functions converge, as the smoothing parameter tends to zero, to the value of the original non-smooth problem. Furthermore, we show that optimal solutions of the approximate problems are $\epsilon$-optimal for the original problem, with approximation guarantees that are uniform in the smoothing parameter.

These results collectively provide a systematic and rigorous approach to analyzing optimal control problems involving a combination of integral and peak cost criteria, and extend the applicability of optimal control theory to settings where both cumulative performance and peak behavior must be simultaneously controlled.}

\subsection*{Connection to hard constraint problem} It is important to note here that the  problem \eqref{eqn_objective} is different from a standard constrained  problem as in \cite{ maurer1977optimal, vinter1998necessary}---where one considers optimizing objective function,  $\int_0^{t_1} L_r(s,x(s),u(s))\,ds + \Psi(x(t_1))$, while   imposing a condition such as 
\begin{eqnarray}
   \label{Eqn_hard_constraint} 
\sup_{s \in [0,t_1]} L_{mx}(s,x(s)) \leq \kappa, \mbox{ for some } \kappa< \infty. 
\end{eqnarray}
Such problems are often addressed by introducing a Lagrange multiplier to incorporate the constraint into the objective, thereby transforming the constrained problem into an unconstrained one resembling \eqref{eqn_objective}. One then seeks an appropriate multiplier that guarantees satisfaction of the original constraint (see \cite{clarke2013functional, gollmann2009optimal, malanowski2004second}). In contrast, our problem is solved for a prescribed multiplier, which directly represents the desired trade-off factor. The two problems also differ fundamentally from the standpoint of applications. A formulation with a hard constraint is appropriate when, for example, a resource is subject to a strict upper bound. In contrast, our formulation is applicable when there is flexibility in determining the optimal level of the resource, with the trade-off governed by the prescribed multiplier.

Another closely related class of problems is that of optimal control with state constraints (see \cite[Section 7]{barron1999viscosity})---these problems can also be seen as problems with hard constraint as in \eqref{Eqn_hard_constraint}. 
Once again, for such problems, an appropriate Lagrange multiplier must be selected, whereas our problem seeks a solution for a given multiplier or trade-off factor.


Moreover, for optimal control problems with state constraints, it is well known that establishing a viscosity solution framework is, in general, a challenging and unresolved issue (see \cite[page~7]{barron1999viscosity}). The control formulation involving \eqref{eqn_objective} is of comparable complexity due to the presence of the supremum term. We overcome this difficulty by constructing a family of smooth approximating problems that converge to the original problem as the smoothing parameter tends to zero, thereby deriving the solution to the original problem through a suitable limiting argument.


Finally, the control framework developed in this work can be effectively applied to a broad range of real-world problems, including electricity load balancing, queueing systems, inventory control, epidemic control, etc. In this paper, we also investigate a queueing control problem as an example application of the proposed framework, while another example concerning inventory control is studied in \cite{dhiman2024optimal}. In our earlier short paper \cite{dhiman2024optimal}, we focused only on the smooth-approximation framework (one specific example) and presented inventory control application. In contrast, the present paper focuses on the combined problem \eqref{eqn_objective}, establishing the existence of its $\epsilon$-optimal solutions through solutions of a sequence of smooth approximating problems that can be solved using standard techniques --- using the relaxed control framework and applying the Maximum theorem. 

\subsection*{Queueing $L^\infty$ control}
  The aim here is  to study the contrasts in the  designs that cater to the two congestion costs (peak as well as cumulative) with varying degrees of importance. Using the $L^\infty$ smooth approximate framework  developed in this paper, we obtain the required comparison. We basically study the  two different (and extreme) Pareto frontiers, one that trades-off between the peak-congestion levels and the cumulative server utilization, and the second that trades-off between the cumulative congestion levels and the cumulative server utilization.

We observe that there is  a significant   reduction in the peak-congestion levels (even up to $27 \%$) when one  includes the latter cost into the optimization problem; however this \textit{improvement is at the expense of the cumulative congestion cost, which degrades significantly ($20 \%$)}.  
%
Thus,  the design for  queuing systems must consider the weight factors for  the three costs judiciously and our proposed framework makes it possible to derive the corresponding optimal policy numerically.  
Interestingly, however, when the cumulative 
server utilization is either too small or is too high, there is negligible difference in the 
performance(s) with and without considering the peak-levels into the design.

 \ignore{ 


The presence of the supremum term introduces significant analytical challenges. Unlike classical optimal control problems, the $L^\infty$ component depends on the entire trajectory of the state process and therefore cannot be directly incorporated into the standard dynamic programming framework. To overcome this difficulty, we introduce a maximum trajectory that tracks the running maximum of the $L^\infty$ cost along the trajectory. This augmented state representation allows the problem to be reformulated as a standard optimal control problem in an extended state space.

The reformulated problem admits a dynamic programming characterization, where the corresponding value function satisfies a Hamilton--Jacobi--Bellman equation on the augmented state space. However, the resulting equation is nonstandard due to the presence of a max-type nonlinearity, which introduces additional analytical complexities. To address these issues, we develop a suitable approximation framework that replaces the supremum term by a smooth approximation, enabling the use of classical analytical techniques.

Under appropriate regularity assumptions on the system dynamics and cost functions, we establish the existence of optimal controls and characterize the value function as the unique viscosity solution of the associated Hamilton--Jacobi--Bellman equation. Furthermore, we show that the solutions of the approximating problems converge to the solution of the original problem, thereby providing a rigorous justification for the approximation scheme.

\medskip

\noindent \textbf{Contributions.} The main contributions of this paper are summarized as follows:
\begin{enumerate}
    \item We formulate a class of optimal control problems involving  combination of integral and $L^\infty$ cost criteria, which extends classical formulations in a nontrivial manner.
    
    \item We introduce an augmented state-space representation that transforms the nonstandard problem into an equivalent control problem amenable to dynamic programming \textcolor{blue}{why dp here?}.
    
    \item We derive the associated Hamilton--Jacobi--Bellman equation and establish a viscosity solution characterization of the value function.
      \textcolor{blue}{do we need to say this? and how?}
    \item We propose a smooth approximation scheme for the $L^\infty$ term and prove convergence of the corresponding value functions to that of the original problem.
    
    \item We establish existence of optimal controls and provide a rigorous analytical framework for control problems involving peak-cost objectives.
\end{enumerate}}

\ignore{
Optimal control theory studies the problem of determining a control policy for a dynamical system so as to optimize a prescribed performance criterion over a given time horizon. Classical formulations of optimal control problems typically involve an objective functional that consists of an integral of a running cost together with a terminal cost. More precisely, for a system governed by the dynamics $\dot{x}(t) = f(t,x(t),u(t)), \quad x(\tau)=x,$
the objective is usually of the form
$J(\tau,x;u) = \int_\tau^{t_1} L(t,x(t),u(t))\,dt + \Psi(x(t_1))$.
Such formulations arise naturally in engineering, economics, operations research, and applied mathematics, where the running cost represents the instantaneous performance of the system and the terminal cost captures the desired final state. Optimal control problems of this type have been extensively studied in the literature; see the classical monographs \cite{fleming2012deterministic,fleming2006controlled}. 
Within this framework, the value function of the control problem satisfies a Hamilton--Jacobi--Bellman (HJB) partial differential equation that arises from Bellman's dynamic programming principle. The HJB equation provides necessary and sufficient conditions for optimality and characterizes the value function of the control problem.

The development of optimal control theory has followed two complementary approaches. The first is the variational approach based on Pontryagin's maximum principle, which provides necessary optimality conditions through a Hamiltonian system involving the state and adjoint variables. The second approach is the dynamic programming framework pioneered by Bellman, where the value function satisfies the Hamilton--Jacobi--Bellman equation and optimal policies can be obtained from the maximization of the Hamiltonian. 
Over the past several decades, a large body of literature has investigated such problems for deterministic and stochastic systems, finite and infinite horizon problems, and various types of running and terminal costs.

While the classical formulation with integral costs captures the cumulative performance of the system, many practical applications require the control of extreme or peak values of certain quantities over the time horizon. In several engineering systems, the worst-case or peak value of a variable can be more critical than its cumulative behavior. For instance, in power systems and communication networks it is important to control peak loads or congestion levels; in inventory and production systems large instantaneous backlogs may be undesirable; and in risk-sensitive decision making one often seeks to limit extreme deviations rather than merely minimizing cumulative costs. Such considerations naturally lead to control formulations involving $L^\infty$-type performance criteria.

Motivated by such applications, several authors have studied optimal control problems involving supremum-type objective functions. In particular, Barron and Ishii \cite{barron1999viscosity} considered a class of optimal control problems where the objective involves a supremum over time of a function that combines the running cost and an instantaneous penalty term. A representative formulation considered in their work takes the form
\begin{equation}
\sup_{t\in[\tau,t_1]}
\left\{
\int_t^{t_1} L_r(s,x(s),u(s))\,ds - L_{mx}(t,x(t),u(t))
\right\}
+ \Psi(x(t_1)).
\end{equation}
Such formulations possess a structure that is particularly amenable to the dynamic programming approach. The associated value function satisfies a Hamilton--Jacobi equation that can be analyzed using the viscosity solution framework, and explicit solution representations can sometimes be obtained through Hopf--Lax type formulas. These works provide a rigorous PDE characterization of optimal control problems involving supremum-type criteria.

Despite these developments, the structure considered in the above formulation is somewhat restrictive from the viewpoint of applications. In many practical systems one simultaneously encounters two types of performance objectives: (i) cumulative performance measured through an integral of running rewards, and (ii) constraints or penalties associated with the peak value of a quantity over the entire time horizon. In such situations the decision maker must balance overall system performance with the need to limit extreme deviations. This naturally leads to optimization problems where both types of costs appear together in the objective function.

In this work we study an optimal control problem that involves a weighted combination of a classical integral cost and an $L^\infty$-type peak cost over the entire time horizon. The objective functional is given by
\begin{eqnarray}
    J(\tau, x;u) &=&   \int_{\tau}^{t_1} L_r(s, x(s), u(s))ds  + \Psi(x(t_1)) 
    - \sup_{s \in [\tau,t_1]} \{ L_{mx} (s, x (s)) \}.
\end{eqnarray}
This formulation generalizes both the classical optimal control problem with integral cost and the supremum-type formulations studied previously in the literature. The parameter $\sigma>0$ represents the relative importance assigned to controlling the peak cost compared with maximizing the accumulated running reward.

The presence of the supremum term introduces significant analytical challenges. Unlike classical optimal control problems, the $L^\infty$ component depends on the entire trajectory of the state process and therefore cannot be directly incorporated into the standard dynamic programming framework. To overcome this difficulty, we introduce an auxiliary state variable that tracks the running maximum of the $L_\infty$ cost along the trajectory. This augmented state representation allows the problem to be reformulated as a standard optimal control problem in an extended state space.

Using this reformulation, we establish the dynamic programming principle for the resulting control problem and characterize the associated value function. In particular, we derive the Hamilton--Jacobi--Bellman equation corresponding to the augmented control problem and show that the value function satisfies this equation in the viscosity sense. Furthermore, we provide a verification theorem that characterizes optimal controls in terms of the solution of the HJB equation.

The main contributions of this work can be summarized as follows. First, we introduce a new class of optimal control problems involving arbitrary weighted combinations of integral costs and supremum-type peak costs over the entire horizon. Second, we develop a dynamic programming framework for this problem by augmenting the state space with a variable that captures the running maximum of the $L^\infty$ term. Third, we derive the associated Hamilton--Jacobi--Bellman equation and establish verification results that characterize optimal controls.

The rest of the paper is organized as follows. In Section~\ref{Aux_state}, we reformulate the original problem by introducing an auxiliary state variable that captures the peak cost term. We then establish the dynamic programming principle for the resulting value function. Subsequently, we derive the Hamilton--Jacobi--Bellman equation associated with the control problem and establish a verification theorem that characterizes optimal controls.}

We introduce the problem in Section \ref{sec_general_problem} and establish the existence results in Section \ref{sec_existence_results}. The smooth approximation framework is developed in Section \ref{sec_approximation_by_smooth}, followed by the numerical solution approach in Section \ref{sec_numerical_results}. A queuing control example is presented in Section \ref{sec_queu_control}, and we conclude in Section~\ref{sec_conclusions}.

\section*{Notation}  We use small letters to represent  variables  (e.g., $x(t)$, for some $t$), while the functions of time $t \in T$ are represented by bold small letters (e.g., $\x = (x(t), t \in T)$, $\bsig = (\sigma(t),t \in T) $). The dependency on parameters is  explicitly indicated only when there is a  requirement, and the parameters are shown in the subscript (e.g., $\x_{\bsig, x_0}$).

\section{Model description}
\label{sec_general_problem}

We consider  a non-standard, but highly relevant   optimal control problem in which the objective combines the conventional integral payoff over the finite horizon \(T=[0,t_1]\), where \(t_1<\infty\), with an \(L^\infty\)-term that accounts for the peak value attained by a certain quantity along the state trajectory. Specifically,
\begin{eqnarray}
\vp (x_0)&:=& \sup_{\u \in \U} J(x_0;\u), \qquad
\U := \{ \u:T  \to U \ | \ \u \text{ is measurable}\}, \label{Eqn_combined_problem}\\
J(x_0;\u) &=& \int_{0}^{t_1} L_r(s,x(s),u(s)) \ ds
+\Psi(x(t_1))
-\sup_{s\in T} \{L_{mx}(s,x(s))\}, \nonumber
\end{eqnarray}
subject to  \vspace{-3mm}
\begin{eqnarray} \dot{x}(s)=f(s,x(s),u(s)),  \quad \mbox{ with initial condition,} \  \  x(0)=x_0. 
\label{Eqn_x_pure_traj}
\end{eqnarray}
Here, \(L_r: T\times\mathbb{R}^n\times U\to\mathbb{R}\) is the running reward, \(f:T\times\mathbb{R}^n\times U\to\mathbb{R}^n\) defines the system dynamics with compact control set \(U\subset\mathbb{R}^p\), and \(\Psi:\mathbb{R}^n\to\mathbb{R}\) is the terminal reward for some $p, n < \infty$ (e.g., \cite{fleming2012deterministic, fleming2006controlled}). The distinguishing feature of \eqref{Eqn_combined_problem} is the inclusion of the \(L^\infty\)-term,  \(\sup_{s\in T} \{L_{mx}(s,x(s))\}\) with \(L_{mx}: T\times\mathbb{R}\to\mathbb{R}\), which   captures the maximum value attained by the performance measure $L_{mx}$ over the time horizon.

\ignore{
We consider a non-standard, but highly relevant control problem  involving a weighted combination of an integral term (running rewards accumulated over the given horizon $T = [0,t_1]$ for some $t_1 < \infty$)  and an $ L^\infty$-term (peak levels reached in the given horizon, $\sup_{s \in T} \{ L_{mx} (s, x (s) ) \}$) in  the objective  function  as  below:  
\begin{eqnarray}
    \vp (x_0)&:=& \sup_{\u \in \U} J(x_0;\u), \mbox{ with }   \U := \{ \u :T \to {U} ,  \mbox{ a measurable function}\},  \label{Eqn_combined_problem}\\
    J(x_0;\u) &:=&   \int_{0}^{t_1} L_r(s, x(s), u(s))ds  + \Psi(x(t_1)) -   \sup_{s \in T} \{ L_{mx} (s, x (s) ) \}, \mbox{ and, } \nonumber \\
 \mbox{ subject to, }   \  \bdot{x} &=& f(s, x(s),u(s) ),  \quad  \mbox{ and } x(0) = x_0, \nonumber
\end{eqnarray}
   for any given initial condition $x_0$;
here, $L_r:T \times \mathbb{R}^n \times  U \rightarrow \mathbb{R}$ is the running reward function, $f:T \times \mathbb{R}^n \times  U \rightarrow \mathbb{R}^n$ drives the dynamics  (for compact control set   $ U \subset \mathbb{R}^p$, where $n,p < \infty$) and $\Psi: \mathbb{R}^n \rightarrow \mathbb{R}$ is the terminal reward function as  considered in the standard optimal control literature (e.g., \cite{fleming2012deterministic, fleming2006controlled}). Importantly, there is an additional inclusion of     the $L^\infty$ term $\sup_{s \in T} \{ L_{mx} (s, x (s)) \}$  in \eqref{Eqn_combined_problem}, with  $L_{mx} : T \times \mathbb{R} \to \mathbb{R}$. 
}

 As already mentioned, the formulation in the literature closest to \eqref{Eqn_combined_problem} involves optimizing \eqref{Eqn_barron_obj_fun} over controls $\u$ (see, e.g., \cite{barron1999viscosity}). This class of problems admits the Dynamic Programming (DP) principle, leading to an appropriate HJB equation (e.g., \cite{barron1999viscosity}). However, the type of combination considered in \eqref{Eqn_barron_obj_fun} does not adequately capture multi-objective problems—particularly those in which some objectives require controlling $L^\infty$  or peak terms over the entire (or partial\footnote{One can control peak levels over any partial horizon $T' \subsetneq T$ by setting $L_{mx}(s,\cdot)\equiv 0$ for $s \in T'$.}) horizon, while others involve multiple integral-type utilities defined through running rewards over the same horizon.
In contrast, the formulation in \eqref{Eqn_combined_problem} accommodates such weighted combinations.

\subsection{Some important applications}
\label{sec_applications}
We now present some applications that require an objective as in \eqref{Eqn_combined_problem} that accounts for both cumulative and worst-case performance.

   \subsubsection{Queueing problem} 
   \label{subsec_queue_example}Consider a queuing system with a single server that faces fluctuating demands. Let ${\bm \alpha} = (\alpha(t), t \in T) $ represent the time-varying arrival rate to the system and ${\bm \mu} =  (\mu(t, u), t \in T, u \in U)$ 
   represent the controlled service rate process chosen by the controller, see Notations. Then, the queue length ${\bm x} $ process at the fluid limit can be approximated by an ODE, when controlled using policy $\u  $ (see \cite{mandelbaum1998strong}, \cite{pender2017approximations}):
\begin{equation}
    \dot{x}(s)
    = \alpha(s) -  \mu(s;u(s)), \quad \forall s \in T, \mbox{ and initial condition } x(0) = x_0.   
    \label{Eqn_queue_dynamics}
\end{equation}
  The literature (e.g., \cite{aland2011exact, bauerle2002optimal, malhotra2009feedback}) typically considers congestion cost given by a function $\g$  that depends upon the cumulative number  in the system and a cost for server utilization    given by function ${\bf h}$,  and considers optimizing a combined problem  as below: 
\begin{equation*}
\sup_{u  } \ 
   - \left (  \int_0^{t_1} \left[ \rho  g(x(s)) + h \big (\mu(s;u) \big ) \right ] ds + \Psi(x(t_1))\right), \ \mbox{ where, } \Psi(x) = \eta g(x),
\end{equation*}%
where $\rho>0$ is the trade-off factor between the two costs and    $\eta \ge 1$ represents the possibly amplified cost at the terminal time (to potentially account for possible loss of customers or for possible extended time period to complete the service of the left-over customers).   

We now consider another important Quality of Service (QoS) metric, namely the peak congestion level attained during the operating period. This can be formally captured by the instantaneous cost function
$L_{mx}(s,x):=  x, \mbox{ for all } x, $
which induces the $L^\infty$-performance criterion
$\sup_{s \in T} \{ x(s) \}$ as in \eqref{Eqn_combined_problem}.
This QoS is important from the perspectives of both the system and the customers. It reflects the physical capacity constraint of the waiting room and/or a maximum burden on the system; it can also represent an important factor that can dither away the customers for future service considerations of the same system. We now propose to optimize the following combined cost that also incorporates the peak congestion levels:
\begin{eqnarray}
\sup_{u  }  \ \  \left (  -\beta \sup_{s \in T} \{ x(s) \} - \int_0^{t_1} \left [ \rho g(x(s)) +  h\big (\mu(s; u)\big ) \right ]  ds - \Psi(x(t_1)) \right ),  \label{Eqn_Queue_problem_with_peak}
\end{eqnarray}%
where $\rho,  \beta \ge 0$ are constants that represent the trade-offs between various components. We will analyze this problem in Section \ref{sec_queu_control}, after the $\epsilon$-optimal solutions  are proposed.

\ignore{
\subsubsection{Inventory problem \cite[Section III]{dhiman2024optimal}}
Consider a manufacturer producing a single product and maintaining inventory over a planning horizon $[0,t_1]$. 
The manufacturer dynamically controls both the production rate $\u$ and the price $\bm{p}$ to optimize inventory levels and steer customer demand. Accounting for time variations and price sensitivity, the demand at time $t$ is modeled as
$d(t) = \left(\alpha(t) - \beta(t)p(t)\right)^+,$
where ${\bm \alpha}$ denotes the time-varying market potential and ${\bm \beta}$ the price sensitivity function (see \cite{}).
Let $x(t)$ denote the inventory level at time $t$, which evolves as below,  when the initial condition $x(0)=x_0$:
\begin{eqnarray}
\bdot{x}(t) = u(t) - \left(\alpha(t) - \beta(t)p(t)\right)^+, \mbox{ for any } t \in T. 
\label{Eqn_inventory_dynamics}
\end{eqnarray}
Positive inventory represents stock on hand, while negative inventory represents backlog. Demand is met immediately at time $t$, when $x(t)>0$ and backlogged otherwise. 

\noindent{\bf Costs and revenue:}
The instantaneous costs and revenues at time $t$ are as follows. The manufacturer incurs a quadratic production cost $au^2(t)$ and earns an instantaneous revenue of $p(t)d(t)$. In addition, depending on whether the inventory level is in surplus or
shortage, the inventory cost is modeled by the piecewise quadratic function (see \cite{dhiman2024optimal}) $h(x)= C_s \mathds{1}_{\{ x<0\}} + C_h \mathds{1}_{\{ x\ge 0\}}$
\[
h(x)=
\begin{cases}
C_sx^2, & x<0,\\
C_hx^2, & x\ge0,
\end{cases}
\]
where $C_h,C_s>0$. The corresponding terminal inventory cost at $x(t_1)$ is
\[
h_{t_1}(x)=
\begin{cases}
C_s^{(t_1)}x^2, & x<0,\\
C_h^{(t_1)}x^2, & x\ge0,
\end{cases}
\]
where the terminal cost coefficients $C_h^{(t_1)},C_s^{(t_1)}>0$ may differ from $C_h$ and $C_s$.

The instantaneous costs and revenues at time $t$ are described as follows. The manufacturer incurs a quadratic production cost given by $a u^2(t)$ and earns an instantaneous revenue of $p(t)d(t)$. In addition, depending on whether the inventory level is in surplus or shortage, the manufacturer incurs an inventory-related cost represented by the piecewise quadratic function $h$ (see \cite{dhiman2024optimal}):
\[
h(x)=
\begin{cases}
C_s x^2, & x<0,\\
C_h x^2, & x\ge 0,
\end{cases}
\]
where $C_h>0$ and $C_s>0$ denote the holding-cost and shortage-cost coefficients, respectively.
Similarly, the terminal inventory cost associated with the final inventory level $x(t_1)$ is represented by
\[
h_{t_1}(x)=
\begin{cases}
C_s^{(t_1)} x^2, & x<0,\\
C_h^{(t_1)} x^2, & x\ge 0,
\end{cases}
\]
where $C_h^{(t_1)}>0$ and $C_s^{(t_1)}>0$ are the terminal holding-cost and shortage-cost coefficients, respectively. These terminal cost coefficients may differ from those incurred during the planning horizon.

Now we describe the instantaneous costs and revenues corresponding to time $t$. 
The manufacturer incurs quadratic production cost $au^2(t)$ and derives instantaneous revenue  $p(t)d(t)$. Further depending upon the shortage or surplus of the inventory, it incurs  an instantaneous cost  $h(x(t))$,  which is again captured by a  piece-wise quadratic function   $h$    for some  positive constants $C_h, C_s$   as below
 (see \cite{dhiman2024optimal}):
$$
h(x) =
\begin{cases}
C_s x^2, & x<0,\\
C_h x^2, & x\ge 0,
\end{cases}
\quad \mbox{ and }\quad 
h_{t_1}(x) =
\begin{cases}
C_s^{(t_1)} x^2, & x<0,\\
C_h^{(t_1)} x^2, & x\ge 0,
\end{cases}
$$
captures similar cost corresponding to  final inventory level for $x=x(t_1)$ which can have a different set of constants $C_h^{(t_1)}$ and $C_s^{(t_1)}$.

, capturing end-of-horizon effects such as backlog or liquidation costs; the  coefficients $C_h^{(t_1)}$ and $C_s^{(t_1)}$ allow these penalties to differ from the running costs.  Thus, the total profit over $[0,t_1]$ is given by
\begin{equation}
    \label{Eqn_obj_fun}
    J(x_0;u,p) = \int_0^{t_1}\Big(p(s)(\alpha(s)-\beta(s)p(s))^+ - au^2(s) - h(x(s))\Big)ds - h_{t_1}(x(t_1)).
\end{equation}
There is  another important aspect in the inventory management  that of minimizing the peak inventory levels reached during the planning horizon.  Clearly, the physical structure that holds the inventory, once constructed, can not be altered easily. This calls for a careful (or an optimal) design of the storage capacity. 
 
This leads to a joint optimization problem that optimizes a given weighted combination of the usual inventory costs (as in \eqref{Eqn_obj_fun}) and the physical storage capacity. Such a consideration is well understood to provide a  better policy, but can be (needs to be) implemented only when there is a flexibility to design storage of any required size.

Towards the above mentioned joint optimization problem, we  consider  ``supremum inventory levels'' or technically introduce the $L^{\infty}$ norm of the inventory levels into the optimization problem \eqref{Eqn_obj_fun} as below:
\begin{eqnarray*}
 v(x_0) &:=& \sup_{(u,p) } \bigg(J( x_0; u, p )  - \sup_{s \in [0,t_1]} \left\{  x(s) \right\} \bigg)    \mbox{ subject to } \eqref{Eqn_inventory_dynamics}. 
\end{eqnarray*}
In \cite{dhiman2024optimal}, we have analyzed this inventory control problem using the $\epsilon$-optimal solutions proposed in this paper. 
}

\subsubsection{Epidemic control problem} 
We consider typical epidemic propagation model, where  infection spreads through contacts between infected and susceptible individuals (say at rate $\bar \beta$) and the infected individuals recover at rate, say $\alpha$. With $i(t)$ representing the fraction of infected individuals at time $t \in  [0,t_1]$,  such propagation is captured by the well known  Susceptible–Infected–Susceptible, or briefly SIS, model 
  given by the following differential equation or ODE (see e.g., \cite{allen1994some, dhiman2023integrative, kuhl2021computational}):
\begin{equation}
    \frac{di(t)}{dt} =  i(t) \left( \bar{\beta}(1-i(t)) - \alpha \right)  \mbox{  for all } t \in T = [0, t_1]. 
    \label{Eqn_SIS}
\end{equation}
However, it is more realistic to consider a time-varying infection rate  rather than a constant  $\bar{\beta}$. There are several factors that can drive such variability, and a vast literature addresses this aspect (see, e.g., \cite{dhiman2023integrative, funk2010modelling, pastor2015epidemic, perra2011towards}). For example, in a recent work \cite{dhiman2023integrative}, the authors model infection rate as evolving with individuals’ responses to epidemic-related information circulating on online social networks (OSNs). Specifically, exposure to misleading or adversarial content can reduce precautionary behavior, whereas reliable or authentic information from health agencies promotes safer practices. Thus, misinformation tends to increase the infection rate, while authentic information
helps to reduce~it.

The authors in \cite{dhiman2023integrative},   consider a framework in which epidemic dynamics are coupled with the information spread process. Using stochastic approximation tools and some ideas from  two-time-scale 
framework applied for the approximating ODEs, the authors propose that such interactions can be captured by \eqref{Eqn_SIS}, after  replacing  
 $\bar{\beta}$ by infection-dependent and time varying spread rate,  $\bar{\beta} + w(i )$. 
 Here, \(w(t)=w(i(t))\) denotes the time-varying spread rate, where
$w(i) :=i(u_a-u_s)$---the constant \(u_a\) captures the intensity of misinformation dissemination by malicious individuals---while \(\u_s=(u_s(t),  t\in T)\) denotes the control policy implemented by a social planner  over the planning horizon $T$---specifically, \(u_s(t)\) represents the rate at which the authentic information is disseminated at time $t$. 
 The resulting controlled infection dynamics, for any policy $\u_s$, are given by:
\begin{eqnarray*}
\frac{di(t)}{dt} &=&  i(t) \bigg( \left(\bar{\beta} + (u_a - u_s(t)) i (t)\right)(1-i(t)) - \alpha \bigg),  \mbox{  for all } t \in T. 
\end{eqnarray*}
It is clear that in such problems, the social planner would be interested in minimizing the cumulative level of infected people, which can be captured by  $\int_0^{t_1} i(t) dt$ and a cost that results from the efforts to  spread the  authentic information on  OSNs, captured by $\int_0^{t_1} u_s(t) dt$. In these scenarios, another crucial performance-characteristic would be the 
  \textit{peak} levels of infection or $\sup_{t \in T} i(t)$. If the infection levels reach a  high  peak   at any time  point during the   horizon, it leads to a  huge burden on the healthcare system. 
Thus, the social planner ideally  would like to consider a multi-objective optimization problem,  which includes such peak levels---as in the following  for some appropriate trade-off factors $\rho, \gamma>0$:
\begin{eqnarray*}
 \sup_{u_s} \bigg(- \rho \sup_{t \in T} \{i(t)\} - \int_0^{t_1} \left( i(t) + \gamma u_s(t) \right) dt \bigg).
\end{eqnarray*}
One can solve the above  problem using 
the solution proposed in the paper as it  fits into the framework of \eqref{Eqn_combined_problem}-\eqref{Eqn_x_pure_traj}.

The formulation in \eqref{Eqn_combined_problem}-\eqref{Eqn_x_pure_traj} can capture many more applications. In \cite{dhiman2024optimal}, we study an inventory control problem, where control of peak inventory level can imply an optimal design of the storage capacity. We now  proceed towards  the existence results.

\section{Existence results}
\label{sec_existence_results}
Our first aim is to establish the existence of a solution to  the problem formulated in \eqref{Eqn_combined_problem}-\eqref{Eqn_x_pure_traj}, more precisely an $\epsilon$-optimal solution. The presence of the $L^\infty$-term implies deviation from the standard framework---we are not even sure if  the classical dynamic programming equations are applicable to such combined problems. 
To prove the existence of solution for \eqref{Eqn_combined_problem}-\eqref{Eqn_x_pure_traj}, we rather take the direct route of establishing `continuity' of the `objective function' and compactness of the `domain'.  
The  main idea is two-fold: (a) to extend the framework to include  relaxed controls and consider  the related weak-topology (see e.g., \cite{warga2014optimal}), that provides the required compactness and continuity properties; and (b) to obtain $\epsilon$-optimal solutions of  the original problem  using the existence in the extended framework and some denseness properties.  


 We begin by listing the required assumptions, which are similar to those considered in the majority of the literature (see, e.g., \cite{fleming2012deterministic,fleming2006controlled,roxin1962existence}). \textit{These assumptions ensure that the optimal control problem in \eqref{Eqn_combined_problem}--\eqref{Eqn_x_pure_traj} admits a solution among pure controls $\U$ when the $L^\infty$ term is absent, i.e., when $L_{mx} \equiv 0$. This can be established using standard results, such as those in \cite{roxin1962existence}. The main aim of this paper is to provide an $\epsilon$-optimal solution to the same problem when the $L^\infty$ control term is included, under the same set of assumptions.} 
\begin{enumerate}[label=\textbf{A.\arabic*}, ref=\textbf{A.\arabic*}]
\setcounter{enumi}{-1}
\item \textit{The control space ${U}$ is a compact set.} \label{assum_a0}

\item \textit{The function $f$ is continuous and   there exists    $K_f>0$, such that for all  $(s,u) \in T  \times U$: \\ 
 \textbf{Lipschitz continuity:}   ($|\cdot|$ represents the Euclidean norm), 
{\begin{eqnarray*}
    & |f(s, x, u)-f(s, x', u)|   < K_f| x- x'|, \mbox{ for all } x, x' \in \mathbb{R}^n. 
\end{eqnarray*}}
\textbf{Uniform  boundedness:} 
$|f(s, x, u)|   \leq  K_f (1 + |x|), \mbox{ for all } x \in \mathbb{R}^n. $}
\label{assum_a1}

\item \textit{The functions $L_r$, $L_{mx}$  and $\Psi$ are continuous.  Further,   for all $s \in T,$ 
$$|L_{mx}(s, x)-L_{mx}(s, x')|  < K_L| x- x'|, \mbox{ for all } x, x' \in \mathbb{R}^n, $$  for some appropriate  $K_L < \infty$.  }
\label{assum_a2}

\item  \textit{The function $L_{mx}$ is continuously differentiable. Further, with 
\begin{eqnarray}
\Delta_{mx} (s, x, u) 
:= 
\frac{\partial L_{mx}(s, x)}{\partial s} + \frac{\partial L_{mx}(s, x)}{\partial x} f(s, x, u), 
\label{Eqn_Delta_mx}
\end{eqnarray}
 there exists $K_\Delta >0$, such that for all $(s, u) \in T \times  U$ 
\begin{eqnarray*}
   \hspace{-7mm} | \Delta_{mx} (s, x, u ) - \Delta_{mx} (s, x', u )| \hspace{-1mm}&\hspace{-1mm}\le\hspace{-1mm}&\hspace{-1mm} K_\Delta|x-x'|,  \forall x, x' \in \mathbb{R}^n, \\
  \mbox{ and  } \quad  |\Delta_{mx} (s, x, u )| \hspace{-1mm}&\hspace{-1mm}\le\hspace{-1mm}&\hspace{-1mm} K_\Delta (1+ |x|), \forall x \in \mathbb{R}^n. 
\end{eqnarray*}
\label{assum_a3}
}
    \item \textit{$f(s, x, U):= \{f(s, x, u): u \in  U \}$ is a convex set, for each $(s, x) \in T \times \mathbb{R}^n$. }\label{assum_a4} 
    \item \textit{There exists $K_r >0$, such that for all $(s, u) \in T \times  U$ 
\begin{eqnarray*}
    \hspace{-7mm}  | L_r (s, x, u ) - L_r (s, x', u )| \hspace{-1mm}&\hspace{-1mm}\le\hspace{-1mm}&\hspace{-1mm} K_r|x-x'|,  \forall x, x' \in \mathbb{R}^n, \\
   \mbox{ and  } \quad |L_r (s, x, u )| \hspace{-1mm}&\hspace{-1mm}\le\hspace{-1mm}&\hspace{-1mm} K_r (1+ |x|),  \forall x \in \mathbb{R}^n. 
\end{eqnarray*}}
\label{assum_a5}
\end{enumerate}

\ignore{
\begin{thm}Value function $v(\cdot, \cdot, \cdot)$ satisfies the following PDE equation
    \begin{eqnarray}
\frac{\partial v}{\partial t} + \sup_{u \in {U}} \left\{ \frac{\partial v}{\partial x} f(t, x, u) + \frac{\partial v}{\partial y}f(t, x, u) \mathds{1} _{\{y \le L_{mx}(t,x)\}} + L_r(t, x, u)\right\} =0 \label{Eqn_HJB}\\
\mbox{ with terminal condition } v(T, x, y) = - \sigma  \vee_{\tau, y}^{t_1} ( x) + \Psi(x). \nonumber
\end{eqnarray}
\end{thm}
   }

We  now proceed towards detailing the relaxed controls.


 \subsection{Relaxed controls and continuity of trajectories}
\label{sec:relaxed_controls}
In this section, we introduce the notion of relaxed controls and the associated relaxed trajectories, which provide a convenient framework for establishing  the existence of optimal controls (see \cite{warga2014optimal}). 
 The space of admissible controls or   the set of pure control functions is given by  (see also \eqref{Eqn_combined_problem}): 
\begin{eqnarray}
   \U := \{ \u :T \to {U} \mid   \u \mbox{ is  measurable}\}.  \label{Eqn_Controls_space}
\end{eqnarray}
The key idea is to consider    a  bigger 
  space of controls that  is compact under some appropriate topology,  and   includes   the above $\U$, after some embedding. We basically consider the measurable functions on the space of   probability measures over set~$U$, as the controls in the extended framework; recall here  $U$ is compact.

Let $\mathcal{P}(U)$ denote the space of probability measures on the control set $U$, endowed with the topology of weak convergence. That is, a sequence $\{\gam_k\} \subset \mathcal{P}(U)$ converges to $\gam \in \mathcal{P}(U)$ if and only if (iff) 
\[
\int_U g(u)\, d\gamma_k(u) \to \int_U g(u)\, d\gamma(u), \mbox{ for  every }  \g \in C(U) := \{ \h:U \to \mathbb{R}, \mbox{ continuous} \},  
\]
the space of  continuous functions on~$U$ (recall in our notations $\g= (g(t), t \in T )$). It is well-known that  $\mathcal{P}(U)$ is itself  a compact metric space under this topology (see e.g., \cite[Theorem IV.1.4]{warga2014optimal}, where rpm$(U) = \P(U)$), when $U$ is compact as in \ref{assum_a0}.

A relaxed control is defined as a measurable mapping from the time interval $T$ into the space of probability measures $\mathcal{P}(U)$. And the     space of relaxed controls is given by:
\begin{eqnarray} \hspace{8mm} \quad  
    \Sscr := \left\{ \bsig: T \to \mathcal{P}(U),  \ \muL \text{ measurable} \right\}, \  \ \muL \mbox{ is the Lebesgue measure on } T.
    \label{eqn_calS_defn}
\end{eqnarray}
The space $\Sscr$ is again endowed with the topology induced by weak convergence, which we  briefly  refer to as  \textit{r-weak topology} and is  defined as below  (see also Appendix \ref{App_relaxed_control}).
\begin{Def} [r-weak convergence]
\label{defn_relaxed_weak_conv}
\textit{A sequence  of relaxed controls $\{\bsig_k\} \subset \Sscr$ converges to $\bsig$ in~$\Sscr$, represented by $\bsig_k \rwc \bsig$,  if
\begin{eqnarray}
\label{Eqn_convergence_sigma_seq_refined}
\hspace{1mm}
\int_{0}^{t_1} \int_U g(t,u)\, \sigma_k(t,du)\, dt
\longrightarrow   \   \int_{0}^{t_1} \int_U g(t,u)\, \sigma(t,du)\, dt ,  
\end{eqnarray}
for all $\g \in \G := L^1 (T, C(U))$ where the space    $\G $ can be described as below:
\begin{eqnarray}
\label{Eqn_calG_defn} \hspace{0mm}
    \G &:=& \left  \{ \g:T \times U \to R: \g \mbox{ is measurable in $t$, continuous in $u$ and } m_\g \in L^1(T) \right  \}, 
\end{eqnarray} 
with $m_\g(t) \ := \  \sup_{u \in U} |g(t,u)|$ for each  $t \in T$.} 
\end{Def}

With slight abuse of notation,  we say  $\bsig=\u \in \U$ , when 
$\sigma(t,\cdot)=\delta_{u(t)}$, the Dirac measure,  for every~$t$.
It is a classical result  that the space $\Sscr$ is compact and metrizable (under $r$-weak topology), $\U \subset \Sscr$  (can be embedded),   and  we reproduce the same  for the sake of completion (see   Appendix~\ref{App_relaxed_control} for more details).

\begin{lem}[{\bf \cite[Chapter IV, Theorems IV.2.1 and IV.2.6]{warga2014optimal}}]
\label{lemma_space_compact} \textit{Assume \ref{assum_a0}. 
(i) Then, the space $\Sscr$ is compact, convex and sequentially compact under r-weak topology, where convergence is given by $\rwc$ as in \eqref{Eqn_convergence_sigma_seq_refined}.  (ii) Further, $\Sscr$ is metrizable. (iii) Furthermore,   $\U$ in \eqref{Eqn_Controls_space} is dense in  $\Sscr$, when one embeds $u \in U$ as Dirac measure $ \delta_u$.} 
\end{lem}
Thus, the relaxed controls of $\Sscr$ extend the pure  controls of $\U$ by allowing the randomized selection of the actions at each time instant and    provide a compact, convex and metrizable domain or control space under the r-weak topology---this will be instrumental in  establishing the existence of solution for  \eqref{Eqn_combined_problem}-\eqref{Eqn_x_pure_traj} in the relaxed space $\Sscr$.  Further the pure  controls are dense in~$\Sscr$---this  helps in deriving $\epsilon$-optimal solutions among the original  $\U$.

Given a relaxed control $\bsig \in \Sscr$, we define the corresponding relaxed dynamics by `averaging' the system dynamics with respect to (wrt) the control measure---for the system \eqref{Eqn_combined_problem}-\eqref{Eqn_x_pure_traj},  we say
 $\x_{\bsig,x_0} = (x_{\bsig, x_0} (t), t \in T)$ is 
   the \textit{relaxed trajectory corresponding to $\bsig \in \Sscr$}, if it is
 the solution of the following ODE   with initial condition $x(0) = x_0$:
\begin{eqnarray}
\label{Eqn_relaxed_dynamics_refined}
\bdot{x}_{\bsig, x_0} (s) = f(s, x_{\bsig, x_0} (s), \sigma(s)) \ \forall s,  \mbox{ with } f(t, x, \sigma (t) ) := \int_U f(t, x ,  u)\, \sigma(t,du),   \forall t, x, \bsig. 
\end{eqnarray}
Similarly define, 
$L_r(t, x, \sigma (t)) : = \int_U L_r(t, x_\bsig(t),  u)\, \sigma(t,du)$ for any $t$.   
When $\bsig=\u \in \U$,  the solution of the relaxed ODE
\eqref{Eqn_relaxed_dynamics_refined} coincides with that of
\eqref{Eqn_x_pure_traj}. 
When there is clarity,  we drop~$\bsig$ and $x_0$ in subscript. We now consider the  control problem with relaxed controls and relaxed trajectories defined above, 
\begin{eqnarray}
\label{Eqn_combined_problem_relaxed}
  \vr (x_0)&:=& \sup_{\bsig \in \Sscr} J(x_0;\bsig ),  \mbox{   subject to \eqref{Eqn_relaxed_dynamics_refined}}, \mbox{ where, } \\
    J(x_0;\bsig ) &:=&   \int_{0}^{t_1} L_r(t, x(t), \sigma (t))dt  + \Psi(x(t_1)) -  \sup_{t \in T}  L_{mx} (t, x (t)) .  
 \nonumber
\end{eqnarray}
Note the above falls back to \eqref{Eqn_combined_problem}-\eqref{Eqn_x_pure_traj},  when the   domain $\Sscr$ is replaced with its subset~$\U$. 

 Under   assumptions \ref{assum_a0}-\ref{assum_a1}, the function $f(t,x,u)$ is continuous in $x$, measurable in $t$, and Lipschitz in $x$ uniformly over $u \in U$. Clearly, these properties are also satisfied by  the averaged vector field $f(t, x,\sigma(t))$. 
  This ensures \textit{the existence and uniqueness of   solution for relaxed ODE \eqref{Eqn_relaxed_dynamics_refined} under each $\bsig$, which ensures the problem is well defined}.

%
We next proceed towards  establishing the existence among relaxed controls.
As a first step, we further establish the compactness and continuity of relaxed trajectories, 
 which again is a direct consequence of the classical results in the theory of relaxed controls.
Towards this, we consider the joint space of relaxed trajectories, relaxed controls, and initial conditions related to \eqref{Eqn_relaxed_dynamics_refined}, as is considered in  \cite[Section V.0]{warga2014optimal} for some compact~$\I $: 
\begin{eqnarray}
\hspace{0mm}
   \A :=\big  \{ (\x,   \bsig , x_0 )  \in  C(T)  \times \Sscr  \times \I: \x = \x_{\bsig, x_0} \mbox{ is the unique solution of \eqref{Eqn_relaxed_dynamics_refined}  under} && \nonumber \\
&&\hspace{-70mm}  \bsig \mbox{ with initial condition }x(0) = x_0   \big \}.  \label{Eqn_A_set} 
\end{eqnarray}

We endow \textit{$\A  \subset \left (C(T)  \times \Sscr  \times \I \right )$, with the product topology}, where: 
\begin{itemize}
    \item $C(T)=C(T;\mathbb{R}^n)$ is equipped with the \textit{strong uniform topology}---here $\x_n \to \x$ iff  $\| \x_n - \x\|_\infty \to 0$, where  norm $\|\x\|_\infty := \sup_{s\in T} |x(s)|$;
    \item  $\Sscr$ is equipped with the \textit{r-weak topology} defined in \eqref{Eqn_convergence_sigma_seq_refined}; and
    \item  $\I \subset \mathbb{R}^n$  is a compact subset. 
\end{itemize} 
Basically, in this product topology, a sequence $(\x_n, \bsig_n, x_{0n}) $  converges to $(\x, \bsig, x_0) $ iff  
$$ 
\| \x_{n} - {\x}\|_\infty  \to  0, \ \bsig_n \rwc \bsig   \mbox{ and } x_{0n} \to x_0. 
$$ 

We now establish compactness of $\A$ and the joint stability of relaxed trajectories wrt variations in both relaxed controls and initial conditions (proof is in Appendix \ref{App_relaxed_trajectories}).
\begin{thm}[{\bf Compactness and continuity}]
\label{thm:relaxed_controls_refined}    \textit{Assume \ref{assum_a0}-\ref{assum_a1} and consider $\Sscr$ 
 defined in  \eqref{eqn_calS_defn} and $\A$ in \eqref{Eqn_relaxed_dynamics_refined}-\eqref{Eqn_A_set}.  
\begin{itemize}
    \item [(a)] The set $\A$ is sequentially compact under product topology. 
    \item [(b)]  We have continuity of relaxed trajectories in uniform topology  wrt $(\bsig, x_0) \in \Sscr \times \I$ in the following sense: 
    \begin{eqnarray}
\label{Eqn_thm1_partb}\qquad \quad \quad
  \mbox{if } \bsig_n \rwc \bsig \mbox{ and } x_{0n} \to x_0 , \mbox{ then }
  \| \x_{\bsig_n,  x_{0n}} - {\x}_{\bsig,  x_{0}}\|_\infty  \longrightarrow  0.  \hspace{7mm}  
\end{eqnarray}   
\end{itemize}}
\end{thm}
The above results are also well known (e.g., see \cite{warga2014optimal}), we reproduce them here again for the sake of completion; further we provide a simple proof of part (b) using Maximum theorem---more specifically using the parametric continuity provided by Maximum theorem (see \cite{berge1877topological}). 
We 
next establish the continuity of the objective function in \eqref{Eqn_combined_problem_relaxed} as final step towards establishing existence.

\subsection{Continuity of max-trajectories and objective function} For a given initial condition $x_0$ and relaxed control $\bsig$, the `running max'-trajectory $\y_{\bsig, x_0}  = (y_{\bsig, x_0}(t),\, t \in T)$ is defined  point-wise as 
\begin{eqnarray}
   \hspace{8mm} y_{\bsig, x_0}(t) := \sup_{s \in [0, t]} L_{mx}(s, x_{\bsig, x_0}(s)), \  \mbox{ for all } t \in T. 
    \label{Eqn_y_defn}
\end{eqnarray}
 \begin{wrapfigure}{r}{0.33\textwidth}
\begin{center}
\vspace{-12mm}
\includegraphics[trim = {1.15cm 11.0cm 0cm 1.6cm}, clip, width=4.7cm, height=3.8cm]{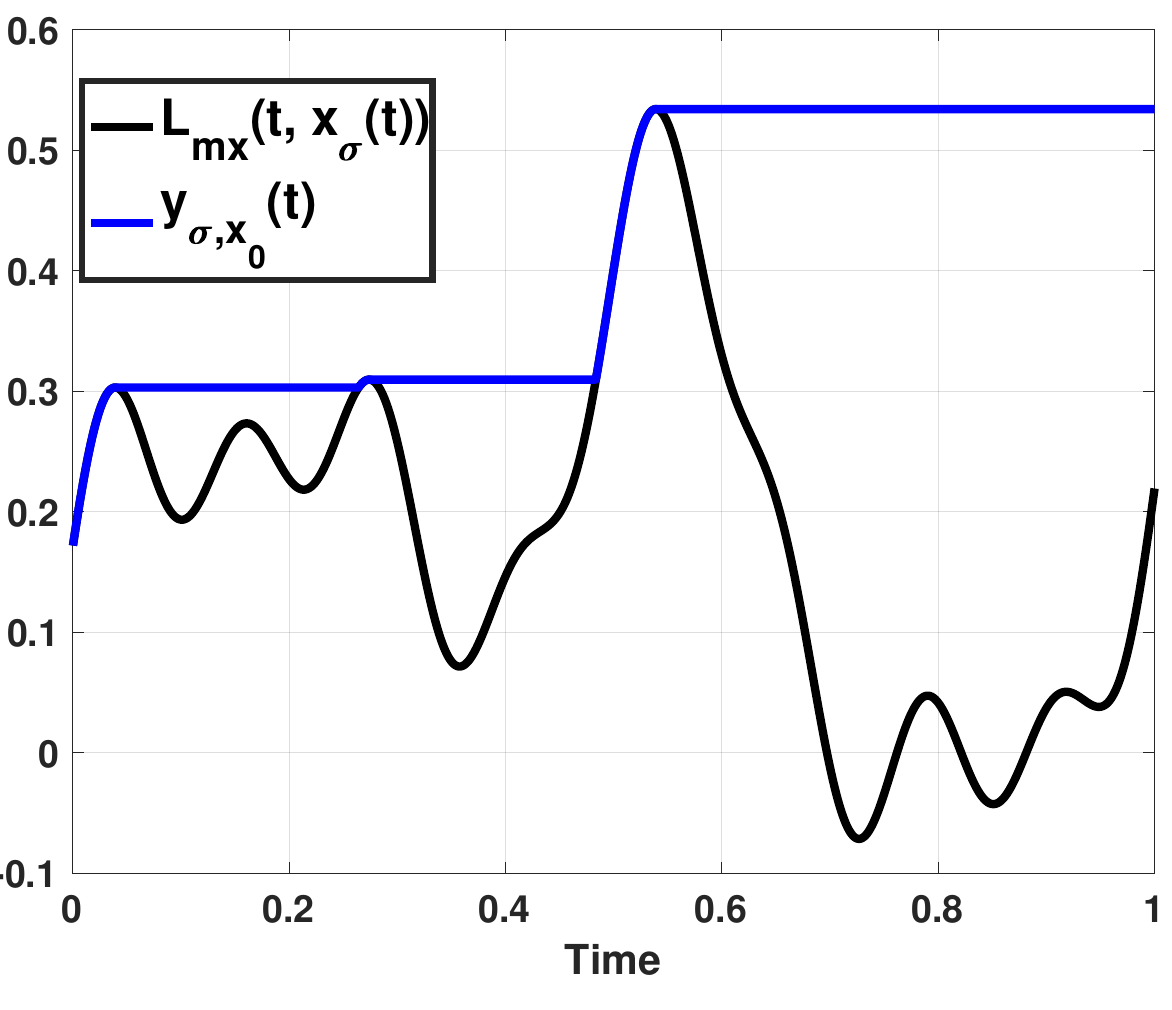}
 \vspace{-7.5mm}
\caption{$L_{mx}, y$ trajectories}
\label{Fig_x_y}
\vspace{-10.5mm}
\end{center}
\end{wrapfigure}
The value of this  function $ y_{\bsig, x_0}(t)$ captures the maximum running cost   up to   time $t$ and  $y_{\bsig, x_0}(t_1) $  precisely equals the $L^\infty$-term of  the combined objective~\eqref{Eqn_combined_problem}. We first establish the continuity of the mapping $(\bsig, x_0) \mapsto \y_{\bsig, x_0}$,  then that of the objective function of \eqref{Eqn_combined_problem_relaxed} and finally prove   the existence of an optimizer among relaxed controls (proof is in Appendix~\ref{App_thm3}). 
\begin{thm}[{\bf Optimal relaxed control}]  
\label{Thm_existence_relx_cont} \textit{Assume \ref{assum_a0}-\ref{assum_a2}, \ref{assum_a5} and consider   problem \eqref{Eqn_combined_problem_relaxed}. 
  \begin{itemize}
    \item [(a)]If $\bsig_n \rwc \bsig \mbox{ and } x_{0n} \to x_0 $, 
  then
\begin{itemize}
    \item[(i)] $\displaystyle \int_0^{t_1}  L_r(t, x_{\bsig_n,x_{0n}}(t),   \sigma_n (t))dt \longrightarrow      \int_0^{t_1}  L_r(t, x_{\bsig, x_0}(t), \sigma (t))dt $, and, 
   \medskip
    \item[(ii)] $\displaystyle  \left \|\y_{\bsig_n,x_{0n}}   - \y_{\bsig, x_0}   \right \|_\infty \longrightarrow   0 $.
\end{itemize}
  \item [(b)]   Moreover, there exists an optimal relaxed control  for \eqref{Eqn_combined_problem_relaxed}.
  \end{itemize}}
    \end{thm}
Part $(a).(i)$  follows,  for example, using the   results  of  \cite{warga2014optimal} and Dominant convergence theorem, while part $(a).(ii)$ is the first substantial result of this paper and  is instrumental in establishing the existence of part $(b)$ and in further   proving the main results. 

We also establish the continuity of the trajectories and the overall objective function wrt the initial condition \(x_0\). This result is of independent interest\footnote{One can prove the continuity of the value function $\vr (x_0, \delta=0)$ wrt initial condition $x_0$, in exactly similar lines---we have already established the continuity of $\x, \y$ and $L_r$  wrt   $x_0$ in Theorems \ref{thm:relaxed_controls_refined}-\ref{Thm_existence_relx_cont}.}, as it is useful for studying the stability of optimal solutions and may facilitate future study on the dynamic programming principle for problems of the form \eqref{Eqn_combined_problem}-\eqref{Eqn_x_pure_traj}.

\textit{We now present the first main result of this paper, which is an immediate corollary of the continuity results of Theorems \ref{thm:relaxed_controls_refined}-\ref{Thm_existence_relx_cont} and denseness of Lemma~\ref{lemma_space_compact}}:

\begin{cor}
\label{Corollary}\textit{Assume \ref{assum_a0}-\ref{assum_a2} and \ref{assum_a5}. 
For  any $\epsilon>0$, 
    there exist an
     $\u^\epsilon \in \U$, which is 
    $\epsilon$-optimal  for the original problem \eqref{Eqn_combined_problem}-\eqref{Eqn_x_pure_traj}. }
\end{cor}

We  have thus established the existence of $\epsilon$-optimal policies in $\U$. However,   computing these  policies can still be challenging---observe here    the $\y$ trajectory  \eqref{Eqn_y_defn} need not be smooth (need not be differentiable) and that the problem \eqref{Eqn_combined_problem}-\eqref{Eqn_x_pure_traj} is not in standard framework for which numerical solutions are known---as of now, it is not even clear whether the dynamic programming principle holds.
Our next focus is on deriving a numerical procedure to compute the solutions, and the idea is to approach
this  by constructing some smooth control problems  that can `approximate' \eqref{Eqn_combined_problem}-\eqref{Eqn_x_pure_traj}.

Towards this,  we  approximate the non-smooth
$\y$ trajectory with a smooth trajectory, which leads to  a `smooth' problem in standard control framework  and  hence   with computable solution. We then establish
some form of parametric continuity  that links  the  solutions  of `smooth' problems with that corresponding to original \eqref{Eqn_combined_problem}-\eqref{Eqn_x_pure_traj}, once again using the  Maximum theorem---thereby establishing  that 
the solutions of `smooth' problems  become $\epsilon$-optimal for \eqref{Eqn_combined_problem}-\eqref{Eqn_x_pure_traj}. 
We begin with introducing the smooth problem(s) that approximate~\eqref{Eqn_combined_problem}-\eqref{Eqn_x_pure_traj}. 
For this set of results, we consider fixed  initial condition $x_0$ and 
again employ the relaxed control framework to achieve our goal.

\section{Approximation by smooth control problems}   
\label{sec_approximation_by_smooth}
The max-trajectory defined in \eqref{Eqn_y_defn} enables rewriting the $L^\infty$-term  in \eqref{Eqn_combined_problem_relaxed}, $\sup_{s \in T} L_{mx}(s, x_\bsig (s))$,  as a terminal cost, since this term  precisely equals $y_{\bsig}(t_1)$. Let the modified objective function be defined as below:
\begin{eqnarray}
    &&W(\x_\bsig,  \y_\bsig; \bsig) := \int_{0}^{t_1} L_r(t, x_\bsig(t), \sigma (t))dt  + \Psi(x_\bsig(t_1)) -  y_\bsig(t_1). \label{eqn_obj_fun_with_y}
\end{eqnarray}   
Intuitively, the dynamics of this max-trajectory can be viewed as the solution of the  ODE,
\begin{eqnarray}
\label{Eqn_y_dynamic_compact}
&&\hspace{-5mm}\bdot{y}_\bsig(s)
= \bigg ( \Delta_{mx} (s, x_\bsig (s), \sigma(s) )  \bigg  )^+  \mathds{1}_{\{L_{mx}(s, x_{\bsig}(s)) \ge y_{\bsig}(s)\}},   \mbox{where }    \Delta_{mx} \mbox{ is defined in } \eqref{Eqn_Delta_mx}, 
\end{eqnarray}
due to the following reasons.  
The above dynamics imply that $\y_\bsig$ is non-decreasing and evolves   for $s \in [s_1, s_2] \subset [0, t_1]$, only  when the trajectory $L_{mx}(s,x_\bsig(s))$ meets the current level $y_\bsig(s)$ with a positive derivative i.e, $\Delta_{mx} (s, x_\bsig(s), \sigma(s)) >0$. In this regime, $y_\bsig(s)$ tracks $L_{mx}(s,x_\bsig(s))$ exactly. If for some time $s$,  $\dt{L}_{mx}(s,x_\bsig(s)) < 0$, then $L_{mx}(s,x_\bsig(s))$ decreases below $y_\bsig(s)$ and $\y_\bsig$ remains constant until $L_{mx}(s,x_\bsig(s))$ increases again and catches up $y_\bsig (s)$. Hence, $\y_\bsig$  the solution of \eqref{Eqn_y_dynamic_compact} (when exists),  effectively captures the running maximum of $L_{mx}(s,x_\bsig (s))$ or the max-trajectory \eqref{Eqn_y_defn}, when further the initial conditions are equal,  $x_\bsig(0) = y_\bsig(0)$ (see Figure~\ref{Fig_x_y} also).



However, a solution to ODE \eqref{Eqn_y_dynamic_compact} need not exist in general.  
Although the max-trajectory~$\y$ in~\eqref{Eqn_y_defn}  exists for any admissible control   $\bsig \in \Sscr$, due to the well-posedness of the $\x$-dynamics, it is generally non-smooth (may not be differentiable)---see the discontinuity in   the right-hand side  of the corresponding ODE~\eqref{Eqn_y_dynamic_compact}. In particular, indicator  $\mathds{1}_{\{L_{mx}(s,x_\bsig (s))\ge y_\bsig (s)\}}$ induces a state-dependent discontinuity, making the direct analysis  of the control problem in \eqref{Eqn_final_problem_compact} significantly more involved---such analysis typically requires handling of Filippov   solutions, defined via differential inclusions (see e.g., \cite{barles2013bellman, filippov2013differential, rao2013hamilton}).
Rather, we replace the ODE \eqref{Eqn_y_dynamic_compact} with a `smooth approximation version' and show that the solutions of the resulting smooth-control problems provide approximate solutions to the original problem \eqref{Eqn_combined_problem}-\eqref{Eqn_x_pure_traj}.
In \cite{dhiman2024optimal}, we studied a class of  control problems, and we now show that their solutions approximate the current problem \eqref{Eqn_combined_problem}-\eqref{Eqn_x_pure_traj}. In the immediate next, we in fact propose   a more general class of smooth   problems that can approximate. 

\subsection{The smooth control problems}
\label{subsec_smooth_approx}
The idea  is to smooth the max-trajectory~\eqref{Eqn_y_defn} by 
replacing  indicator $\mathds{1}_{\{L_{mx}(s,x_\bsig(s))\ge y_\bsig(s) \}}$  in \eqref{Eqn_y_dynamic_compact} with a  continuous  $\psi_\delta$, yielding
\begin{eqnarray}
\bdot{y}_\bsig^\delta(s) &=& \big ( \Delta_{mx} (s, x_\bsig (s),\sigma(s) )  \big  )^+  \hspace{-1mm}\psi_\delta(L_{mx} (s,x_\bsig(s))-y_\bsig^\delta (s))  , \mbox{ where }\nonumber\\
 \psi_\delta(d) &:=& \phi \left(1 + \frac{d}{\delta}\right)\mathds{1}_{\{d \in [-\delta,0]\}} + \mathds{1}_{\{d>0\}}, \label{eq:ydelta} 
\end{eqnarray}
 \begin{wrapfigure}{r}{0.27\textwidth}
\begin{center}
\vspace{-12mm}
\includegraphics[trim = {0.5cm 14.0cm 1.5cm 1.6cm}, clip, width=4.3cm, height=4.1cm]{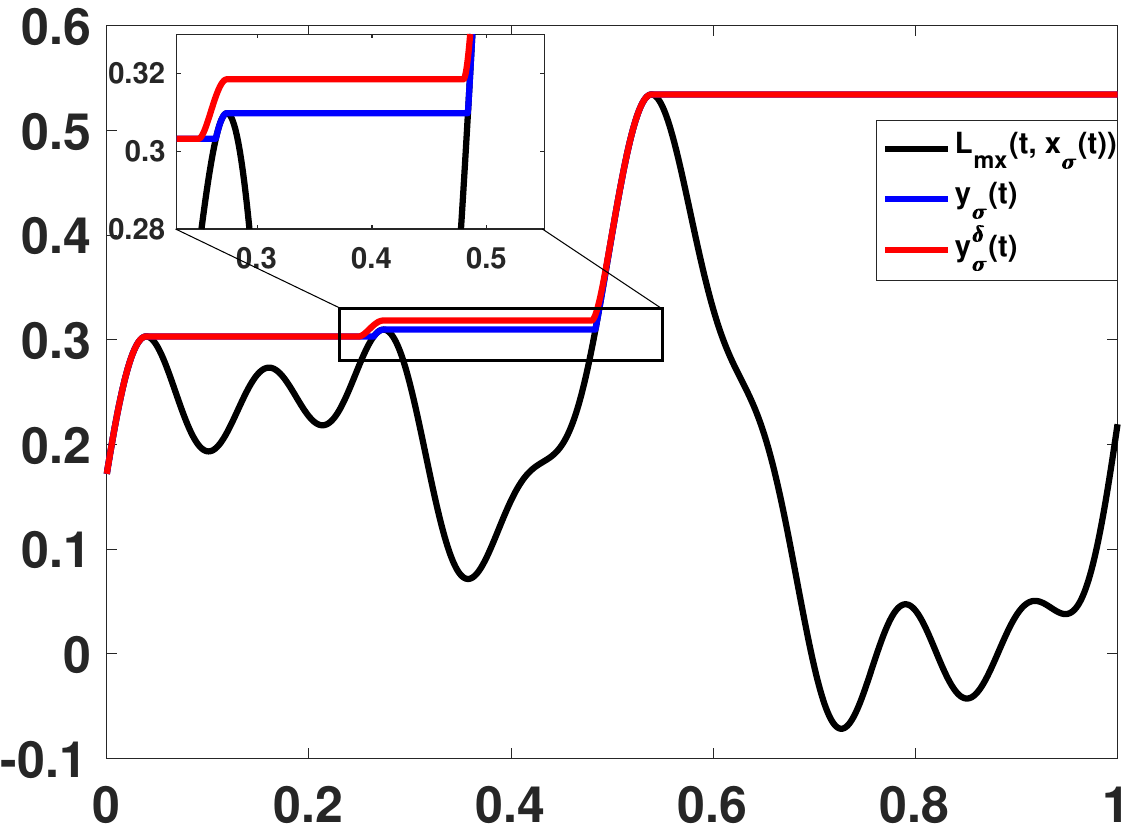}
 \vspace{-5mm}
\caption{$L_{mx}, \y_\bsig, \y_\bsig^\delta $ trajectories}
\label{Fig_y_delta}
\vspace{-2mm}
\end{center}
\end{wrapfigure} 
using  \textit{any Lipschitz continuous increasing function $\phi: [0,1] \to [0,1]$  satisfying, $\phi(0) =0$ and $\phi(1) =~1$.} This approximation (e.g., \cite{fleming2006controlled}) ensures Lipschitz continuity of the $\y$-ODE for any $\delta>0$, provided $\Delta_{mx}$ is sufficiently smooth, thereby enabling tractable analysis of the associated control problem. 

Now, we propose the   smooth problems  parametrized by $\delta > 0$, which can `approximate' ($W$ is defined in \eqref{eqn_obj_fun_with_y}): 
\begin{eqnarray}
&&\vr^\delta = \sup_{\bsig \in \Sscr} W(\x_\bsig,  \y_\bsig^\delta; \bsig) \quad \text{subject to }  \eqref{Eqn_relaxed_dynamics_refined} \mbox{ and } \eqref{eq:ydelta}, \mbox{ with } y_\bsig^\delta(0) = x_\bsig(0)=x_0. 
\label{Eqn_final_problem_compact_relaxed}
\end{eqnarray}
We first prove that a pure control in $\U$ optimizes the above smooth  problem:
\begin{lem}
\label{Lem_exist_optimal_control}
\textit{Fix any $\delta > 0$. Assume \ref{assum_a0}-\ref{assum_a3}.  Then the following are true for smooth-approximate control problem \eqref{Eqn_final_problem_compact_relaxed}. \\
(i)  The joint ODE 
\eqref{Eqn_relaxed_dynamics_refined}-\eqref{eq:ydelta}    has unique solution,  represented by   relaxed trajectories $(\x_\bsig, \y^\delta_\bsig)$, for any  $(\bsig, \delta)$. Further, 
$
\| (\x_{\bsig_n}, \y_{\bsig_n}^\delta) - (\x_{\bsig}, \y_{\bsig}^\delta) \|_\infty \to 0,$ for any sequence,  $\bsig_n \rwc \bsig$.\\
(ii) Additionally assume \ref{assum_a4}-\ref{assum_a5}. Then there exists an optimal solution $\u^*_\delta \in \U $ for the following counterpart of \eqref{Eqn_final_problem_compact_relaxed}, that optimizes   among the  pure controls: 
\begin{eqnarray}
&&\vp^\delta = \sup_{\u \in \U} W(\x_\u,  \y_\u^\delta; \u) \quad \text{subject to }  \eqref{Eqn_relaxed_dynamics_refined} \mbox{ and } \eqref{eq:ydelta}, \mbox{ with } y_\u^\delta(0) = x_\u(0)=x_0. 
\label{Eqn_final_problem_compact}
\end{eqnarray}
Moreover, $\u^*_\delta$ is also optimal for the relaxed smooth-approximate problem \eqref{Eqn_final_problem_compact_relaxed}}.
\end{lem}
\textbf{Proof:}  We   prove the existence of unique solution for joint ODE \eqref{Eqn_relaxed_dynamics_refined}-\eqref{eq:ydelta} and 
establish  the   continuity result of part (i), simultaneously       using Theorem \ref{thm:relaxed_controls_refined}.(b). 
Towards this, it suffices to show that the joint ODE \eqref{Eqn_relaxed_dynamics_refined}-\eqref{eq:ydelta} also satisfies the assumptions 
\ref{assum_a0}-\ref{assum_a1} required by the theorem. 
%
The  
$\x_\bsig$ components already satisfy the required conditions by \ref{assum_a1}
and hence Theorem \ref{thm:relaxed_controls_refined}.(b) is applicable to joint $(\x_\bsig, \y_\bsig^\delta)$, if we show that the RHS of the  $\y_\bsig^\delta$-ODE in \eqref{eq:ydelta} also satisfies the Lipschitz continuity and uniform boundedness conditions of \ref{assum_a1}. These conditions  are  proved in Appendix \ref{app_lemma_exitence}. We also have  the proof of   part (ii), which is established   using the existence (among pure controls) results of \cite{roxin1962existence},  denseness of Lemma~\ref{lemma_space_compact} and the  continuity of the mapping  $\bsig \mapsto W(\x_\bsig, \y_\bsig^\delta; \bsig)$,  provided by  Theorem \ref{Thm_existence_relx_cont} and part~(i).  \eop

Thus, for smooth problems \eqref{Eqn_final_problem_compact_relaxed}, relaxed and pure controls yield the same value.
  Consequently, one may restrict the attention to   pure controls among  $\U$ and 
 the  standard numerical techniques can   compute $\u^*_\delta \in \U$, for any $\delta > 0.$

\subsection{Convergence}
We now move towards one of the main results of this paper, where we establish that the solution $\u^*_\delta$ of the smooth problem \eqref{Eqn_final_problem_compact_relaxed},  will  be $\epsilon$-optimal for the original problem \eqref{Eqn_combined_problem}-\eqref{Eqn_x_pure_traj}---and then any numerical solution of the smooth control problem  will be an  approximate solution for \eqref{Eqn_combined_problem}-\eqref{Eqn_x_pure_traj}, where the required level of accuracy can be derived by choosing appropriate $\delta$. 

The proof relies on the parametric continuity of the value function of appropriate sequence of  `control problems'   using the Maximum Theorem (see \cite[Chapter VI]{berge1877topological}). 
This theorem holds  under sufficiently general assumptions---it requires that the objective function is jointly continuous in the variable and the parameter,  and  that the domains,   one for each parameter, are sufficiently `well behaved'. 
To apply the theorem and establish the desired convergence, we first identify the objective function, parameter, and domains connecting the relaxed problems in \eqref{Eqn_combined_problem_relaxed} and \eqref{Eqn_final_problem_compact_relaxed}. We then relate their counterparts under pure controls $\U$, ultimately obtaining an approximate solution to \eqref{Eqn_combined_problem}-\eqref{Eqn_x_pure_traj}. 

We begin with the description of a  parametrized objective function that incorporates  the above   mentioned connection---consider the objective  function $\Gamma : \Sscr \times [0, \infty) \to \mathbb{R}$,    constructed using function $W$ of \eqref{eqn_obj_fun_with_y} and the parameter $\delta \ge 0$ as below: 
\begin{eqnarray}
&&\hspace{-6mm}\Gamma^* (\delta) = \sup_{\bsig \in \Sscr}  \Gamma  ( \bsig, \delta), \mbox{ where, }  \label{Eqn_Gamma_star}\\
    &&\hspace{-6mm}\Gamma  ( \bsig, \delta) :=   W(\x_\bsig, \y_\bsig^\delta; \bsig) = \int_{0}^{t_1} L_r(t, x_\bsig(t), \sigma (t))dt  + \Psi(x_\bsig(t_1)) - y_\bsig^\delta (t_1), \mbox{ with,  } \label{Eqn_Gamma}\\
    &&\hspace{-4mm}\x_\bsig  \mbox{  is the solution of the ODE \eqref{Eqn_relaxed_dynamics_refined},  and  } \nonumber\\
    &&\hspace{-4mm}\y^\delta_\bsig  = 
    \left \{ \begin{array}{llll}
    \mbox{  the solution of the  ODE \eqref{eq:ydelta} }      & \mbox{ when } \delta > 0,   \nonumber\\
  \mbox{ is the max-trajectory \eqref{Eqn_y_defn} }   & \mbox{ when } \delta = 0.    & 
    \end{array}
    \right . \nonumber
\end{eqnarray}
In the above,   any relaxed control $\bsig \in \Sscr$ forms a decision variable and the parameter $\delta $ links the required optimization problems,   because of the following reasons:
\begin{itemize}
    \item    the optimization problem  $ \sup_{\bsig\in \Sscr}\Gamma  ( \bsig, \delta) $, for any $\delta > 0$, exactly equals the smooth-problem  \eqref{Eqn_final_problem_compact_relaxed}, in other words,   $\vr^\delta =  \Gamma^*  (   \delta) $;  
    \item 
    for $\delta=0$, $\sup_{\bsig\in\Sscr}\Gamma(\bsig,\delta)$ is the original relaxed problem in \eqref{Eqn_combined_problem_relaxed}, so $\Gamma^*(0)$ equals its value.
\end{itemize}
Thus, the idea behind establishing the desired convergence (and then $\epsilon$-optimality) is to show that the  value function $\Gamma^*(\delta)$ is continuous wrt $\delta$. 
More precisely, we only require the continuity at $\delta = 0$, i.e., 
  the convergence of the relaxed values, 
\(
\Gamma^*(\delta_n) \to \Gamma^*(0), 
\)
for  any given sequence $\delta_n \stackrel{n \to \infty}{ \longrightarrow} 0$; let $\Theta := \{\delta_n\}_{n\ge 1}$.
To this end, 
we  apply  the  Maximum Theorem to  the parametrized function $\Gamma$ defined in \eqref{Eqn_Gamma}, after  restricting its joint domain to  $\Sscr \times \Theta$, by   proving  the following major steps:

\begin{enumerate}[label=(\textbf{\Roman*}), ref=(\textbf{\Roman*})]
\setcounter{enumi}{0}
\item  let $\Sscr_\delta$ represent the domain of optimization, when the parameter is $\delta$, then one needs to prove that $\delta \mapsto \Sscr_\delta$ is compact and continuous correspondence.  \label{step_I}
\item  one needs to prove that $\Gamma$ is jointly continuous at  $(\bsig, \delta) \in \Sscr\times \Theta$ in some appropriate topology. 
\label{step_II}
\end{enumerate} 

First, observe that the domain of optimization is the same  for any  $\delta$, i.e.,   
    $\Sscr_\delta = \Sscr$
    and thus the Step  \ref{step_I} is immediate by   
 Lemma \ref{lemma_space_compact}, after  empowering $\Sscr$ with r-weak  topology.  
 
 Towards proving the Step  \ref{step_II},   it suffices to  establish sequential continuity, as $\Sscr$ is metrizable by Lemma \ref{lemma_space_compact}.
By Theorems \ref{thm:relaxed_controls_refined}-\ref{Thm_existence_relx_cont}, we have already proved that 
the  trajectory $\x_\bsig$, the  max trajectory $\y_\bsig = \y^{0}_\bsig$ and the running cost $L_r$ are sequentially continuous wrt the relaxed control $\bsig$. Further  $\x_\bsig$ and $L_r$  do not depend upon $\delta$.  
\newcommand{\II}{{\mathbb I}}
We next \textit{establish the joint continuity of the max-trajectory $\y_\bsig^\delta$ wrt   $(\bsig, \delta) \in \Sscr \times \Theta$}:
\begin{thm}[{\bf $(\bsig,\delta)$-continuity}] 
\label{Thm_ydel_convg} \textit{Assume \ref{assum_a0}-\ref{assum_a3}. 
 Then  $\{\y^0_\bsig: \bsig \in \Sscr\}$, the  max-trajectories of \eqref{Eqn_y_defn},  and $\{(\x_\bsig, \y^\delta_\bsig):  \bsig\in \Sscr, \delta > 0\}$,  the solutions   of  \eqref{Eqn_relaxed_dynamics_refined}-\eqref{eq:ydelta}   satisfy
the following continuity results   in uniform topology. \\
(i)  Fix $\delta> 0$.  
Then 
$
\| (\x_{\bsig_n}, \y_{\bsig_n}^\delta) - (\x_{\bsig}, \y_{\bsig}^\delta) \|_\infty \to 0,$ for any sequence,  $\bsig_n \rwc \bsig$.\\
(ii) Fix $\bsig \in \Sscr$ and  $\delta > 0$. Then 
$\|\y^{\delta}_\bsig -  \y^0_\bsig \|_\infty \le \delta$. \\
  (iii) Further,  
$
\| (\x_{\bsig_n}, \y_{\bsig_n}^{\delta_n}) - (\x_{\bsig}, \y_{\bsig}^0) \|_\infty \to 0,$ for any    $\bsig_n \rwc \bsig$ and $\delta_n \to 0$.}
\end{thm}
\textbf{Proof:} \textbf{Part (i)} is directly given by part (i) of Lemma \ref{Lem_exist_optimal_control}. \\
\textbf{Part (ii):}
We first define the following brief notations and some definitions:
 $$L_{mx}(t) = L_{mx}(t, x(t)), \  \  \bdot{L}_{mx} (t) = \Delta_{mx} (t, x(t), u(t) ), \forall  t. 
 $$  
Define $\Dw_0  := 0,$ and then the  following terms  for each  $i \ge 1$, recursively using \eqref{eq:ydelta},  
\begin{eqnarray}
    \Up_i &:=& \inf\{ t > \Dw_{i-1}: L_{mx}(t)> y^\delta(t) - \delta, \ \bdot{L}_{mx}(t) >0 \} = \inf\{ t > \Dw_{i-1}:  \bdot{y}^\delta(t)  >0 \}, \nonumber\\
    \Dw_i &:=& \inf\{ t > \Up_i: L_{mx}(t) \le  y^\delta(t) - \delta \mbox{ or }  \bdot{L}_{mx}(t) \le  0 \}= \inf\{ t > \Up_i:  \bdot{y}^\delta(t)  =0 \} \label{Eqn_definition_ui_di}
\end{eqnarray}
(we suppress $\bsig$ while retaining $\delta$ in notations for simplicity and clarity). 
Continue the above definitions till one of them touches $t_1$ and  
these imply the following:
 \begin{eqnarray*}
     &&\bdot{y}^\delta(t)  =0  \mbox{ for any } t \in \cup_{i\ge 1} [ \Dw_{i-1}, \Up_i], \mbox{ while,  } \bdot{y}^\delta(t)  > 0,\mbox{ for any } t \in \cup_{i\ge 1} [  \Up_i, \Dw_{i}]. 
 \end{eqnarray*}
Towards proving part (ii), it   suffices to prove the following,
\begin{eqnarray}
&&y^0(t) \le     y^\delta (t) \le y^0 (t) + \delta,  \mbox{ for all } t \in T. 
\label{Eqn_Sw}
\end{eqnarray}
As the functions involved in \eqref{Eqn_definition_ui_di} are absolutely continuous,  one can have at most countably many  $\{\Dw_i, \Up_i\}_i$ (see \cite{royden2010real}).
Thus, the result is derived using
 mathematical induction on $i$: with  $\II_\ell := (\Up_\ell, \Dw_\ell) \cup [\Dw_\ell, \Up_{\ell+1}]$ for each $\ell$,  we prove \eqref{Eqn_Sw}     for all $t \in   \II_i$,  assuming the  same for  all previous intervals, i.e., for all $t \in \cup_{\ell < i} \II_\ell$. Detailed proof is in  Appendix \ref{app_thm_y_ctny}.

\textbf{Part (iii):} For the given pair of sequences, one can split using   triangle inequality  as below and then using the upper bound of part (ii), which is uniform across $\bsig \in \Sscr$:
\begin{eqnarray*}
\|(\x_{\bsig_n}, \y^{\delta_n}_{\bsig_n}) - (\x_{\bsig}, \y^0_{\bsig})\|_\infty &\le&  \|(\x_{\bsig_n},\y^{\delta_n}_{\bsig_n})- (\x_{\bsig_n},\y^{0}_{\bsig_n}
    )\|_\infty 
    + \|(\x_{\bsig_n},\y^{0}_{\bsig_n}) - (\x_{\bsig}, \y^0_{\bsig})\|_\infty \\
   &\le& \delta_n  + \|(\x_{\bsig_n}, \y^{0}_{\bsig_n}) - (\x_{\bsig}, \y^0_{\bsig})\|_\infty. 
 \end{eqnarray*}
Next, the above converges to $0$ as $n \to \infty$ using Theorem \ref{thm:relaxed_controls_refined}.(b) and Theorem \ref{Thm_existence_relx_cont}.(a).(ii). \eop

Thus, we have established the  continuity arguments of various components that build  the  objective function $\Gamma$. 
We finally establish the major result of this paper:
\begin{thm} [{\bf Approximation and  convergence}] 
\label{Thm_epsilon_optimality}\textit{Assume \ref{assum_a0}-\ref{assum_a5} and  
consider    $\delta_n \to 0.$   
Then with $\u_n^*$ representing the solution of Lemma \ref{Lem_exist_optimal_control},  for each  $n$, we have:
$$
\lim_{n \to \infty}
\left |  \Gamma (\u^*_n, 0)  - \Gamma^* (0) \right | = 0, \mbox{ where  }
\Gamma (\u^*_n, 0) = J(x_0; \u_n^*), 
$$%
is the objective function of  \eqref{Eqn_combined_problem}    and $\Gamma^* (0)$ is the corresponding value. }
\end{thm}
\textbf{Proof:} 
Consider  any arbitrary sequence $\{\bsig_n\}_{n \geq 1}  \subset \Sscr$ such that $\bsig_n \rwc \bsig \in \Sscr$.  
Then,
\begin{eqnarray*}
 \Gamma(\bsig_n, \delta_n) &=&   \int_{0}^{t_1} L_r(t, x_{\bsig_n} (t), \sigma_n (t))dt  - y_{\bsig_n}^{\delta_n} (t_1) + \Psi(x_{\bsig_n}(t_1)) \\
  &\longrightarrow&   \int_{0}^{t_1} L_r(t, x_\bsig (t), \sigma (t))dt  - y_{\bsig}^{0} (t_1) + \Psi(x_{\bsig}(t_1))=  \Gamma(\bsig,  0),
\end{eqnarray*}
establishing  sequential continuity of $\Gamma$ on $\Sscr\times \Theta$, with $\Theta = \{\delta_n\}$, 
because of the following:\\
$\bullet$ integral terms converge by Theorem \ref{Thm_existence_relx_cont}.(a).(i), observe initial condition $x_0$ is fixed;\\
$\bullet$ second and third   terms converge 
     by Theorem \ref{Thm_ydel_convg}.(iii) and continuity of $\Psi$ (by \ref{assum_a2}).

Thus,  the Maximum Theorem (see \cite[Chapter VI]{berge1877topological}) is applicable  (see also Step \ref{step_I}) and hence we have:  $\Gamma^*(\delta_n) \to \Gamma^*(0).$
Also, by Lemma \ref{Lem_exist_optimal_control},   $\Gamma^*(\delta_n) = \Gamma( \u_n^*, \delta_n)$, for any $n$.    Thus,  
from \eqref{Eqn_Gamma} and by uniform (across   $\bsig$ or $\u$) upper bound of  Theorem \ref{Thm_ydel_convg}.(ii), we have, 
 \begin{eqnarray*}
\left |  \Gamma (\u_n^*,  0)  - \Gamma^* (0) \right | &\le&   \left |  \Gamma (\u_n^*,  0) - \Gamma(\u_n^*, \delta_n)  \right | + \left |  \Gamma (\u_n^*,  \delta_n) - \Gamma^*(0)  \right | \\
&=& \ |y^{\delta_n}_{\u_n^*} (t_1) - y^{0}_{\u_n^*} (t_1) |  + \left |  \Gamma^* (  \delta_n) - \Gamma^*(0)  \right | \le  \delta_n  + \left |  \Gamma^* (  \delta_n) - \Gamma^*(0)  \right | \to  0.  
 \end{eqnarray*}
    \eop

\subsection{Numerical solutions}
\label{sec_numerical_results}
  Thus, by Theorem \ref{Thm_epsilon_optimality}, the solution $\u^*_\delta$ of \eqref{Eqn_final_problem_compact_relaxed}, whose existence in $\U$ is guaranteed by Lemma \ref{Lem_exist_optimal_control}, is $\epsilon$-optimal for sufficiently small $\delta$. Fix such a $\delta>0$, after choosing an appropriate $\phi$ in \eqref{eq:ydelta}, and additionally assume that:
\begin{enumerate} [label=\textbf{A.\arabic*}, ref=\textbf{A.\arabic*}]
\setcounter{enumi}{5}
    \item \textit{There exists $K_\Psi>0$ such that  \begin{eqnarray*}
    |\Psi(x) - \Psi(x')|&<& K_\Psi| x- x'|, \mbox{ for all } x, x' \in \mathbb{R}^n, \mbox{ and }\\
      |\Psi(x)| &\leq&  K_\Psi (1 + |x|), \mbox{ for all } x \in \mathbb{R}^n. 
    \end{eqnarray*}}
    \label{assum_a6}
\end{enumerate}

It is well known (see \cite[Lemma 2.2]{zhou1993verification}) that the value function
$(t,x_0,y_0)\mapsto v(t,x_0,y_0)
$
of the smooth problem \eqref{Eqn_final_problem_compact_relaxed} is the unique viscosity solution of the HJB equation
\begin{eqnarray}
\label{Eqn_hjb_general}
&&
\frac{\partial v}{\partial t}
+\sup_{u\in U}
H\left(t,x,y,u,\frac{\partial v}{\partial x},\frac{\partial v}{\partial y}\right)=0,
\quad \mbox{ satisfying, }
v(t_1,x,y)=-y+\Psi(x),
\end{eqnarray}
where
\(
H(t,x,y,u,p,q)
:=pf(t,x,u)
+q\bigl(\Delta_{mx}(t,x,u)\bigr)^+
\psi_\delta\bigl(L_{mx}(t,x)-y\bigr)
+L_r(t,x,u), 
\)
  with 
\begin{eqnarray*}
    \Delta_{mx}(t,x,u)
&=&\frac{\partial L_{mx}(t,x)}{\partial t}
+\frac{\partial L_{mx}(t,x)}{\partial x}f(t,x,u), \mbox{ and }\\
\psi_\delta(d)
&=&\phi\left(1+\frac{d}{\delta}\right)\mathds{1}_{\{d\in[-\delta,0]\}}
+\mathds{1}_{\{d>0\}}.
\end{eqnarray*}
Recall that we are specifically interested in the value $v(0,x_0,x_0)$. Moreover, if $f$, $L_r$, and $\Psi$ are continuously differentiable, the corresponding verification result holds (see \cite[Theorem 3.2]{zhou1993verification}). Thus, solving the HJB equation \eqref{Eqn_hjb_general} yields a solution of \eqref{Eqn_final_problem_compact_relaxed} in $\U$, which is $\epsilon$-optimal for \eqref{Eqn_combined_problem}-\eqref{Eqn_x_pure_traj} for sufficiently small $\delta$.

One may also employ the \textit{Pontryagin maximum principle} (see, e.g., \cite{fleming2006controlled}) to compute the above \textit{candidate}. To this end, we solve the adjoint (co-state) equations
\[
\frac{d\lambda_x}{dt}=-\frac{\partial H}{\partial x},
\qquad
\frac{d\lambda_y}{dt}=-\frac{\partial H}{\partial y},
\]
with terminal conditions
\(
\lambda_x(t_1)=-\frac{\partial\Psi(x(t_1))}{\partial x},\
\lambda_y(t_1)=-1.
\)
The optimal control is then obtained by maximizing the Hamiltonian pointwise:
\[
u_\delta^*(t)\in\arg\max_{u\in U}
H\bigl(t,x(t),y^\delta(t),u,\lambda_x(t),\lambda_y(t)\bigr) \mbox{ for each } t \in T,\]
using the computed state and co-state trajectories $\x$, $\y^\delta$, ${\bm \lambda}_x$, and ${\bm \lambda}_y$.

We now apply the above smooth approximation framework to study the queuing control problem introduced in Subsection \ref{subsec_queue_example}.

\section{Queuing control: minimizing peak occupancy}
\label{sec_queu_control} 
We study the problem \eqref{Eqn_queue_dynamics}--\eqref{Eqn_Queue_problem_with_peak}, where the service rate is of the form
\(\mu(s, x(s); u(s))=(\alpha(s)+x(s))u(s)
\) $\forall s$, for a control function~$\u$. Thus, the service rate at each time (say at time $s$) is proportional to both the arrival rate ${\alpha(s)}$ and the current occupancy $x(s)$.
We consider the following assumptions: 
\begin{enumerate}[label=\textbf{C.\arabic*}, ref=\textbf{C.\arabic*}]
\item \textit{The   function ${\bm \alpha}$ is  Lipschitz continuous  and is  strictly positive  on interval $T$}. \label{assum_c1}
\item \textit{The control space  $U = [0, \bar u] $ for some $\bar u < \infty$}. \label{assum_c2}
\item \textit{We consider linear and quadratic  cost functions\footnote{One can   consider  other  cost functions  that satisfy \ref{assum_a5}.}}:  $
g(x) = x,     \mbox{ and }  h(\mu)  =   (\mu- \mu_{id})^2.  
$ \label{assum_c3}
\end{enumerate}
In \ref{assum_c3}, 
  $\mu_{id} >0$ represents the \textit{ideal operating service rate} of the system  at which the server incurs minimum cost;  basically the system is maximum efficient to serve at rate~$\mu_{id}$.
The problem \eqref{Eqn_queue_dynamics}--\eqref{Eqn_Queue_problem_with_peak} is again of the form \eqref{Eqn_combined_problem}--\eqref{Eqn_x_pure_traj} and can therefore be analyzed using the solution techniques of Subsection \ref{sec_numerical_results}. In particular, we take
 $\phi(s)= e^{1-s^{-2}}$  in  \eqref{eq:ydelta}.

\subsubsection*{Smooth approximation}
The smooth problem \eqref{Eqn_final_problem_compact_relaxed} for this setting is given by
\begin{align}
\label{Eqn_queue_objective_fun}
\sup_{u}\quad &
-\int_0^{t_1}\left[\rho x(t)+h\bigl(\mu(t)\bigr)\right]dt
-g_{t_1}\bigl(x(t_1),y^\delta(t_1)\bigr),
\quad
g_{t_1}(x,y^\delta):=\eta x+\beta y^\delta,
\\
&\text{subject to}\quad
\bdot{x}(t)=\alpha(t)-\mu(t, x(t); u(t) ),
\qquad
\bdot{y}^\delta(t)=\bigl(\bdot{x}(t)\bigr)^+
\psi_\delta\bigl(x(t)-y^\delta(t)\bigr),  
\label{Eqn_queuing_smooth_variant}
\end{align}
where $\mu(t, x; u) = (\alpha(t)+x)u$ and 
$
y^\delta(t)$ is the  smooth approximation of the peak congestion level up to time $t$  given by, $\sup_{s\in[0,t]}x(s)
$.

The assumptions \ref{assum_a0}–\ref{assum_a6}, required for Lemma \ref{Lem_exist_optimal_control}, Theorem \ref{Thm_epsilon_optimality}, and for ensuring that the HJB equation \eqref{Eqn_hjb_general} satisfies the conditions of the verification theorem, are automatically fulfilled under assumptions \ref{assum_c1}–\ref{assum_c3}.
 In view of this, we obtain the approximate numerical solution   for 
\eqref{Eqn_queue_dynamics}-\eqref{Eqn_Queue_problem_with_peak} by solving the following HJB PDE (see \eqref{Eqn_hjb_general}):  
\begin{eqnarray} \label{Eqn_queue_hamilton}
 &&\hspace{4mm}\frac{\partial v}{\partial t}  + \sup_{u }  H\left(t, x, y^\delta, u, \frac{\partial v}{\partial x} ,\frac{\partial v}{\partial y^\delta} \right) = 0, \mbox{ with }  v(t_1, x, y^\delta) = - \beta y^\delta - \eta x,     \\ 
&& \hspace{3mm}\mbox{ where }   H(t, x, y^\delta, u, p, q) := p\bdot{x} + q \bdot{y}^\delta - \rho x  -  h(\mu(t; u,x)). \nonumber
  \end{eqnarray}

\subsection{Numerical results}
We obtain the numerical results using the Pontryagin maximum principle. The co-state equations related to Hamiltonian $H$ in \eqref{Eqn_queue_hamilton}  are given by:
\begin{eqnarray*}
\frac{d \lambda_{x}}{dt} &=& - \frac{\partial H}{\partial x} =  \lambda_x u  -    {\lambda_y^\delta} \left ( \frac{ \partial \psi_\delta ( x- y^\delta) }{\partial x} (\alpha - \mu )^+  -  \psi_\delta ( x- y^\delta) u \mathds{1}_{\{\alpha \ge \mu \}}  \right )  +  \rho  + 2  u   (\mu - \mu_{id}) \\ 
    \frac{d \lambda_y^\delta}{d t} &=&  - \frac{\partial H}{\partial y} = -  {\lambda^\delta_{y}} \left (\frac{ \partial \psi_\delta ( x- y^\delta) }{\partial y^\delta} (\alpha - \mu)^+\right ), \mbox{ with the boundary conditions}\\
    \lambda_x(t_1) &=&  \frac{\partial g_{t_1}(x,y^\delta)}{\partial x} = -  \eta , \  \
\lambda_y^\delta(t_1)  = \frac{\partial g_{t_1}(x,y^\delta)}{\partial y^\delta} = - \beta. 
\end{eqnarray*}
One can solve the above equations along with \eqref{Eqn_queuing_smooth_variant}  to obtain state and costate trajectories, $( \x^*,  \y^{\delta*}, {\bm \lambda}^*_x, {\bm \lambda}^{\delta*}_y)$.
And then to compute the control policy, first 
observe that the Hamiltonian \eqref{Eqn_queue_hamilton} depends primarily on the sign of $(\alpha-\mu)$. Thus, for  each fixed $t$, we separately optimize over
$$
u\in\left[0,\frac{\alpha(t)}{\alpha(t)+x(t)}\right]
\quad\text{and}\quad
u\in\left[\frac{\alpha(t)}{\alpha(t)+x(t)},\bar u\right],
$$
(more precisely over non-empty sub-intervals given above), and then combine the resulting optimizers. For brevity, let 
$\lambda_x = \lambda_x^*(t)$,  $\lambda^\delta_y =\lambda_y^{\delta*}(t) $, $x=x^*(t) $,  $y^\delta=y^{\delta*}(t) $, $\alpha = \alpha(t)$.
The two sub-optimization problems, corresponding sub-optimizers $u_1^*$ and $u_2^*$ and the overall optimizer   at time $t$ are given by (using \eqref{Eqn_queue_objective_fun}-\eqref{Eqn_queue_hamilton} and \ref{assum_c3}):

\noindent $\bullet$ For the first interval, we optimize the following   over    $u \in [0,  \min\{{\bar u},\nicefrac{\alpha }{(\alpha+x)}\} ]$,
     \begin{eqnarray*}
    && O_1( x, y^\delta, u, \lambda_{x} , \lambda_{y}^\delta) := \bigg (\lambda_x + \lambda_y^\delta \psi_{\delta} (x-y)\bigg) (\alpha  - \mu ) - \rho x  -  h(\mu),   \\
 && \mbox{The first sub-optimizer,   } u_1^* := \mbox{\small$\max\left \{0, \min\left\{{\bar{u}} , \left(\frac{\alpha}{\alpha +x}\right),   \frac{2 \mu_{id} \mathds{1}_{\{u>0\}}- \lambda_x - \lambda_y^\delta   \psi_\delta ( x -y^\delta ) ) }{2  (\alpha +x)  \mathds{1}_{\{u>0\}}} \right \}\right\}$}. 
    \end{eqnarray*}
    
\noindent $\bullet$ 
The second sub-problem maximizes the following over \(u\in\left[\nicefrac{\alpha}{(\alpha+x)},\bar u\right]\), when non-empty:
\begin{eqnarray*}
    &&O_2( x, y^\delta, u, \lambda_{x} , \lambda_{y}^\delta) =  \lambda_x (\alpha- \mu)  - \rho x  -  h(\mu).  \\
     && \hspace{-15mm}\mbox{The second sub-optimizer, } u_2^*  :=  \mbox{\small$\max \left \{ \left(\frac{\alpha}{\alpha +x}\right), \  \min \left \{  \bar{u},  \frac{2  \mu_{id} \mathds{1}_{\{u>0\}}- \lambda_x  }{2 (\alpha +x)  \mathds{1}_{\{u>0\}}} \right \}\right \}$}. 
    \end{eqnarray*} 
  Thus, bringing back the dependency on $t$,  for any $t$,   
\begin{eqnarray}
\hspace{3mm}u^*(t) =
\begin{cases}
u_1^*(t), & \text{if } \frac{\alpha(t)}{(\alpha(t) + x^*(t))}> \bar u , \\
\arg\max\limits_{u \in \{u_1^*(t), u_2^*(t)\}}
H\big(t, x^*(t), y^{\delta*}(t), u, \lambda_x^*(t), \lambda_y^{\delta*}(t)\big), & \text{otherwise}.
\end{cases}
\end{eqnarray}



{\ignore{
Similarly, this framework can be applied to electricity load balancing problems. For instance, consider a power grid, where $d(t)$ represents the electricity demand at time $t$, and $g(t)$ represents the power generation rate controlled by the system operator. The imbalance between supply and demand, denoted by $x(t)$, can be modeled by the following ODE:
\begin{equation*}
    \dot{x}(t) =  g(t)- d(t).
\end{equation*}
The goal is to minimize the cost associated with the imbalance and the cost of power generation. A typical cost function might include the integral of the squared imbalance and the squared generation rate:
\begin{equation*}
    \int_0^{t_1} \left( x(t) + \gamma g(t)^2 \right) dt + x(T),
\end{equation*}
\noindent where $\gamma$ is a weighting factor. Alternatively, to account for peak load constraints, a combined cost function similar to the queuing system can be used:
\begin{equation*}
    \sup_{t \in [0, T]} x(t) + \gamma \int_0^{t_1} g(t)^2 dt + x(T).
\end{equation*}
\noindent This formulation ensures that the system not only minimizes the cumulative imbalance and generation costs but also avoids exceeding the grid's capacity limits.}}


\subsubsection*{Effect of varying $\beta$ and $\rho$} The objective function   \eqref{Eqn_Queue_problem_with_peak},   is a weighted combination of three distinct components: peak congestion  $y(t_1) = \sup_{t\in T} \{x(t)\}$, cumulative congestion $\int_0^{t_1} g(x(t))\,dt$, and server utilization $\int_0^{t_1} h(\mu(t))\,dt$. Typically,  in most of the existing queuing literature,  the trades-off between two objectives functions are studied, the cumulative congestion cost and the server utilization cost. Here, our formulation incorporates an additional objective,  the  peak term $y(t_1)$ and we investigate the impact of this new term. 
Specifically, we examine how the weights \(\beta\) and \(\rho\), associated with the \(L^\infty\)-term and congestion cost, respectively, affect the optimal control policies and corresponding state trajectories. The numerical results are obtained using the smooth approximate variants \eqref{Eqn_queue_objective_fun}--\eqref{Eqn_queue_hamilton}.

Figure~\ref{Fig_queue_sigma_effect} compares two representative cases: one with a larger $\beta$, placing greater emphasis on the peak term, and another with a smaller $\beta$. The two cases are chosen such that the server utilization cost, $\int_0^{t_1} h(\mu(t))dt$, is approximately the same at their respective optima, achieved by tuning the corresponding parameter sets.  With high $\beta$ and low $\rho$, the policy minimizes the peak, keeping $x(t)$ close to its initial value (blue curve). Here, congestion is less penalized, so the trajectory avoids large peaks without reducing cumulative integral cost. In contrast, when $\rho$ is large and $\beta$ is small, the policy minimizes cumulative congestion, resulting in a lower trajectory (red curve) for most of the horizon, but tolerates a sharp rise near the end---in other words, \textit{when one does not  account for peak levels in the control problem, significant congestion can occur towards the end of the horizon, as it contributes negligibly to the cumulative  cost}. 
This contrast illustrates the trade-off between the peak control and the cumulative congestion when the priorities differ, see the table also. 
\begin{figure*}[htbp]
\hspace{-0.3cm}
    \centering
    \begin{minipage}{5.5cm}
\includegraphics[trim = {0.5cm 12.3cm 0.4cm 0.2cm}, clip, width = 6cm, height = 5.15cm]{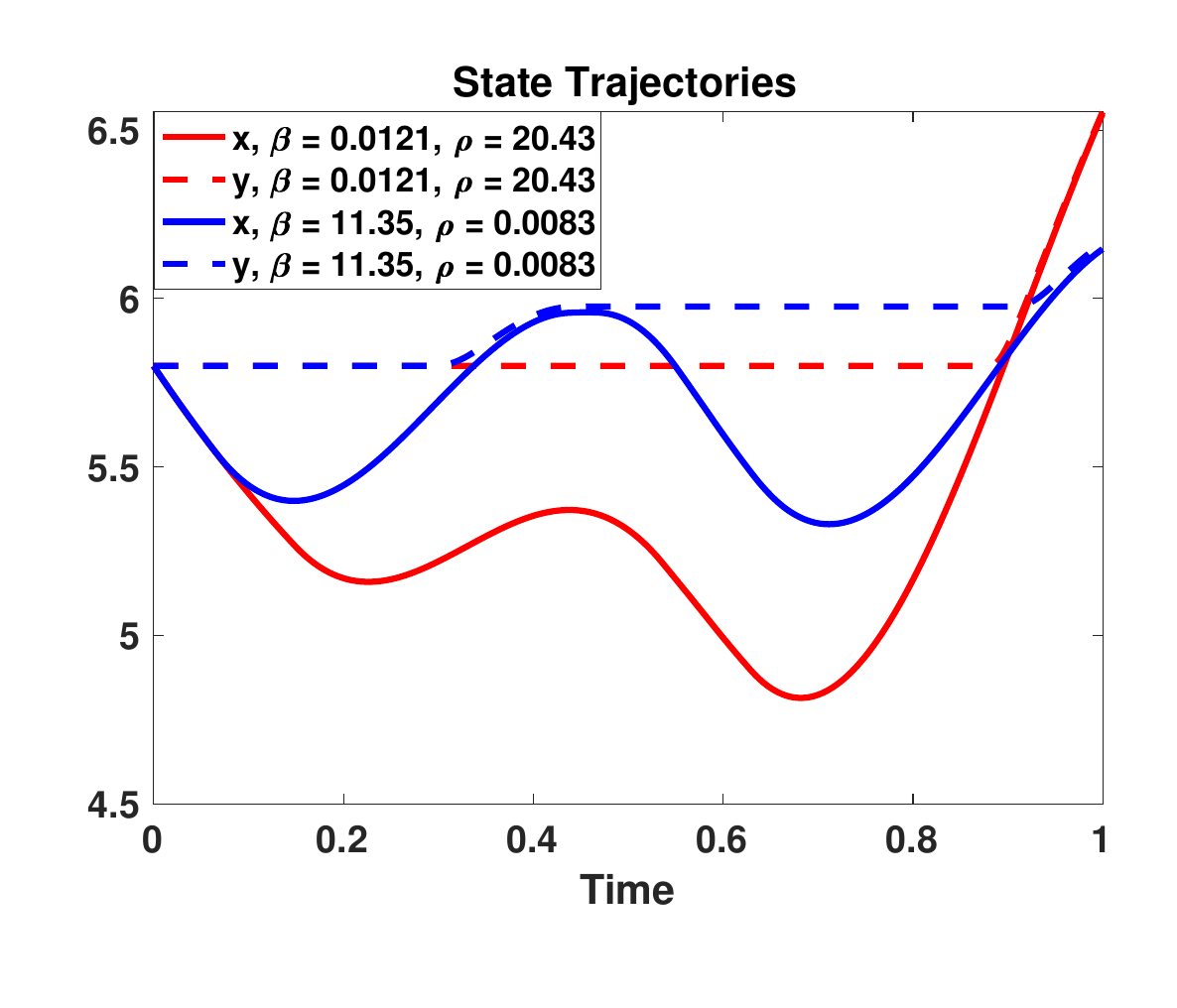}
    \end{minipage}
    \hspace{0.2cm}
        \begin{minipage}{5.5cm}
\includegraphics[trim = {0.4cm 13cm 3.6cm 0cm}, clip, width = 6cm, height = 5.20cm]{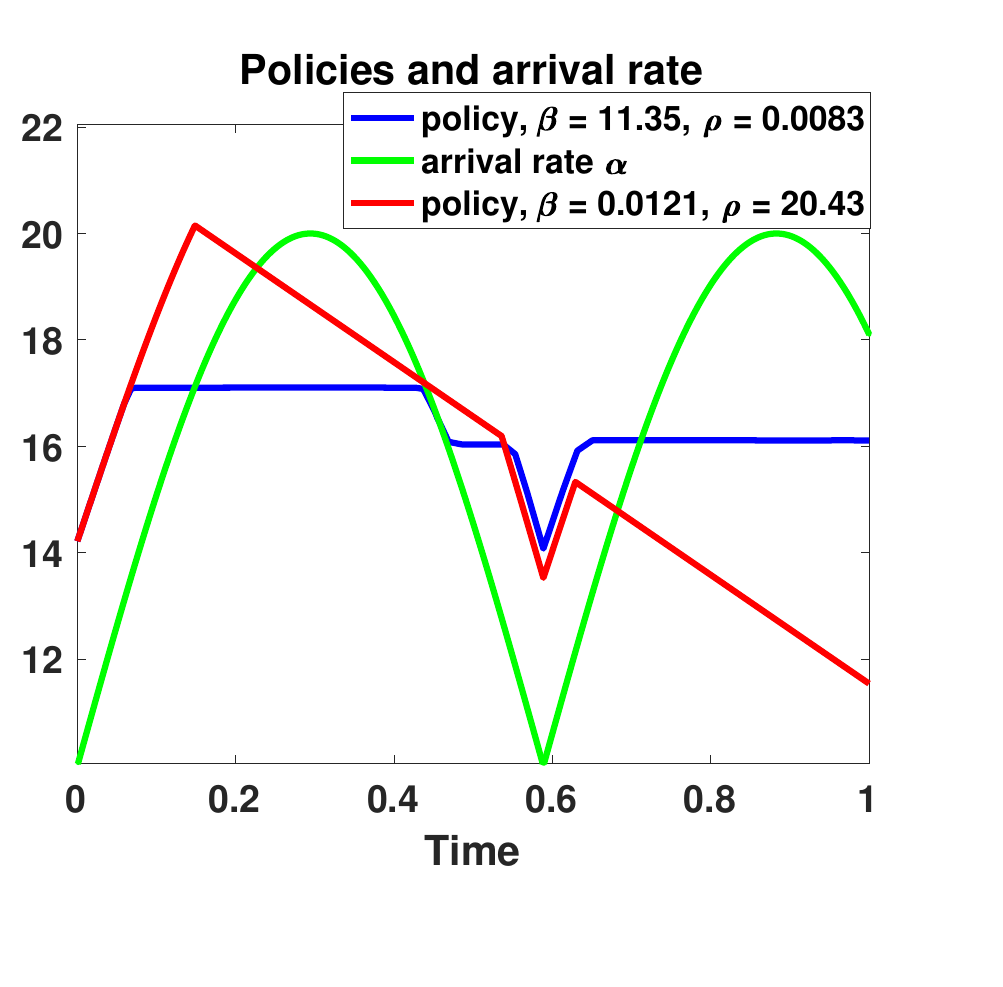}
    \end{minipage}
    \hspace{-0.1cm}
        \begin{minipage} {4.2cm}
{\small \begin{tabular}{|l|l|l|l|l|}
\hline
$\beta$ & $\rho$ &  $y^{\delta*}(t_1)$  \\ \hline
0.012       & 20.43 &  6.5536 \\ \hline
1135      & 0.0083& 6.1631  \\ \hline

\end{tabular}}
    \end{minipage}
    \caption{ State trajectories $(x^*(t), y^{\delta*}(t))$ and optimal policy $u^*(t)$, when $\delta = 0.2, t_1=1, \bar{u} = 0.9,  \eta = 1, \mu_{id} = 11.5, \int_0^{t_1} h(\mu(t))dt = 26.10, \alpha (t) = 10(1+ \cos{|\frac{3 \pi }{2} +1.7 \pi t|}). $}
    \vspace{-3mm}
    \label{Fig_queue_sigma_effect}
\end{figure*}

\subsubsection*{Pareto frontiers}  
From \eqref{Eqn_Queue_problem_with_peak} and \eqref{Eqn_queue_objective_fun}, one can observe that both peak congestion and cumulative congestion costs depend on the state trajectory. 
Nonetheless, the two costs accumulate in a very different way (over the entire time horizon),  thus one may not be able to control them simultaneously.  To study this aspect, we consider parts of two distinct  Pareto frontiers. 
We basically study  two different (and extreme) Pareto frontiers, one that trades off between the peak-congestion  and the cumulative server utilization, and the second that trades off between the cumulative congestion levels and the cumulative server utilization.  These cases highlight how different cost formulations affect the trade-offs among system objectives.

The first Pareto frontier is obtained by fixing $\rho=0$ and varying $\beta$, and is represented by the blue curve with stars in the left sub-figure of Figure~\ref{Fig_pareto_with_non_zero_mu_ideal}. It plots the cumulative server utilization cost $\int_0^{t_1}h(\mu(t)),dt$ against the peak level $y^\delta(t_1)$ at the corresponding solutions. The second frontier is obtained by fixing $\beta=0$ and varying $\rho$, and is represented by the red curve with stars in the middle sub-figure. It plots the cumulative server utilization cost $\int_0^{t_1}h(\mu(t)),dt$ against the cumulative congestion cost $\int_0^{t_1}g(x(t)),dt$.
Furthermore, at each point on the two frontiers, we tune $\beta$ and $\rho$, respectively, so that the corresponding cumulative server utilization costs are approximately equal.

We additionally plot a  curve with squares in the left sub-figure, showing the peak level $y^\delta(t_1)$ against the cumulative server utilization cost $\int_0^{t_1}h(\mu(t)),dt$, obtained by fixing $\beta=0$ and varying $\rho$, corresponding to the solutions shown in the middle sub-figure. This curve does not constitute a Pareto frontier; rather, the two curves of the sub-figure provide a comparison of the peak-level performance $y^{\delta*}(t_1)$ between solutions that explicitly prioritize peak congestion and those that do not.

Finally, in the rightmost sub-figure of Figure~\ref{Fig_pareto_with_non_zero_mu_ideal}, we plot the normalized (percentage) differences between the two blue curves and between the two red curves in the first two sub-figures. The blue curve represents the difference in peak levels attained at the respective optimizers, with and without explicit emphasis on peak congestion. Similarly, the red curve represents the difference in cumulative congestion costs between solutions that do and do not explicitly emphasize cumulative congestion.

Finally in the right most sub-figure of Figure~\ref{Fig_pareto_with_non_zero_mu_ideal}, we plot the normalized (or percentage) differences between the two  blue curves and that between two red curves  of the first two sub-figures. 
The blue curve in this sub-figure    illustrates the difference in peak levels attained at respective optimizers,  with  (blue in first sub-figure) and without (blue in second sub-figure) providing emphasis on peak-levels. Likewise the red curve in the right most sub-figure illustrates the difference in cumulative congestion cost attained with and without prominence to cumulative cost.

\begin{figure*}[htbp]
    \centering
    \begin{minipage}{5cm}
\includegraphics[trim = {3.0cm 13.cm 2.5cm 1cm}, clip, width = 5.5cm, height = 5.cm]{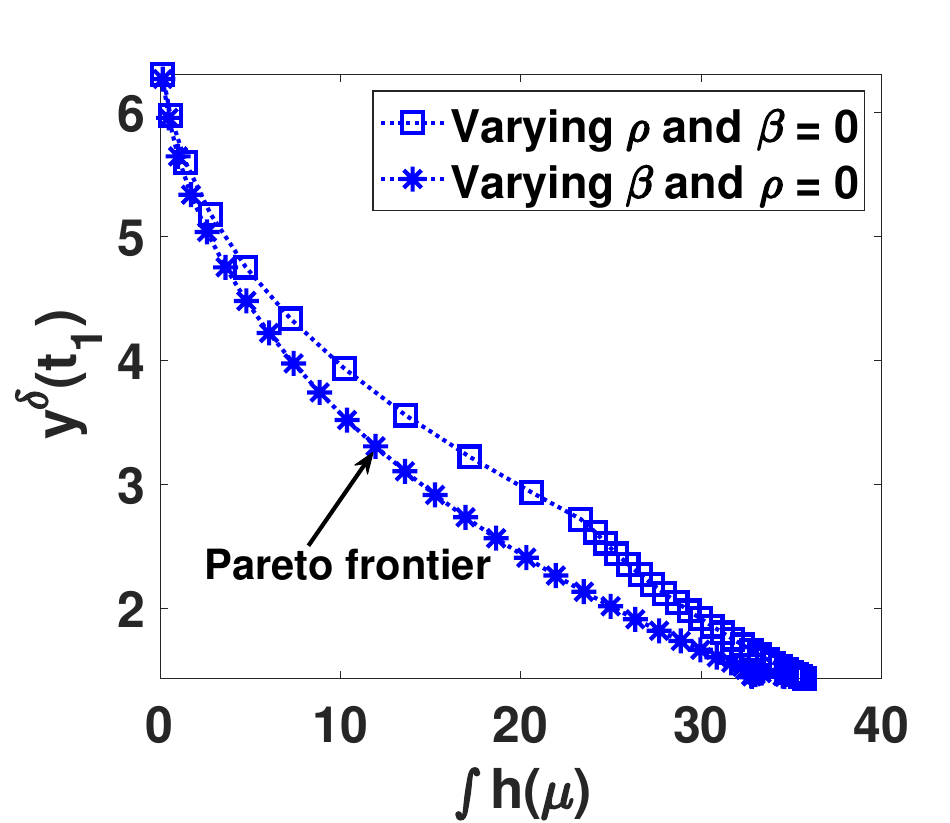}
    \end{minipage}
    \hspace{0.2cm}
        \begin{minipage}{5cm}
\includegraphics[trim = {0cm 12.7cm 4.7cm 1cm}, clip, width = 5.4cm, height = 5.04cm]{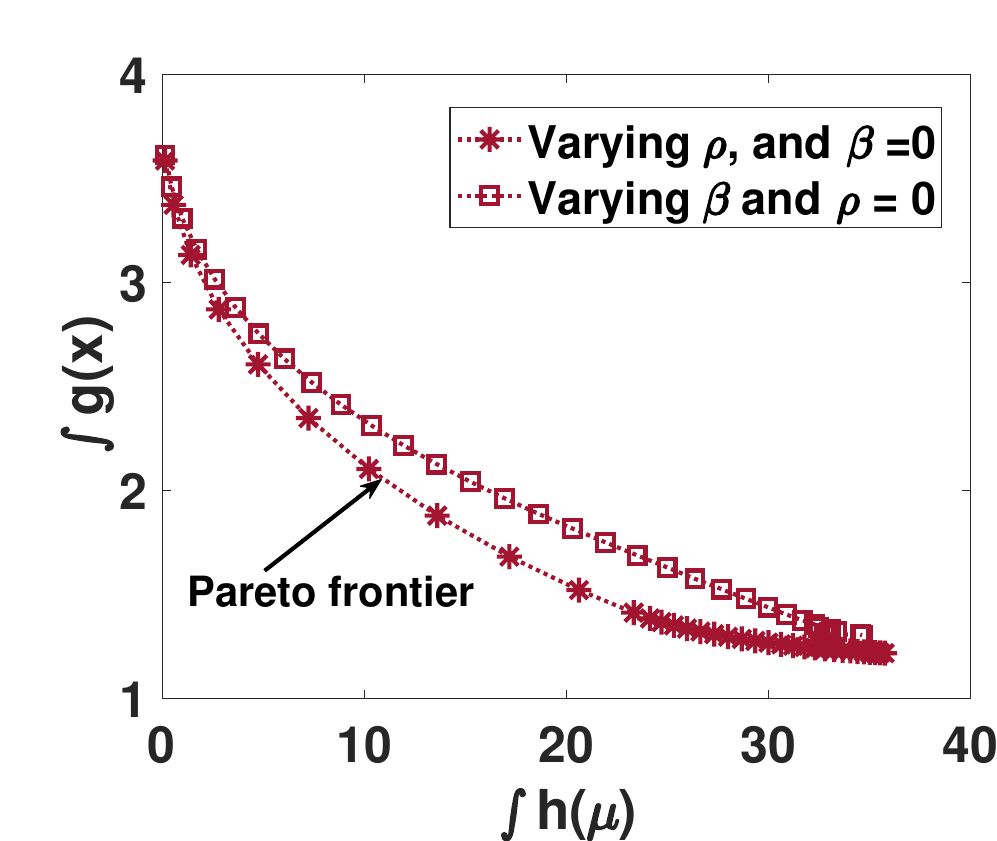}
    \end{minipage}
     \hspace{0.2cm}
        \begin{minipage}{5cm}
\includegraphics[trim = {3cm 12.5cm 3.4cm 0.4cm}, clip, width = 5.5cm, height = 5.5cm]{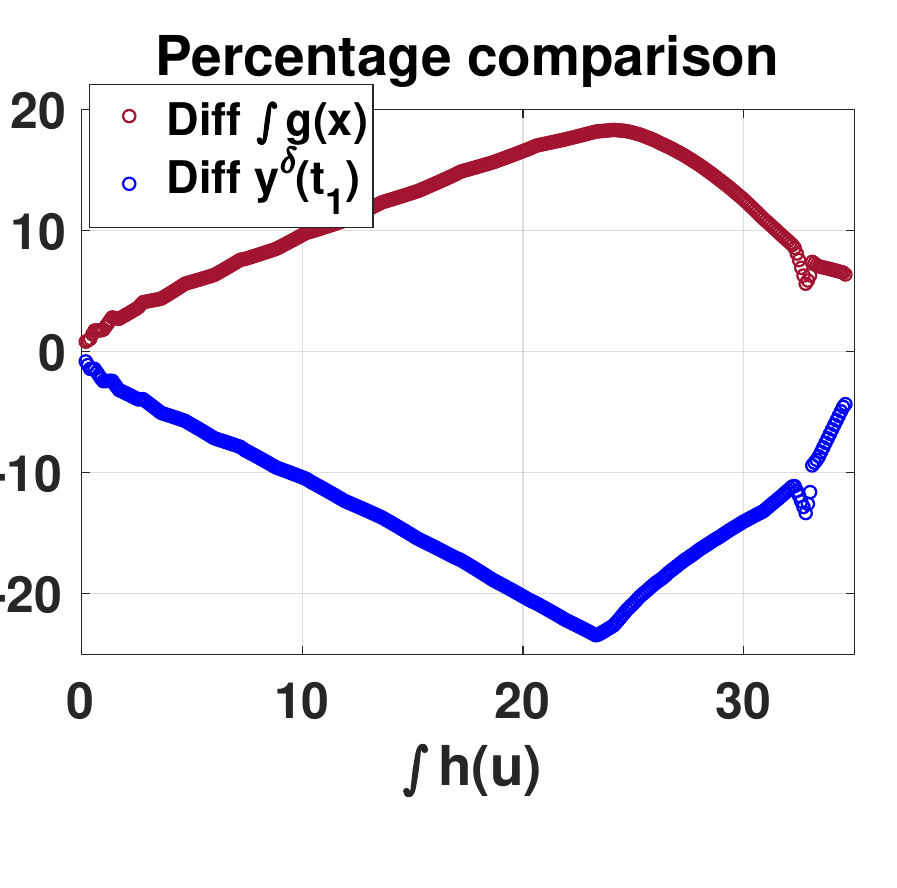}
\vspace{-10mm}
    \end{minipage}
    \caption{ Pareto frontiers:  $\delta = 0.2, \bar{u} = 0.95,   \alpha(t) = 1+|\cos\left(\frac{3 \pi}{2} + t \frac{1.7 \pi }{3000} \right)|, t_1=1, \mu_{id} = 11.5$.  
    }
\label{Fig_pareto_with_non_zero_mu_ideal}
\vspace{-3mm}
\end{figure*}

We observe that there is a significant reduction in the peak-congestion levels (even up to $27 \%$) when one includes the latter cost into the control problem; however this \textit{improvement is at the expense of the cumulative congestion cost, which degrades significantly (around $20 \%$)}, see the right sub-figure.  
The difference is minimal when the cumulative server utilization cost is either very low or very high. However, the error becomes significantly larger in the mid-range, which is typical in practical applications. 
Thus,  the design for  queuing systems must consider a three objective Pareto frontier, where 
the weight factors for  the three costs have to be chosen judiciously;  our proposed framework makes it possible to derive the corresponding $\epsilon$-optimal policy numerically.  
\subsubsection*{Convergence with $\delta$}
We next examine the convergence of the smooth approximation as the smoothing
parameter $\delta$ decreases. Figure~\ref{Fig_y_delta_converge} compares the
optimal state trajectories and corresponding policies for
$\delta=0.2,\,0.1,$ and $0.04$. 
\begin{figure*}[htbp]
\vspace{-3mm}
\hspace{-0.1cm}
    \centering
    \begin{minipage}{5.3cm}
\includegraphics[trim = {0.5cm 12.5cm 0.7cm 1.5cm}, clip, width=5.8cm, height=4.3cm]
{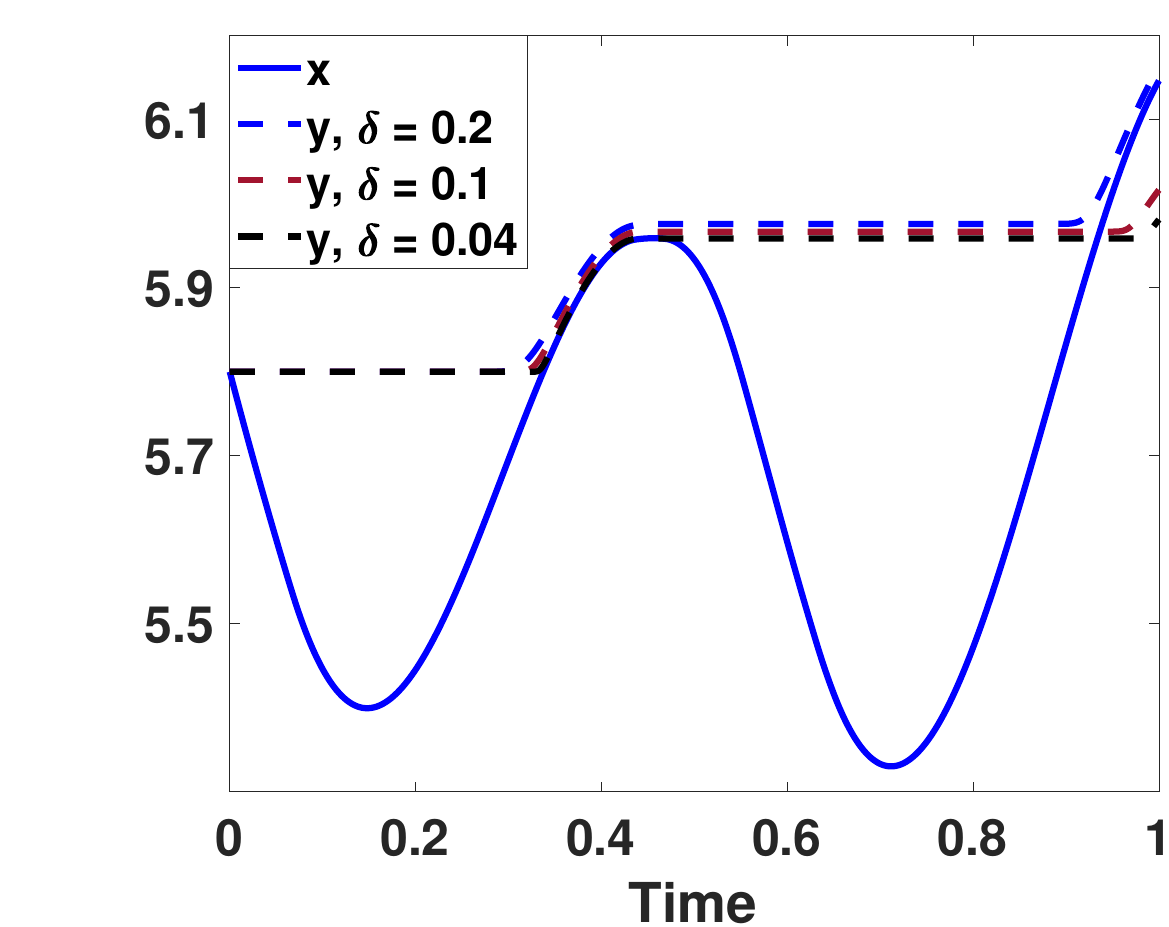}
    \end{minipage}
    \hspace{0.2cm}
        \begin{minipage}{5.3cm}
\includegraphics[trim = {0.4cm 13cm 0cm 0.8cm}, clip, width = 5.5cm, height = 4.70cm]{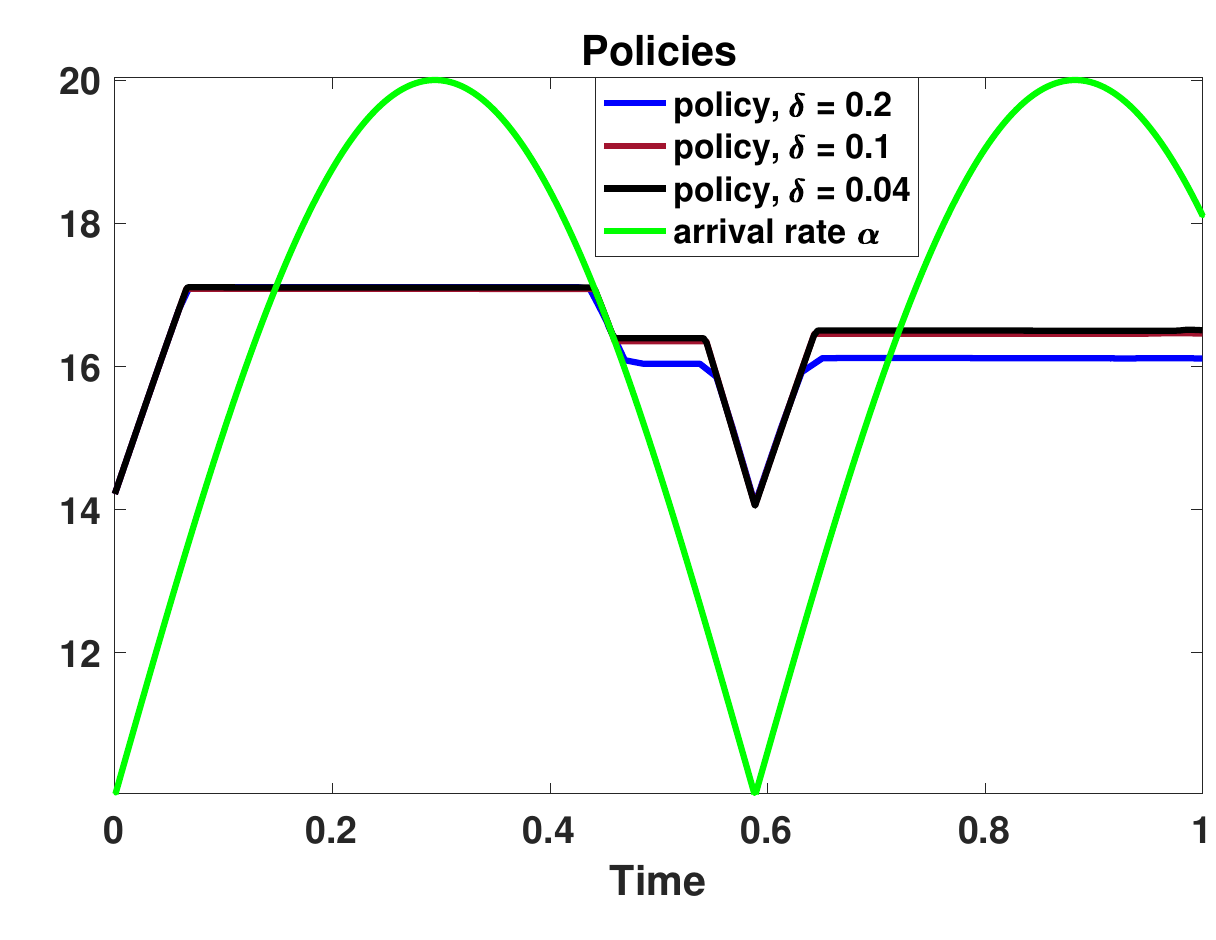}
    \end{minipage}
    \hspace{2mm}
        \begin{minipage} {3.2cm}
{\small \begin{tabular}{|l|l|l|l|l|}
\hline
 $\delta$ &  $y^{\delta*}(t_1)$  \\ \hline
 0.2 &  6.1631 \\ \hline
 0.1& 6.0166  \\ \hline
0.04& 5.9814  \\ \hline
\end{tabular}}
    \end{minipage}
    \caption{ State trajectories $(x^*(t), y^{\delta*}(t))$ and optimal policy $u^*(t)$, when $\beta = 11.35, \rho= 0.0083, t_1=1, \bar{u} = 0.9,  \eta = 1, \mu_{id} = 11.5, \alpha (t) = 10(1+ \cos{|\frac{3 \pi }{2} +1.7 \pi t|}). $}
    \vspace{-3mm}
    \label{Fig_y_delta_converge}
\end{figure*}
The left sub-figure shows that the state
trajectories corresponding to the three values of $\delta$ are already very
close, and the trajectories become progressively more aligned as $\delta$
decreases. A similar behaviour is observed for the optimal policies in the
middle sub-figure, where the policies for $\delta=0.1$ and $\delta=0.04$ are
nearly indistinguishable.

The terminal values $y^{\delta *}(t_1)$, reported in the accompanying table,
also exhibit a stabilizing behaviour as $\delta$ decreases. In particular, the
change in the terminal peak estimate becomes substantially smaller when
$\delta$ is reduced from $0.1$ to $0.04$ than when it is reduced from $0.2$ to
$0.1$. This provides numerical evidence that the solutions of the smooth
problems are approaching a limiting solution as $\delta\to0$.

Overall, the state trajectories, optimal policies, and terminal peak estimates
suggest that the effect of smoothing becomes negligible for sufficiently small
$\delta$, which is consistent with the convergence result established above.

\ignore{

\begin{figure*}[htbp]
\hspace{-0.1cm}
    \centering
    \begin{minipage}{4.5cm}
\includegraphics[trim = {4cm 13.cm 3cm 1cm}, clip, width = 5cm, height = 4.cm]{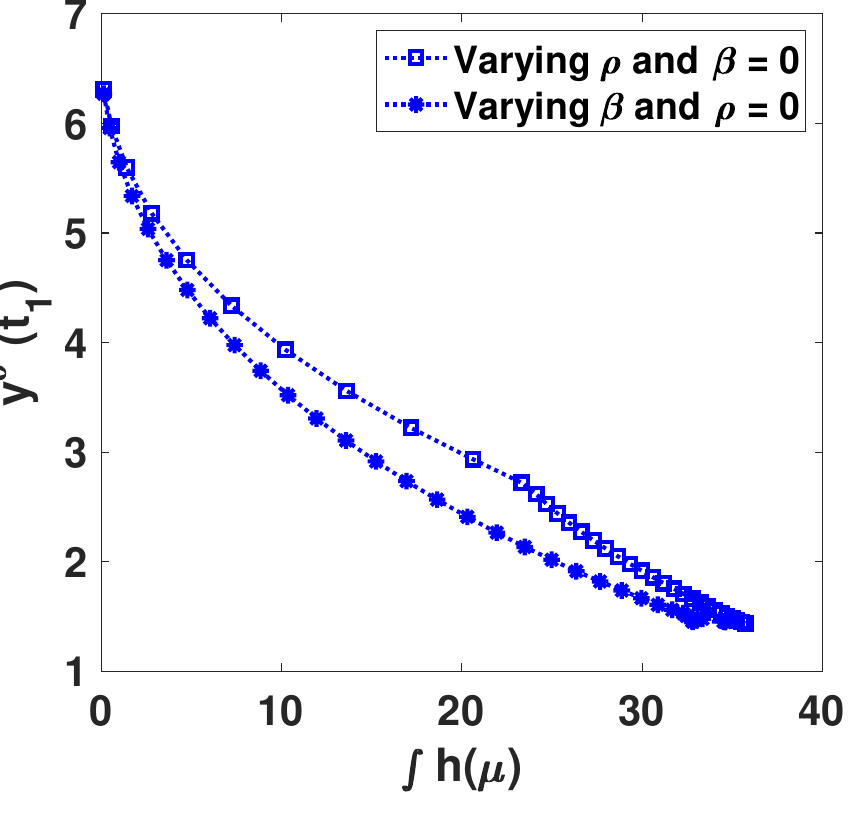}
    \end{minipage}
    \hspace{0.2cm}
        \begin{minipage}{5cm}
\includegraphics[trim = {0cm 12.7cm 4.7cm 1cm}, clip, width = 5cm, height = 4.04cm]{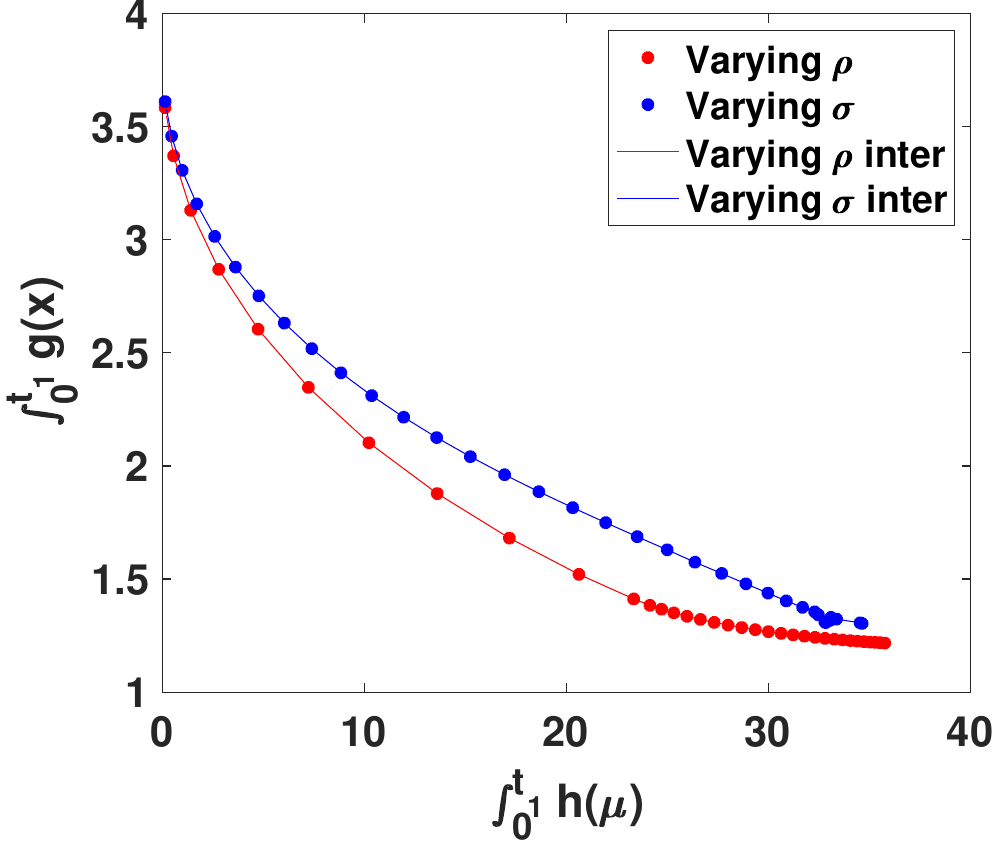}
    \end{minipage}
     \hspace{0.2cm}
        \begin{minipage}{5cm}
\includegraphics[trim = {3cm 12.5cm 3.4cm 0.4cm}, clip, width = 5cm, height = 4.04cm]{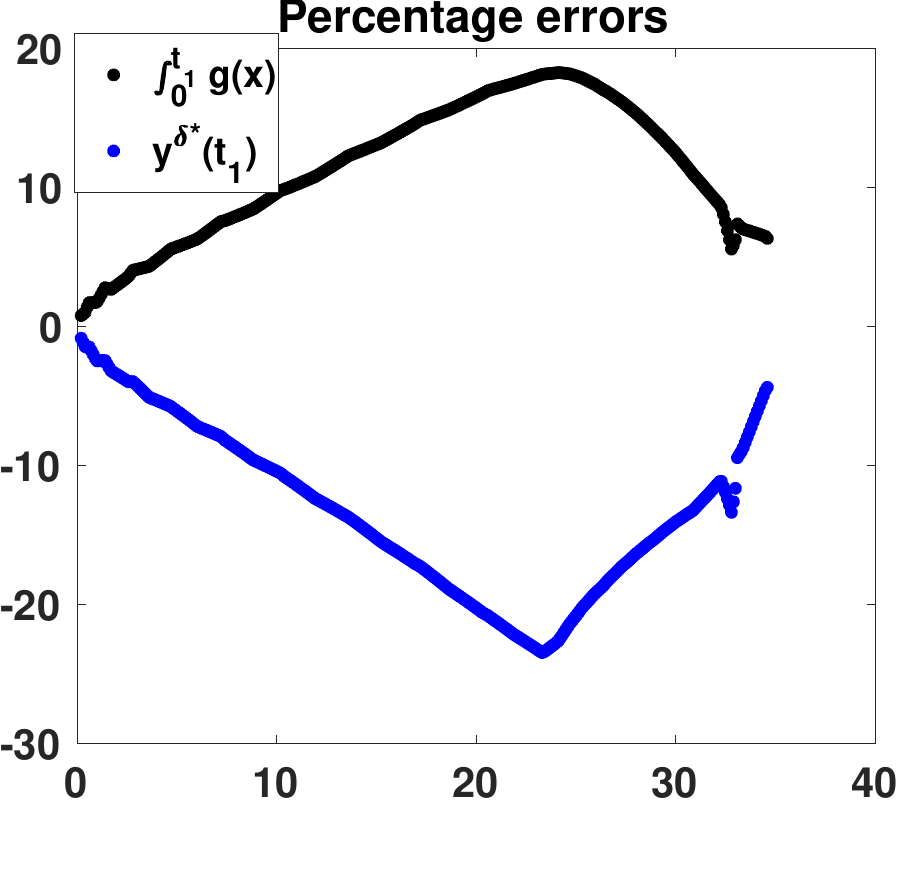}
    \end{minipage}
    \caption{ Difference in Pareto frontiers:  $\delta = 0.2, \bar{u} = .95,   \alpha(t) = 1+|\cos\left(\frac{3 \pi}{2} + t \frac{1.7 \pi }{3000} \right)|, t_1=1, \mu_{id} = 11.5$.  
    }
\label{Fig_pareto_with_non_zero_mu_ideal}
\vspace{-3mm}
\end{figure*}}

\ignore{

\begin{thm} {\color{magenta}[Continuity of $y^\delta$]} Fix $\sigma$ and $L_{mx} \equiv x$. Then 
$$\sup_{s \in T}\left|y^{\delta_n}(s) -  y^{0}(s)\right| \to 0, \mbox{ when } \delta_n \to  0.$$  
\end{thm}

\textbf{Proof:}
We begin with some  definitions that are  used throughout the proof. 
Define $\Dw_0  := 0,$ and then the  following terms  for each  $i \ge 1$, recursively,  
\begin{eqnarray}
    \Up_i &:=& \inf\{ t > \Dw_{i-1}: x(t)> y^\delta(t) - \delta, \ \bdot{x}(t) >0 \} = \inf\{ t > \Dw_{i-1}:  \bdot{y}^\delta(t)  >0 \}, \nonumber\\
    \Dw_i &:=& \inf\{ t > \Up_i: x(t) \le  y^\delta(t) - \delta \mbox{ or }  \bdot{x}(t) \le  0 \}=  \inf\{ t > \Up_i:  \bdot{y}^\delta(t)  =0 \}.  
    \label{Eqn_definition_ui_di}
\end{eqnarray}

Towards proving the theorem, it   suffices to prove the following, for any $\delta> 0$,
\begin{eqnarray}
&&y^0(t) \le     y^\delta (t) \le y^0 (t) + \delta,  \mbox{ for all } t \in T, \mbox{ and, } \nonumber\\
&&x(t) \ge  y^\delta (t)- \delta,  \mbox{ for all } t \in \cup_{i \ge 1} [\Dw_{i-1}. \Up_i],  \label{Eqn_Sw}
\end{eqnarray}
As the functions are absolutely continuous,  one can have atmost countably many  $\{\Dw_i, \Up_i\}_i$ (see \cite{royden2010real}).
%
Thus, the result is derived using
 mathematical induction on $i$---with  $\I_\ell := (\Up_\ell, \Dw_\ell) \cup [\Dw_\ell, \Up_{\ell+1}]$ for each $\ell$,  we prove \eqref{Eqn_Sw}     for all $t \in   \I_i$,  assuming the  same for  all previous intervals, i.e., for all $t \in \cup_{\ell < i} \I_\ell$. 
 The proof in each induction step  is derived using the following  two major parts:
\begin{enumerate}
    \item [1.]   We  will show that $x(t) \in [y^\delta(t) - \delta, y^\delta(t)]$ and   $y^0(t) \in [x(t), y^\delta(t)]$, for all $t \in \cup_i(\Up_i, \Dw_i)$. 
\item [2.] Second step is to show that   
$y^\delta(t) = y^\delta (\Dw_i)$ and $y^0(t) = y^0 (\Dw_i)$ for  $t \in [\Dw_i, \Up_{i+1}]$ for any $i$ (both the trajectories remain constant in these intervals). 
\end{enumerate}
Further we will see that the initial step (or the base case) of the induction proof follows exactly as any  other inductive step---if  $\Up_1 = 0$,  the initial step  is merged with  part (a), otherwise with part (b). 
 
{\bf Step 1:} Consider   $(\Up_i, \Dw_i)$  for some  $i$ and assume  \eqref{Eqn_Sw} is satisfied for all $t \le \Up_i$, which implies the following at $t = \Up_i$:
$$
y^0(\Up_i) \in [y^\delta(\Up_i) - \delta, y^\delta(\Up_i)] \mbox{ and } x(\Up_i)  \ge y^\delta(\Up_i) - \delta.
$$ 
Observe the above conditions are satisfied even for the first interval,   when $\Up_1 = 0$, as the initial conditions  are  $y^0(0) = y^\delta(0) = x(0)$ and the proof of the induction step remains the same even for such  an initial step. Hence we provide one proof as below.

To begin with, from \eqref{Eqn_definition_ui_di},
  $ x(t) >  y^\delta (t) - \delta$ and hence 
  \begin{eqnarray}
  \label{Eqn_one_dir}
      y^0(t) \ge x(t) >  y^\delta (t) - \delta \mbox{ for all } t \in (\Up_i, \Dw_i). 
  \end{eqnarray}
  If further  $y^0$ does not change, i.e., if  $y^0(t) = y^0 (\Up_i)$,   for all  $ t \in (\Up_i, \Dw_i)$, then  \eqref{Eqn_Sw} is satisfied for this interval $(\Up_i, \Dw_i)$--- since $y^\delta$ is non-decreasing,  $y^\delta (t) \ge y^\delta (\Up_i) \ge y^0(\Up_i) = y^0(t)$, for such $t$.  And this part of the inductive step is proved. 
  
Next, consider the case when $y^0$ changes:  there exists an  $ s \in  [\Up_i, \Dw_i)$ such that,
\begin{eqnarray}
\label{Eqn_y0_in_ui_to_di}
 \hspace{10mm}   y^0(\Up_i) =  y^0(t) > x(t) \mbox{ for  } t \in (\Up_i, s) \mbox{ and } \ 
    y^0(t) =  x(t),  \mbox{ for all } t \in  [s, \Dw_i) . 
\end{eqnarray}

From the definitions in \eqref{Eqn_definition_ui_di},
   the derivative  $\bdot{y}^\delta (t) > 0$,  for all $t \in (\Up_i, \Dw_i)$; further from \eqref{eq:ydelta}, we have  $0< \bdot{y}^\delta (t) \le \bdot{x}(t)$. By virtue of the smaller derivative for~$y^\delta$, we only have  the following two possibilities in this interval and sub-case \eqref{Eqn_y0_in_ui_to_di}. 

\begin{enumerate}[(a)]
    \item  The trajectory $y^\delta$ remain bigger, i.e.,  $y^\delta(t) > x(t)$ for all $t \in (\Up_i, \Dw_i)$.

    Since $y^\delta$ is non-decreasing,   we have $y^0(t) = y^0(\Up_i) \le y^\delta (\Up_i) \le y^\delta (t)$ for all $t \in (\Up_i, s)$, for  $s$  of \eqref{Eqn_y0_in_ui_to_di}; and for $t \in [s, \Dw_i)$, we again have 
    $y^\delta (t) > x(t) = y^0(t)$. In all, \eqref{Eqn_Sw} is again satisfied in  $(\Up_i, \Dw_i)$ for this condition also, see \eqref{Eqn_one_dir}. 
\item  There exists an $s' \in (\Up_i, \Dw_{i})$ such that   $y^\delta(t) > x(t)$ for $t \in (\Up_i, s')$ and $y^\delta(t) = x(t)$  for $t\in [s', \Dw_i)$.

Now for $t < s'$, \eqref{Eqn_Sw}  is satisfied as in previous $(a)$ case; and at $s'$,  we have 
$x(s') \le y^0(s') \le y^\delta(s') = x(s')$, implying $x(s') = y^\delta(s') = y^0(s')$;  finally for $t \in  (s' , \Dw_i)$,  \eqref{Eqn_Sw} is satisfied with
$y^\delta (t) = x(t) = y^0(t)$  (again using \eqref{Eqn_one_dir} and note here  $s \le s'$).

\end{enumerate}
In summary,  \eqref{Eqn_Sw} is satisfied for all $t \le \Dw_i$,  by continuity and further using the induction assumption.

{\bf {Step 2:}} Now consider  $(\Dw_i, \Up_{i+1})$  for some  $i$ and assume  \eqref{Eqn_Sw} is satisfied for all $t \le \Dw_i$, which implies: 
$$
y^0(\Dw_i)  = [y^\delta(\Dw_i),  y^\delta(\Dw_i) +\delta] \mbox{ and } y^\delta (\Dw_i) -\delta < x(\Dw_i).
$$
Again the initial step also satisfies the above because of the given initial conditions, when  $\Up_1 > 0 = \Dw_0$

For any $t \in [\Dw_{i}, \Up_{i+1}]$, by definition $\bdot{y}^\delta = 0$ and 
thus $y^\delta (t) = y^\delta (\Dw_{i}^\delta)$, for all such $t$.  Further, $x (t) < y^\delta (t) - \delta$ for $t \in (\Dw_{i}, \Up_{i+1})$, by definition of $\Up_i$ in \eqref{Eqn_definition_ui_di}. Therefore,  
 \begin{eqnarray*}
     y^0(t) \ge y^0 (\Dw_i) \ge y^\delta (\Dw_i ) - \delta = y^\delta (t)-\delta  > x(t), \mbox{ for all such } t   
 \end{eqnarray*}
 implying $y^0$ trajectory never meets $x$ trajectory or that $y^0(t) = y^0 (\Dw_i)$ for $t \in (\Dw_{i}, \Up_{i+1})$.
 
In summary, the comparison status at $\Dw_i$ continues for all $t \in [\Dw_i, \Up_{i+1}]$, hence $y^\delta(t) \in [y^0(t), y^0(t) + \delta]$  for such $t$,  and so \eqref{Eqn_Sw} is satisfied for all $t \le \Up_{i+1}$, after using the induction assumption. \eop

As $\y$ is continuous in $\sigma$, thus, if $\sigma_n \to \sigma$ in weak topology, imply,  for any fixed $\delta_n$
$$\sup_{s \in T}\left|y^{(\delta_n, \sigma_n)}(s) -  y^{(\delta_n, \sigma)}(s)\right| \to 0. $$ 
Thus, 
\begin{eqnarray*}
    \sup_{s \in T}\left|y^{(\delta_n, \sigma_n)}(s) -  y^{(0, \sigma)}(s)\right| &\le& \sup_{s \in T}\left|y^{(\delta_n, \sigma_n)}(s) -  y^{(\delta_n, \sigma)}(s)\right| +  \sup_{s \in T}\left|y^{(\delta_n, \sigma)}(s) -  y^{(0, \sigma)}(s)\right| \\
    & \rightarrow& 0, \mbox{ as } n \to \infty. 
\end{eqnarray*}}

\Remove{
We investigate the same in immediate next. 
\section{Dynamic programming equation}
\label{sec_DP} In this section, we relabel the control space $\U$ as $\U(T)$.  
We begin with   introducing a mapping $\vee$ below, for any fixed  $r \in [\tau, t_1]$ and $y$ 
\begin{eqnarray*}
    &&\vee_{r, y}^{t_1} (s; \x) :=\max\left \{ y,  \sup_{t \in [r, s]} L_{mx} (t, x(t) ) \right  \} \mbox{ for all } s \in [r, t_1], \mbox{ and } \vee_{r, y}^{t_1} (t_1; \x):= \vee_{r, y}^{t_1} ( \x). 
\end{eqnarray*}   
This leads to the following control problem with objective function:
\begin{equation}
    J( \tau, x, y;\u) =   \int_{\tau}^{t_1} L_r(s, x(s), u(s))ds - \vee_{\tau, y}^{t_1} ( \x) + \Psi(x(t_1)). 
\label{Eqn_objective_fun}
\end{equation}
Observe that the second term depends on entire trajectory $\x := \{ x(s)\}_{s}$ driven by control trajectory $\u := \{ u(s)\}_{s}$, with initial condition $x(\tau) = x$. 
Thus, define the \textit{value function} by 
\begin{equation}
    v(\tau, x,y) = \sup_{\u  \in \mathcal{U}([\tau, t_1]) } J(\tau, x, y;\u). 
    \label{Eqn_value_fun}
\end{equation}

 We have now converted the control problem in standard framework, we will establish well known principles of optimal control theory. One of them is dynamic programming principle. Proof is in Appendix \ref{App_lem_dp}. 
\begin{lemma}
\label{Lem_value_fun_ineq} For every initial condition $(\tau, x,y)$ and  $r \in [\tau, t_1]$, control $\u \in \mathcal{U}([\tau, r])$, we have
\begin{equation}
    v(\tau,x,y)  \ge    \left\{ \int_\tau^r L_r(t, x(t), u(t)) dt + v(r, x(r), z) \right\}, \mbox{ where } z_{\u} = \vee_{\tau, y}^r (\x).
    \label{Eqn_value_ineq1}
\end{equation}
where $\x$ is the state  trajectory corresponding to control $\u$ with initial condition $x(\tau) = x$. 
\end{lemma}
\begin{thm}{\textbf{[Dynamic Programming]}}
    For every initial condition $(\tau, x,y)$ and $r \in [\tau, t_1]$, we have 
    \begin{equation}
    \label{Eqn_DP}
        v(\tau,x,y)  = \sup_{\u \in \mathcal{U}([\tau, r])} \left\{ \int_\tau^r L_r(t, x(t), u(t)) dt + v(r, x(r), z_\u) \right\}, \mbox{ where } z_\u = \vee_{\tau, y}^r (\x). 
    \end{equation}
\end{thm}
\textbf{Proof:} Consider any $\delta > 0$ and from \eqref{Eqn_value_fun}, it is possible to choose a $\u \in \mathcal{U}([\tau, t_1])$ such that:
$$
\begin{aligned}
\delta +v(\tau, x, y) & \leq J(\tau, x, y; \u) \\
& =\int_\tau^r L_r(s, x(s), u(s)) d s+\int_{r}^{t_1} L_r(s, x(s), u(s)) d s +\Psi( x(t_1)) \\
& \hspace{10mm}-  \max \left\{ y, \sup_{t \in [\tau, r]} L_{mx}(t, x(t)), \sup_{t \in [r, t_1]} L_{mx}(t, x(t)) \right\} \\
& {=}\int_\tau^r L_r(s, x(s), u(s)) d s+\int_{r}^{t_1} L_r(s, x(s), u(s)) d s \\
& \hspace{10mm} -  \max \left\{ z_\u, \sup_{t \in [r, t_1]} L_{mx}(t, x(t)) \right\}+\Psi( x(t_1)) \\
& {=} \int_\tau^r L_r(s, x(s), u(s)) d s+J(r , x(r), z_\u ; \u'), \quad (\mbox{where } \u' := \u|_{[r,t_1]}) \\
& {\leq} \int_\tau^r L_r(s, x(s), u(s)) d s+ v(r, x(r),z_\u).
\end{aligned}
$$
Since $\delta$ is arbitrary, we have proved \eqref{Eqn_DP}, further using 
Lemma \ref{Lem_value_fun_ineq}. \eop

The dynamic programming relation in \eqref{Eqn_DP} leads to a non-linear PDE with discontinuities induced by the supremum term. At present, the existence of solutions to the PDE remains an open question. However, one may attempt to characterize solutions to the associated HJB equation using the framework of Filippov-type solutions, as developed in \cite{barles2013bellman, rao2013hamilton}. We leave this direction as a topic for future research.}


\ignore{

\begin{lemma}{[Continuity of $\vee$]}
\label{lem:vee_continuity}
Assume that $L_{mx}:T\times\mathbb{R}^n\to\mathbb{R}$ is bounded and
uniformly continuous in $(t,x)$, and let  $K_\infty$ be an increasing function such that
\[
\big|L_{mx}(t,x)-L_{mx}(s,z)\big|
\le K_\infty(|t-s|+|x-z|).
\]
Let $\x(\cdot) = \{x(s) \}_s$ and $\tilde \x(\cdot) = \{\tilde{x}(s) \}_s$ be two continuous trajectories on $[\tau,T]$ with $\x(\tau) = x_\tau$ and $\tilde{\x}(\tau) = \tilde{x}_\tau$, which further
satisfy
\begin{equation}
\label{eq:traj_estimate}
|x(s)-\tilde x(s)|\le |x_\tau-\tilde {x}_\tau|e^{K(s-\tau)}, \qquad s\in[\tau,T],
\end{equation}
for some constant $K>0$. Then,
\begin{equation}
\label{eq:vee_continuity}
\big|
\vee_{\tau,y}^{t_1}(\x)-\vee_{\tau,y}^{t_1}(\tilde{\x})
\big|
\le
K_\infty\!\left(|x_\tau-\tilde {x}_\tau|e^{KT}\right).
\end{equation}
\end{lemma}

\medskip
\noindent \textbf{Proof:}
By definition, for the two trajectories $\x$ and $\tilde{\x}$ given in the hypothesis 
\[
\vee_{\tau,y}^{t_1}(\x)
=
\max\Big\{
y,\,
\sup_{t\in[\tau,T]} L_{mx}(t,x(t))
\Big\},
\qquad
\vee_{\tau,y}^{t_1}(\tilde{\x})
=
\max\Big\{
y,\,
\sup_{t\in[\tau,T]} L_{mx}(t,\tilde x(t))
\Big\}.
\]
By continuity of $\x$ and $\tilde \x$,   $A:=\sup_{t\in[\tau,T]} L_{mx}(t,x(t)) = \max_{t\in[\tau,T]} L_{mx}(t,x(t))$ and $
B:=\max_{t\in[\tau,T]} L_{mx}(t,\tilde x(t))$.
Using the  inequality
$$|\max\{y,A\}-\max\{y,B\}|\le |A-B|,$$ we obtain
$\big|
\vee_{\tau,y}^{t_1}(\x)-\vee_{\tau,y}^{t_1}(\tilde{\x})
\big|
\le
\big|A-B\big|$.
Moreover, for any bounded functions $f$ and $g$,
\[
\big|\max_t f(t)-\max_t g(t)\big|
\le
\max_t |f(t)-g(t)|.
\]
Applying this inequality yields
\[
|A-B|
\le
\max_{t\in[\tau,T]}
\big|
L_{mx}(t,x(t)) - L_{mx}(t,\tilde x(t))
\big|.
\]
By uniform continuity of $L_{mx}$,
\[
\big|
L_{mx}(t,x(t)) - L_{mx}(t,\tilde x(t))
\big|
\le
K_\infty(|x(t)-\tilde x(t)|),
\qquad \forall\, t\in[\tau,T].
\]
Using \eqref{eq:traj_estimate} and monotonicity of $K_\infty$, we obtain
\[
K_\infty(|x(t)-\tilde {x}(t)|)
\le
K_\infty(|x_\tau-\tilde {x}_\tau|e^{KT}),
\qquad \forall\, t\in[\tau,T].
\]
Therefore,
$|A-B|
\le
K_\infty(|x_\tau-\tilde {x}_\tau|e^{KT})$,
which proves \eqref{eq:vee_continuity}. \eop

\begin{proposition}[Uniform and joint continuity of the value function]
\label{prop:uniform_joint_continuity}
Assume that:
\begin{itemize}
\item[(A1)] The functions $f,L$ are bounded and satisfy, for all $u\in U$,
\begin{eqnarray*}
|f(t_1,x_1,u)-f(t_2,x_2,u)|
&\le&
K_f\big(|t_1-t_2|+|x_1-x_2|\big),\\
|L_r(t_1,x_1,u)-L_r(t_2,x_2,u)|
&\le&
K_L\big(|t_1-t_2|+|x_1-x_2|\big).
\end{eqnarray*}
\item[(A2)] The terminal cost $\Psi$ is bounded and Lipschitz continuous, i.e.,
\[
|\Psi(x_1)-\Psi(x_2)|\le K_\Psi |x_1-x_2|.
\]
\item[(A3)] The function $L_{mx}$ is bounded and satisfies
\[
|L_{mx}(t_1,x_1)-L_{mx}(t_2,x_2)|
\le
K_\infty\big(|t_1-t_2|+|x_1-x_2|\big).
\]
\item[(A4)] The control set $U$ is compact.
\end{itemize}
Then the value function
\[
v(\tau,x,y)
=
\sup_{u(\cdot)\in\mathcal U([\tau,T])}
\left\{
\int_\tau^{t_1} L_r(s,x(s),u(s))\,ds
-
\sigma\,\vee_{\tau,y}^{t_1}(x)
+
\Psi(x(t_1))
\right\}
\]
is uniformly continuous on $T\times\mathbb R^n\times\mathbb R$, and hence
jointly continuous in $(\tau,x,y)$.
\end{proposition}

\textbf{Proof:}
We establish uniform continuity by showing continuity in each variable with
bounds that are uniform over the remaining variables.

\medskip
\noindent\textbf{Continuity wrt $x$.}
Fix $(\tau,x,y)$ and $(\tau,\tilde x,y)$, and let
$u(\cdot)\in\mathcal U([\tau,T])$ be arbitrary.
Denote by $x(\cdot)$ and $\tilde x(\cdot)$ the corresponding state trajectories. 
By assumption (A1) and Gr\"onwall’s inequality,
\[
|x(s)-\tilde x(s)|
\le
|x-\tilde x|e^{K_f(s-\tau)},
\qquad s\in[\tau,T].
\]
In particular,
\[
\sup_{s\in[\tau,T]}|x(s)-\tilde x(s)|
\le
|x-\tilde x|e^{K_f T}.
\]
Using assumption (A3) and Lemma~\ref{lem:vee_continuity}, it follows that
\[
\big|
\vee_{\tau,y}^{t_1}(x)-\vee_{\tau,y}^{t_1}(\tilde x)
\big|
\le
K_\infty\big(|x-\tilde x|e^{K_f T}\big).
\]
Consequently,
\begin{eqnarray*}
|J(\tau,x,y;u)-J(\tau,\tilde x,y;u)|
&\le&
\int_\tau^{t_1}
|L_r(s,x(s),u(s))-L_r(s,\tilde x(s),u(s))|\,ds
\\
&&\quad
+\sigma\big|
\vee_{\tau,y}^{t_1}(x)-\vee_{\tau,y}^{t_1}(\tilde x)
\big|
+
|\Psi(x(t_1))-\Psi(\tilde x(T))|
\\
&\le&
T K_L\big(|x-\tilde x|e^{K_f T}\big) +
\sigma K_\infty\big(|x-\tilde x|e^{K_f T}\big)+
K_\Psi\big(|x-\tilde x|e^{K_f T}\big).
\end{eqnarray*}
Taking the supremum over all admissible controls yields
\[
|v(\tau,x,y)-v(\tau,\tilde x,y)|
\le
C_x |x-\tilde x|,
\]
where $C_x>0$ is independent of $(\tau,x,y)$.

\medskip
\noindent\textbf{Continuity wrt $\tau$.}
Fix $(\tau,x,y)$ and $r>\tau$.
Let $u(\cdot)\in\mathcal U([\tau,T])$ and restrict it to $[r,T]$.
Then
\[
|J(\tau,x,y;u)-J(r,x,y;u)|
\le
\int_\tau^r |L_r(s,x(s),u(s))|\,ds
+
|J(r,x(r),y;u)-J(r,x,y;u)|.
\]
By boundedness of $L$ and the estimate obtained above for the state trajectory,
there exists $C_\tau>0$ such that
\[
|J(\tau,x,y;u)-J(r,x,y;u)|
\le
C_\tau |r-\tau|.
\]
Taking the supremum over $u(\cdot)$ gives
\[
|v(\tau,x,y)-v(r,x,y)|\le C_\tau |r-\tau|.
\]

\medskip
\noindent\textbf{Continuity wrt $y$.}
For fixed $(\tau,x)$ and any $y,\tilde y\in\mathbb R$,
\[
|\vee_{\tau,y}^{t_1}(x)-\vee_{\tau,\tilde y}^{t_1}(x)|
\le
|y-\tilde y|,
\]
which implies
\[
|v(\tau,x,y)-v(\tau,x,\tilde y)|
\le
\sigma |y-\tilde y|.
\]

\medskip
\noindent
Combining the above estimates, there exists a constant $C>0$ such that
\[
|v(\tau,x,y)-v(\tilde\tau,\tilde x,\tilde y)|
\le
C\big(|\tau-\tilde\tau|+|x-\tilde x|+|y-\tilde y|\big)
\]
for all $(\tau,x,y)$ and $(\tilde\tau,\tilde x,\tilde y)$ in
$T\times\mathbb R^n\times\mathbb R$.
Therefore, $v$ is uniformly continuous on this domain.
Since uniform continuity implies joint continuity, the proof is complete.
 \eop

\begin{thm}
Assume the existence of optimal control $\u^*$. Then, the value function $v$ is the unique Lipschitz continuous viscosity solution of the Hamilton--Jacobi--Bellman equation \eqref{Eqn_hjb}.
\end{thm}
\textbf{Proof:} We divide the proof into four parts: (i) viscosity sub-solution property, (ii) viscosity super-solution property, (iii) terminal condition, and (iv) uniqueness.

\medskip
\noindent\textbf{Viscosity sub-solution property.}
Fix $(t,x,y)\in[0,T)\times\mathbb R^d\times\mathbb R$.
Let $\phi\in C^1(T\times\mathbb R^d\times\mathbb R)$ be a test function such that
\[
(v-\phi)(t,x,y)=0
\quad\text{and}\quad
(v-\phi)(s,z,w)\le0
\]
for all $(s,z,w)$ in a neighborhood of $(t,x,y)$.
By definition, $\phi$ touches $v$ from above at $(t,x,y)$.

Let $u^*(\cdot)$ be the optimal control for initial data $(t, x, y)$ and let $x^*(\cdot)$ be the corresponding state  with $x^*(t) = x$. By the Dynamic Programming Principle, for every $h>0$ sufficiently small,
\begin{equation}
v(t,x,y) = 
\int_t^{t+h} L_r(s,x^*(s),u^*(s))\,ds + v(t+h,x^*(t+h),\vee_{t, y}^{t+h} (\x^*)).
\label{eq:DPP-sub}
\end{equation}
Since $\phi\ge v$ in a neighborhood of $(t,x,y)$, i.e., 
$$\phi(t+h,x^*(t+h),\vee_{t, y}^{t+h} (\x^*)) \ge v(t,x^*(t),\vee_{t, y}^{t} (\x^*))$$ 
and $\phi(t,x,y)=v(t,x,y)$, we may replace $v$ by $\phi$ on the right-hand side of
\eqref{eq:DPP-sub}, obtaining
\begin{equation}
0 \le \int_t^{t+h} L_r(s,x^*(s),u^*(s))\,ds + \phi (t+h,x^*(t+h),\vee_{t, y}^{t+h} (\x^*)) - \phi(t,x^*(t),\vee_{t, y}^{t} (\x^*).
\label{eq:sub-phi}
\end{equation}
We now expand the test function.
Since $\phi$ is continuously differentiable and $(x(\cdot),\vee(\cdot))$ are continuous,
a first-order Taylor expansion yields
\begin{eqnarray*}
\phi(t+h,x(t+h),\vee_{t, y}^{t+h} (\x))) &= &\phi(t,x,y)
+h\,\phi_t(t,x,y) +h\,D_x\phi(t,x,y)\cdot f(t,x,u)
\\
&& +h\,D_y\phi(t,x,y)f_y(t,x,u)\,
\mathds{1}_{\{y=L_{mx}(t,x),\,f_y(t,x,u)>0\}} +o(h^2),
\end{eqnarray*}
where $o(h^2)/h\to0$ as $h\downarrow0$. Substituting the expansion into \eqref{eq:sub-phi}, dividing by $h$, and letting $h\downarrow0$, we obtain
\begin{eqnarray*}
    0 &\le& L_r(t,x,u^*) + \phi_t(t,x,y) + D_x\phi(t,x,y)\cdot f(t,x,u^*) \\
    && \hspace{50mm}+ D_y\phi(t,x,y)\,f_y(t,x,u^*)\, \mathds{1}_{\{y=L_{mx}(t,x),\,f_y(t,x,u^*)>0\}}.
\end{eqnarray*}
Since the control $u^*$ is optimal, we have 
\begin{eqnarray*}
   && \phi_t(t,x,y) +\sup_{u\in U}
\left\{ D_x\phi(t,x,y)\cdot f(t,x,u)  \right.\\
&& \left. \hspace{40mm} +
D_y\phi(t,x,y)\,f_y(t,x,u)\,
\mathds{1}_{\{y=L_{mx}(t,x),\,f_y(t,x,u)>0\}}+ L_r(t,x,u)\right\} \le 0.
\end{eqnarray*}
This proves that $v$ is a viscosity sub-solution of \eqref{Eqn_hjb}.

\medskip
\noindent\textbf{Viscosity super-solution property.}
Let $\phi\in C^1$ touch $v$ from below at $(t,x,y)$.
Fix $\varepsilon>0$.
By the Dynamic Programming Principle, there exists an $\varepsilon$-optimal control
$\u^\varepsilon\in\mathcal U([t,T])$ such that
\[
v(t,x,y)
\ge
\int_t^{t+h} L_r(s,x^\varepsilon(s),u^\varepsilon(s))\,ds
+
v(t+h,x^\varepsilon(t+h),\vee_{t, y}^{t+h}(\x^\varepsilon))
-
\varepsilon h.
\]
Repeating the previous expansion and letting $h\downarrow0$, then $\varepsilon\downarrow0$,
yields the reverse inequality, proving the viscosity super-solution property.

\medskip
\noindent\textbf{Terminal condition.}
At $t=T$, by definition of the supremum functional,
$\vee_{T,y}^{t_1}(\x)= \max\{y, L_{mx}(T, x(T))\} $,
and hence $v(T,x,y) = -\sigma \max\{y, L_{mx}(T, x)\}+\Psi(x)$,
which establishes the terminal condition in the viscosity sense. \eop

\begin{thm}
    \begin{enumerate}
        \item 
 Let $W(\cdot , \cdot, \cdot)$ be the solution of the following boundary value problem:
      \begin{eqnarray*}
           && \hspace{-10mm} v_t +  \sup_u \left \{ v_x \cdot f(t, x, w) +   v_y  \cdot f_y(t,x,w) \mathds{1}_{\{y =  L_{mx}(t,x), f_y(t,x,w)>0\}}+   L_r(t, x, w)      \right \} = 0, \label{Eqn_hjb}\\
            && \hspace{5cm}\mbox{ with }  v(T, x, y) = - \sigma \max\{y, L_{mx}(T, x) \}+ \Psi(x). \nonumber
     \end{eqnarray*}
     Then $W (t,x, y) \ge v (t,x, y)$ for all $(t,x, y)$. 
     \item  If there exists a $\u^* \in \mathcal{U}([\tau, t_1])$ such that   with $\x^*$ representing the ODE solution under $\u^*$ with  some initial condition $(\tau, x, y)$  and for almost all $ t \in [\tau, t_1]$
           \begin{eqnarray}
             && \hspace{-1cm}\sup_{u \in U} \left\{v_x f(\tau,x,u)  + v_y f_y(t,x,u) \mathds{1}_{\{y =  L_{mx}(t,x), f_y(t,x,u)>0\}}+L_r(t,x,u)\right\} \nonumber \\
              && \hspace{1cm}= v_x f(\tau,x^*,u^*)  + v_y f_y(t,x^*,u^*) \mathds{1}_{\{y =  L_{mx}(t,x^*), f_y(t,x^*,u^*)>0\}}+L_r(t,x^*,u^*), \label{Eqn_verification_cond}
     \end{eqnarray}
    then $u^*$ is optimal for that $(\tau, x, y)$ and $W(\tau, x, y) = v(\tau, x, y)$.  
    \end{enumerate}
    \end{thm}
   \textbf{Proof:}
Consider any $\u \in \mathcal{U}([t,T])$ and corresponding state trajectory $\x$. Using multivariate calculus and the dynamic programming equation \eqref{Eqn_DP}, we obtain
\begin{eqnarray}
&&W\left(T, x\left(T\right), \vee_{t, y}^{t_1} (\x)\right)=  W(t, x, y)+\int_t^{T}\left[\frac{\partial}{\partial t} W(s, x(s), \vee_{t, y}^{t_1} (s;\x)) \right. \nonumber\\
&& \hspace{40mm}+f(s, x(s), u(s)) \cdot D_x W(s, x(s), \vee_{t, y}^{t_1} (s;\x))\nonumber \\
&& \hspace{40mm}\left. + f(s, x(s), u(s)) \cdot D_y W(s, x(s), \vee_{t, y}^{t_1} (s;\x)) \mathds{1}_{\{y =  L_{mx}(s,x), f_y(s,x,w)>0\}}\right] d s\nonumber\\
&&  \hspace{33mm}\leq W(t, x,y)-\int_t^{T} L_r(s, x(s), u(s)) d s. \quad (\mbox{using \eqref{Eqn_DP}})\label{Eqn_verifi_part1}
\end{eqnarray}
Also, $W\left(T, x\left(T\right), \vee_{t, y}^{t_1} (\x) \right)= - \sigma \max\{y, L_{mx}(T,x(T))\} + \Psi(x(t_1))$. Hence
\begin{equation}
W(t, x,y) \geq J(t, x,y ; \u) .
\label{Eqn_verif_part1_result}
\end{equation}
We get \eqref{Eqn_DP} by taking the supremum over $\u$.
To prove the second assertion of the theorem, let $\u^* \in \mathcal{U}([t,T])$ satisfy \eqref{Eqn_verification_cond}. We redo the calculation above with $\u^*$. This yields \eqref{Eqn_verifi_part1} with an equality. Hence
$$
W(t, x,y)=J\left(t, x,y ; \u^*\right) .
$$
By combining this equality with \eqref{Eqn_verifi_part1}, we conclude that $\u^*$ is optimal at $(t, x,y)$. \eop

\begin{thm}\textbf{[Verification Theorem]} {\color{red}[In viscosity sense]}
\label{thm:HJB_full}
Let $v(t,x,y)$ denote the value function of the augmented optimal control problem,
and assume that $v$ is differentiable wrt $(t,x)$ at the point
$(t,x,y)$ with $y \ge L_{mx} (t,x)$. Then the following statements hold.

\begin{enumerate}
\item For every $w\in\mathcal U$,
$$
v_t(t,x,y)+v_x(t,x,y) f(\tau,x,w)  + v_y(t,x,y) f_y(t,x,w) \mathds{1}_{\{y =  L_{mx}(t,x), f_y(t,x,w)>0\}}+L_r(t,x,w)\le 0.
$$

\item If there exists an optimal control policy $u^*(\cdot)$ such that $\lim_{s\downarrow t}u^*(s)=w^*\in\mathcal U,$  for some $w^*$ then 
\begin{eqnarray*}
&&\hspace{-5mm}v_t(t,x,y)+v_x(t,x,y)  f(\tau,x,w^*) \\
  &&\hspace{15mm}+ v_y(t,x,y) f_y(t,x,w^*) \mathds{1}_{\{y =  L_{mx}(t,x), f_y(t,x,w^*)>0\}}+L_r(t,x,w^*)=0.
\end{eqnarray*}
\end{enumerate}
\end{thm}

\textbf{Proof:} \underline{\emph{Proof of part (a).}} Consider $u$ such that $u(s) \to w$ as $s \to t$. 
Substituting the $r = t+h$ in Dynamic Programming Principle (DPP) (see \eqref{Eqn_value_ineq1}),
\begin{equation}
    v(t,x,y)
\stackrel{a}{\ge}
\int_t^{t+h}L_r(s,x(s),u(s))\,ds
+
v(t+h,x(t+h), \vee_{t,y}^{t+h}(\x)).
\label{Eqn_DPP_inequality}
\end{equation}
Since $v$ is differentiable wrt $(t,x)$ at $(t,x,y)$ and
$\bdot x(t)=f(t,x,w)$ and define $\bdot{L}_{mx}(t, x) =: f_y(t, x, w)$, a first-order Taylor expansion yields {\color{red}Does $v_y$ exist in viscosity sense?}
\begin{eqnarray*}
    v(t+h,x(t+h), \vee_{t,y}^{t+h}(\x))
&=&
v(t,x,y)
+
h\,v_t(t,x,y)
+
h v_x(t,x,y)\cdot f(t,x,w)  \\
&& \hspace{14mm}+ h v_y(t,x,y)\cdot f_y(t,x,w) \mathds{1}_{\{y =  L_{mx}(t,x), f_y(t,x,w)>0\}} +
o(h^2),
\end{eqnarray*}
where $o(h^2)/h\to0$ as $h\downarrow0$. Also, observe that $\lim_{h \to 0} \vee_{t,y}^{t+h}(\x) = y$.  

We now substitute this expansion into \eqref{Eqn_value_ineq1}
and subtracting $v(t,x,y)$ from both sides and dividing by $h$ gives
\begin{eqnarray*}
    && \hspace{-15mm}0 \ge \frac{1}{h}\int_t^{t+h}L_r(s,x(s),u(s))\,ds
+
v_t(t,x,y)
+ v_x(t,x,y)\cdot f(t,x,w) \\
&& \hspace{15mm}+ v_y(t,x,y)\cdot f_y(t,x,w) \mathds{1}_{\{y =  L_{mx}(t,x), f_y(t,x,w)>0\}}
+
\frac{o(h^2)}{h}.
\end{eqnarray*}
Since $L$ is continuous, the integral term converges as (here $u(s) \to w$, when $s \to t$)
$$
\frac{1}{h}\int_t^{t+h}L_r(s,x(s),u(s))\,ds
\;\longrightarrow\;
L_r(t,x,w)
\quad\text{as }h\downarrow0,
$$
and by construction $\frac{o(h^2)}{h} \to 0$. Passing to the limit $h\downarrow0$,
we obtain
$$
v_t(t,x,y)
+ v_x(t,x,y) \cdot f(t,x,w) + v_y(t,x,y) \cdot f_y (t,x,w) \mathds{1}_{\{y =  L_{mx}(t,x), f_y(t,x,w)>0\}}
+
L_r(t,x,w)
\le 0.
$$

Because the control value $w\in\mathcal U$ was arbitrary, this inequality holds
for all admissible controls, which establishes part~(a).

\medskip

\noindent \underline{\emph{Proof of part (b).}}
Assume that there exists an optimal control policy $u^*(\cdot)$ such that
$u^*(s)\to w^*$ as $s\downarrow t$.
Now, we have  equality at $(a)$ in place of inequality in \eqref{Eqn_DPP_inequality} and hence the result.  \eop
}

\section{Conclusions}
\label{sec_conclusions}

This work provides a systematic framework for analyzing optimal control problems that simultaneously account for cumulative (integral) performance and peak ($L^\infty$) behavior. While classical formulations address these criteria separately, the combined problem considered here captures a more realistic trade-off that arises in many practical systems, while solving multi-objective problems.

From a theoretical standpoint, we establish the existence of optimal solutions within the relaxed control framework and further show that classical pure controls can approximate these solutions arbitrarily well. At the same time, the non-smoothness introduced by the supremum term makes direct computation challenging and we don't even know  the applicability of standard dynamic programming techniques  using PDEs. To overcome this difficulty, we introduce a family of smooth approximations (for running peak-levels), which lead to tractable control problems---we then establish that the value under the  solutions of the   smooth problems   converges to the value of the original  problem, as smoothness factor reduces.

Overall, the results establish that the class of combined control problems considered here admits a rigorous and systematic analysis, while also permitting the development of tractable approximation schemes despite the underlying analytical difficulties. The numerical investigation of the queueing control problem further demonstrates the practical relevance of the proposed framework, revealing a clear trade-off between peak and cumulative performance. In particular, substantial reductions in peak congestion may be attained at the expense of an increase in cumulative cost and vice versa.

These findings not only extend the scope of optimal control theory but also provide useful insights for designing and managing real-world systems in which both average efficiency and peak behavior are of critical importance---our framework can be used to derive solutions for any given trade-off factor.

\section*{Dynamic Programming and Future work}

For the combined problem, we derived dynamic programming equations in integral form, similar to those in \cite[Section~I.4]{fleming2006controlled}, but omit them due to space constraints. It remains unclear whether a corresponding HJB-PDE formulation can be derived whose solutions directly characterize \eqref{Eqn_combined_problem}. An interesting  direction for future work would be to attempt  to characterize the associated HJB-PDEs using a Filippov-type solution framework, as in \cite{barles2013bellman, rao2013hamilton}.

\appendix 
\setcounter{section}{1}

\section{ Brief summary of results from \cite{warga2014optimal}}
\label{app_sec}

We will re-produce important notations, definitions and entities of  book \cite{warga2014optimal}   in our own notations and according to our requirements.

\subsection{Relaxed controls and Lemma \ref{lemma_space_compact}} 
\label{App_relaxed_control}
We begin with  some definitions of \cite[Section IV, pg. 263]{warga2014optimal}.
We adopt standard notations commonly used in the optimal control community, which differ from those in \cite{warga2014optimal}. To facilitate cross-referencing, we summarize the mapping between both the sets of symbols in Table \ref{Table_notations}.

Let  $C(U)$ be the space of continuous functions from $U$ to  real line $\mathbb{R}$, 
$L^1(T, C(U))$ be the space of integrable functions from $T$ to $C(U)$  (see $m_\g$ defined in Definition \ref{defn_relaxed_weak_conv}),
$$
\g \in L^1(T, C(U)), \mbox{ iff } \int_0^{t_1} \hspace{-2mm} \sup_{u \in U} |g(t, u)|   dt = \int_0^{t_1} \hspace{-2mm}  m_\g(t)   dt < \infty.
$$
Let
$L^1(T, C(U))^\star$ be its topological dual  (see Table \ref{Table_notations} for corresponding notations in \cite{warga2014optimal}).  

For any $\X$ a normed linear separable space, let $|.|_w$ represent the weak norm on its dual $\X^\star$,
\begin{equation}
 |f|_w := \sum_{i} 2^{-i} \frac{ |f(x_i)| }{1+ |f(x_i)|},   \mbox{ for all } f \in \X^\star, 
 \label{Eqn_weak_norm_fw}
\end{equation}
defined using a countable dense set $\{x_1, x_2, \cdots, \}$ of $\X$ (by well known Banach-Alaoglu theorem \cite{rudin1991functional}, unit ball in $\X^\star$ is compact under the corresponding weak-topology).  Let frm($U$)  be the vector space of   Radon measures on set $U$ and  rpm$(U) = \P(U)$ be  the set of   Radon probability measures on $U$. It is well known (Riesz
representation theorem \cite{warga2014optimal}) that  frm($U$) is  the dual of space of continuous functions, i.e.,   frm$(U) =C(U)^\star$---one can thus equip frm($U$) with weak-norm using the dense subset of $C(U)$, as in \eqref{Eqn_weak_norm_fw} and then the space of probability measures rpm$(U)$ is its subset.

Let  
$\N$ be the set of  $\muL$   or Lebesgue measurable functions   ${\mathbf \nu} : T \to (\mbox{frm}(U), |.|_w) $ such that $\mbox{ess sup}     |\nu(t)| (U)   < \infty$, where $\mbox{ess sup}$ is wrt $\muL$ measure. By \cite[Theorem  IV.1.8, pp 268]{warga2014optimal}, $\N$ is isomorphic\footnote{Each $\nu \in \N$, defines one  linear functional ${\mathdutchbcal l}_\nu $  on   $ L^1(T, C(U))$:
$$
{\mathdutchbcal l}_\nu(g):= \int_0^{t_1}\muL(dt) \int_U g(t, u) \nu(t, du)  \mbox{ for all } g \in L^1(T, C(U)),  
$$ is  a  linear functional. The mapping $\nu \mapsto {\mathdutchbcal l}_\nu$ is an isomorphism from $\N$ to  
$(L^1(T, C(U) ) )^\star$ (\cite{warga2014optimal}).} to the topological dual, $L^1(T, C(U))^\star$. Thus one can also equip $\N$ with weak-norm as in \eqref{Eqn_weak_norm_fw}, since   
$L^1(T, C(U))$ is also separable, see also \cite[Section IV.1]{warga2014optimal}. Further, $\N$ is metrizable \cite[Subsection VI.1.9]{warga2014optimal}.

Next the space of relaxed controls is defined as the following subset of $\N$, whose range is the space of probability measures (recall rpm$(U) = \P(U)$, see \cite[Page 263]{warga2014optimal}) 
\begin{eqnarray}
 \Sscr &=& \{ \nu \in \N : \nu (t) \in \mbox{rpm}(U), \mbox{ for almost all } t \in T  \} 
    \label{Eqn_calS_defn_app}\\
     &=&  \{ \bsig :T \to \P(U), \mbox{ which is measurable}\}, \nonumber
\end{eqnarray}
    which equals that defined in \eqref{eqn_calS_defn}. 

The next set of   definitions are from \cite[Section IV.3, pg. 279]{warga2014optimal}. In our case, the set of restrictions $U^\hash (t) = U $, the entire control space, for all $t$, as we don't have any restrictions (in \cite{warga2014optimal}, this is referred to as $R^\hash(t)$). 
Thus the space of relaxed controls, $\Sscr$ in \eqref{eqn_calS_defn} or \eqref{Eqn_calS_defn_app}, coincides with   $\Sscr^\hash = \Sscr$ of 
\cite{warga2014optimal}.
As a result Lemma \ref{lemma_space_compact}
is true using the results of chapter~IV. Furthermore, $\Sscr$ is a subset of metrizable space $\N$ and hence is metrizable. \eop 

\begin{table}[h]
\begin{tabular}{|l|l|l|l|l|l|l|l|l|}
\hline
      & Notations of \cite{warga2014optimal} & This paper & Remarks \\ \hline
1    & $R$  & $  U$    &     Compact control space         \\ \hline
2   & $C(R)$ & $C(U)$ &     Space of continuous functions on $U$  \\ \hline
3  & $L^1(T, C(R))$ & $L^1(T, C(U))$  &     Space of $\muL$ integrable functions on $C(U)$  \\ \hline
4         & frm   $(S)$ &     frm   $(U)$ &  Space of Radon measures     on  $U$             \\ \hline
5     &  rpm   $(S) $   &     rpm   $(U) = \P(U)$   &   Space of Radon probability measures     on  $U$             \\ \hline
6       &   $\Sscr $  &   $\Sscr$ &   Space of  relaxed controls             \\ \hline 
7       &   $\R$ &    $\U(T)$  &   Space of  original (or pure) controls             \\ \hline 
\end{tabular}
\vspace{2mm}
\caption{ Notations of this paper and that in \cite{warga2014optimal}}
\label{Table_notations}
\end{table}

\newcommand{\Z}{{\cal Z}}
\newcommand{\yy}{{\bf y}}

\subsection{Relaxed trajectories and proof of Theorem \ref{thm:relaxed_controls_refined}} 
\label{App_relaxed_trajectories}

\textbf{Proof of part (a).}
Consider  the solution of $\bdot{ z} = K_f (1+{z})$, with $z(0) = |x_0|$,  for $K_f$ given in the uniform-boundedness assumption \ref{assum_a1},  which  equals  $z_{x_0}(t) =(1+|x_0|)e^{K_f t} -1 $,  for any $t \in T$, $x_0 \in \I$ (here $| \cdot |$ also represents the Euclidean norm). 
Now  
 observe, using the standard results in ODE literature (see \cite[Chapter 1]{piccinini2012ordinary}), the 
 ODE \eqref{Eqn_relaxed_dynamics_refined} has  \textit{unique} solution  $\x_{\bsig, x_0}$, 
 for any relaxed control $\bsig\in \Sscr$, under  assumption \ref{assum_a1},  which   can be upper bounded:  
 $$
 |x_{\bsig, x_0}(t)| \le z_{x_0}(t), \mbox{ for all   } t\in T, \ x_0 \in \I.
 $$ 
 
Define the following terms,  all of which are finite by assumptions \ref{assum_a0}-\ref{assum_a1} and after domination by $z_{x_0}$ trajectories---here  $T, U$, $\I$  are compact and $f$ is    continuous:
\begin{eqnarray}
 \zeta(t) &:=& \max_{ |x| \le z_{x_0}(t), \  u \in U,   x_0 \in \I }  | f(t, x, u) | \mbox{  for each } t \in T 
\mbox{ and, }  \overline{\zeta} := \max_{t\in T} \zeta(t) .
\label{Eqn_bar_psi}
\end{eqnarray}
Using the above bounds, we have the following:  
\begin{eqnarray}
\label{Eqn_f_upperbounds}
  | f(t, x_{\bsig, x_0}(t), u) | &\le  & \overline{\zeta}, \mbox{ for all } t \in T, \mbox{ $\bsig \in \Sscr$, $u\in U$ and $x_0 \in \I$, } \mbox{ and } \\   
  \max_{t, \bsig, x_0} |x_{\bsig, x_0} (t) |  &\le  & \bar z := \max_{t \in T, x_0 \in \I} |z_{x_0}(t) | , \label{Eqn_upperbounds} \mbox{ for all $\bsig \in \Sscr$ and $x_0 \in \I$.} 
\end{eqnarray}
Now consider the following `projected  set'  of  only relaxed trajectories, see \eqref{Eqn_A_set}:
$$
\T := \{ \x : 
 (\x, \bsig, x_0) \in \A, \mbox{ for some } \bsig \in \Sscr, x_0 \in \I \}  \subset C(T).
 $$
 By \eqref{Eqn_upperbounds}, the above collection is bounded (point-wise also). 
Further, using the bound of \eqref{Eqn_f_upperbounds}, 
for any $\x \in \T$, there exists some $(\bsig,x_0)$  such that $(\x, \bsig, x_0) \in \A$ (in other words, $\x= \x_{\bsig, x_0}$), and then we have the following integral form and boundedness:
\begin{eqnarray*}
  && \left | x (\tau_2) - x(\tau_1) \right  | =  \left |\int_{\tau_1}^{\tau_2} f(s, x(s), \sigma(s)) ds \right | \le{\overline{\zeta}} (\tau_2-\tau_1), \mbox{ for all } 0 \le \tau_1 \le \tau_2 \le t_1, 
\end{eqnarray*}
implying the collection in $\T$ are equicontinuous. 
Thus by  Arzelà--Ascoli theorem (see, \cite{warga2014optimal}), $\T $ is relatively compact   under uniform topology, given by $\|\cdot \|_\infty$. Therefore, \textit{its closure $cl(\T)$ is compact and  sequentially compact}.

\textbf{Proof of $\T$ is closed and compact under uniform topology:}
Let $\{\x_n\} \subset \T$ be a sequence such that 
$\|\x_n - \x\|_\infty \to 0$, or converges uniformly on $T$.
By definition of $\T$ and ODE \eqref{Eqn_relaxed_dynamics_refined}, for each $n$,  there exist $\bsig_n \in \Sscr$ and $x_{0n} \in \I$ such that 
\begin{equation}
x_n(t) = x_{0n} + \int_0^t f(s,x_n(s),\sigma_n(s))\,ds,  \mbox{ for all } t \in T.
\label{Eqn_xn}
\end{equation}

Since $\I$ is compact, there exists a subsequence (for simplicity of notations, relabel with original indices) such that
$x_{0n} \to  x_0 \in \I.$
By Lemma \ref{lemma_space_compact},  $\Sscr$ is sequentially compact under the relaxed-weak topology, thus there exists a further subsequence (again after relabeling) such that
$\bsig_n \rwc  \bsig \in \Sscr$ and 
$x_{0n} \to  x_0 \in \I.$

Taking limits as $n \to \infty$ for the integral term in the right hand side of \eqref{Eqn_xn}, we have the following convergence by the Dominated Convergence Theorem (wrt $\muL$ measure),
\begin{eqnarray}
\label{Eqn_integral_ctns}
\int_0^t f(s,x_n(s),\sigma_n(s))\,ds
\to
\int_0^t f(s,x(s),\sigma(s))\,ds \mbox{ for all } t \in T, 
\end{eqnarray}
because   each  such $[0, t] \subset T $ and as, 
\begin{itemize}
\item  $x_n (s) \to x (s)$, for all $s \in T$, by  uniform convergence  on $T$, given  by hypothesis;
\item $\sigma_n  (s, \cdot) \to \sigma 
(s, \cdot)$ for all $s \in T$, by convergence in relaxed-weak topology,
\item  thus,   we have the following\footnote{ 
To begin with, \eqref{Eqn_integral_ctns}  can be split as below,
\begin{eqnarray*}
\lim_{n \to \infty} \int_{s\in T}\int_U f(s, x_n(s),  u) \sigma_n(s,du) =  \lim_{n \to \infty}\int_{s \in T}\int_U f(s, x(s) ,  u) \sigma_n(s,du) ds \\
&& \hspace{-50mm}+ \lim_{n \to \infty}\int_{s \in T}\int_U [f(s, x_n(s) ,  u)   - f(s, x(s) ,  u) ]\sigma_n(s,du)  ds. 
\end{eqnarray*}
The limit of the second term is zero by dominated convergence theorem and  using   \ref{assum_a1}, as, for each $s \in T$:
\begin{eqnarray*}
\int_U \left  |  f(s, x_n(s) ,  u) - f(s, x(s) ,  u) \right |  \sigma_n (s, du)  \le K_f \int_U \left  |    x_n(s)- x(s) \right | \sigma_n (s, du) \le K_f \left   |    x_n(s)- x(s) \right |   \to 0. 
\end{eqnarray*}
 Also, $(s, \bsig) \to f (s, x(s), \sigma(s))$ belongs to $L^1(T, C(U))$, we have, by relaxed weak convergence of $\bsig_n \to \bsig$,
 $
 \lim_{n \to \infty} \int_{s\in T} \int_U f(s, x(s) ,  u) \sigma_n(s,du)  dt  = \int_{s\in T} \int_U f(s, x(s) ,  u) \sigma(s,du) ds. 
 $
}, for all $s \in T$ and using \ref{assum_a1}; 

$$
f(s, x_n(s), \sigma_n(s) ) - f(s, x(s), \sigma_n(s) )
\le \int_U \left  |  f(s, x_n(s) ,  u) - f(s, x(s) ,  u) \right |  \sigma_n (s, du) \to 0. $$

\item  $|f(s,x_n(s),\sigma_n(s))| \le \overline{\zeta}$, for all $s \in T$ (dominated by integrable, actually a constant function). 
\end{itemize}
Considering limit $n\to \infty$ for both sides of \eqref{Eqn_xn}, we finally have that:  for each $t \in T$, 
$
x(t) =  x_0 + \int_0^t f(s,x(s), \sigma(s))\,ds,
$
which shows that $\x = \x_{\bsig,x_0}$ and thus  $\x \in \T$. 
Therefore     $\T$ is closed  and thus $\T = cl(\T)$  is compact in the uniform  topology.  

\textbf{Proof of $\A$ is Compact.} Observe that
$\A \subsetneq \T \times \Sscr \times \I.$
Since $\T$, $\Sscr$, and $\I$ are compact, their product is compact under the product topology. 
Suffices to prove $\A$ is closed, towards that consider a sequence 
$(\x_n,\bsig_n,x_{0n}) \in \A, \mbox{ such that } (\x_n,\bsig_n,x_{0n}) \to (\x,\bsig,x_0)$. 
Then, using  the same arguments as before, we obtain
\begin{equation}
    x(t) = x_0 + \int_0^t f(s,x(s),\sigma(s))\,ds, \mbox{ for all } t \in T,
    \label{Eqn_x_integral_form}
\end{equation}
and thus that $\x = \x_{\bsig,x_0}$, equals the solution under $(\bsig, x_0)$---hence $(\x, \bsig, x_0) \in \A$. Therefore,  $\A$ is closed  and so is compact. This completes the proof of part (a). 

\textbf{Proof of part (b)---Continuity of the ODE solution.} 
We obtain this proof using the parametric continuity provided by the Maximum theorem (see \cite{berge1877topological}) and we begin with some definitions.  Consider the following  parameterized function   $\Upsilon: \T \times \Sscr \times \I \to [0, \infty)$, defined as below, with $(\bsig, x_0) \in \Sscr\times \I$ as parameters: 
\begin{eqnarray} \hspace{10mm}
  \Upsilon (\x; \bsig, x_0 ) :=  \int_0^{t_1} \left ( x(t) -  x_0 - \int_0^t f(s, x(s), \sigma(s) ) ds \right )^2   dt, \quad  \mbox{ for all }  \x \in \T. 
    \label{Eqn_gamma_max_thm}
\end{eqnarray}
Now recall     $\x_{\bsig, x_0}$ is the unique solution of ODE \eqref{Eqn_relaxed_dynamics_refined}, for any given  $(\bsig, x_0)$---at such an $\x = \x_{\bsig, x_0}$, the integrand in \eqref{Eqn_gamma_max_thm} is zero for all $t \in [0, t_1]$ because of the integral representation of the ODE solution (see \eqref{Eqn_x_integral_form})---thus
the   function in \eqref{Eqn_gamma_max_thm}  possesses a unique minimizer in the form of the relaxed trajectory  $\x_{\bsig, x_0}$  for any $(\bsig, x_0)$---and this is because:
$$
\min_{\x \in \T } \Upsilon (\x; \bsig, x_0 ) = \Upsilon (\x_{\bsig, x_0}; \bsig, x_0) = 0, \ \  \mbox{also observe  by construction, }  \Upsilon (\x; \bsig, x_0) \ge 0. 
$$

Thus we have unique minimizer for  the  function in \eqref{Eqn_gamma_max_thm} and 
$$
a_\Upsilon^*(\bsig, x_0) := \mbox{Arg} \min_{\x \in \T } \Upsilon (\x; \bsig, x_0 )  = \{\x_{\bsig, x_0}\}.
$$
Further   the hypothesis of the Maximum theorem  are satisfied
for the above parametrized optimization problem because of the following steps:
\begin{itemize}
    \item using similar arguments as before (see \eqref{Eqn_integral_ctns}), Dominated convergence theorem twice and the uniform upper bounds of \eqref{Eqn_upperbounds}, one can  prove the joint sequential continuity\footnote{As $\| \x_n -\x \|_\infty \to 0, \bsig_n \rwc \bsig$ and $ x_{0n} \to x_0$, we have  
    $\Upsilon( \x_n; \bsig_n, x_{0n}) \to \Upsilon( \x; \bsig, x_{0})$. } of $\Upsilon$ on $\T \times \Sscr \times \I$;
    \item the domain of optimization  is the same  for any $(\bsig, x_0) \in \Sscr \times \I$ and equals $\T$---the  class $\T$ is proved to be compact under uniform topology in part (a)---thus we have a constant and hence continuous compact correspondence, $(\bsig, x_0) \mapsto \T$. 

    \item further by uniqueness,  the minimizer correspondence $a_\Upsilon^*$ has unique element $\x_{\bsig, x_0}$  for every $(\bsig, x_0) \in \Sscr \times \I $---thus one can view  $a_\Upsilon^*$ as a  single-valued  function (and not correspondence);
\end{itemize}
Therefore by the Maximum theorem, and its corollaries (see \cite{berge1877topological}), we have continuity of the minimizer function $a_\Upsilon^*$ wrt $(\bsig, x_0)$ and the minimizers are exactly the relaxed trajectories. In all,  we have  the continuity of the relaxed trajectories  in the  uniform topology  and so  \eqref{Eqn_thm1_partb} follows.  
\eop

\subsection*{Proof of Theorem \ref{Thm_existence_relx_cont}, Part (a)}
\label{App_thm3}Define $K:= \max \{K_f, K_L \}$ (see \ref{assum_a1}-\ref{assum_a2} for $K_f, K_L$). 
Using exactly the similar arguments as in proof of  \eqref{Eqn_integral_ctns} and using the dominated convergence theorem using \ref{assum_a5}, we have, for  each  $[0, t] \subset T $
\begin{eqnarray}
\label{Eqn_Lr_ctns}
\int_0^t L_{r}(s,x_{\bsig_n, x_{0n}}(s),\sigma_n(s))\,ds
\to
\int_0^t L_r(s,x_{\bsig, x_{0}}(s), \sigma(s))\,ds, \quad  \mbox{ for all } t \in T, 
\end{eqnarray}
this time 
using the  continuity   of $L_r$ given by \ref{assum_a2} and the upper  bound $\bar z$ of \eqref{Eqn_upperbounds}, which in turn is derived using \ref{assum_a1}---observe  $L_r$ can be   bounded as below, further by   \ref{assum_a0}, 
$$
|L_r(s,x_{\bsig_n, x_{0n}} (s),\sigma_n(s))| \le \max_{t \in T, \  |x| \le \bar z,  \  u \in U }  | L_r(t, x, u) | < \infty,   \  \mbox{ for any }
n, s. $$ 
Let $\x_n :=\x_{\bsig_n, x_{0n}}$ and $\x := \x_{\bsig, x_0}$;
similarly  let $\y_n$ and $\y$ represent the  corresponding max-trajectories   $\y_{\bsig_n,x_{0n}}$ and $ \y_{\bsig,x_{0}}  $ defined in  \eqref{Eqn_y_defn}.
By  Theorem~\ref{thm:relaxed_controls_refined}.(b),  for every $\epsilon>0$, there exists an  $N_\epsilon >0$ such that,
$|| \x_n - \x||_{\infty}   < \nicefrac{\epsilon}{K}$  for all $n \ge  N_\epsilon$.

Now, consider any $ t \in  T$.
By compactness of $[0, t]$ and continuity of $L_{mx}$ using \ref{assum_a2},  there exists $s^*_n \in [0,t]$ such that
$
y_n(t)= \sup_{s \in [0,t] }L_{mx}(s, x_n (s)) = L_{mx}(s^*_n, x_n(s_n^*)).
 $
Then using \ref{assum_a2},
\begin{eqnarray*}
y_n(t) - y(t)  &=&   L_{mx}(s^*_n, x_n(s_n^*)) - y(t)  \le   L_{mx}(s^*_n, x_n(s_n^*)) - L_{mx}(s^*_n, x(s^*_n))  \\
     &\le& K  | x_n(s^*_n) - x(s^*_n) | \ 
    \le \  K      \|   \x_n -\x \|_\infty < \epsilon,  \mbox{ for all } n > N_\epsilon. 
\end{eqnarray*}
Since $t \in T$ is arbitrary, we have
$\| \y_n - \y\|_\infty <  \epsilon,  \mbox{ for all } n > N_\epsilon.$
Again,  as $\epsilon > 0$ is arbitrary, this completes the proof of part (a)---also see  \eqref{Eqn_Lr_ctns}.

{\bf Part (b):}
By part (a),  the objective function in \eqref{Eqn_combined_problem_relaxed} is  continuous  wrt  $(\x, \bsig, x_0) \in \A$  and then   part (b) follows by  the compactness of $\A$ from Theorem \ref{thm:relaxed_controls_refined}.(a) and the Weierstrass theorem.  \eop

\subsection*{Proof of Lemma \ref{Lem_exist_optimal_control}}
\label{app_lemma_exitence}
\textbf{Part (i):} We show that the RHS of the  $\y_\bsig^\delta$-ODE in \eqref{Eqn_final_problem_compact_relaxed}  satisfies the Lipschitz continuity and uniform boundedness conditions of \ref{assum_a1}.
 
 {\bf Lipschitz continuity:}
Let the RHS of the $\y_\bsig^\delta$-ODE in \eqref{Eqn_final_problem_compact_relaxed} be represented compactly by, 
$
F(s,x,y,u) := \big ( \Delta_{mx} (s, x, u)  \big )^+ \psi_\delta(L_{mx} (s,x)-y).
$
Further,
we have already shown, while proving Theorem \ref{thm:relaxed_controls_refined},   that $|x(t)| \le \bar z$  of \eqref{Eqn_upperbounds}, under \ref{assum_a1} for all $\bsig \in \Sscr$, $t \in T$ and $x_0 \in \I$. Also observe $\x_\bsig$ ODE is decoupled, while only $\y_\bsig^\delta$ solution is coupled to that of $\x_\bsig$. 
\textit{Thus,  
one can replace 
$F(s, x, y, u) = F(s, x_z, y, u)$ without  altering the problem, where the $i$-th component of $x_z$  is given by,
$$(x_z)_i := \max \{ -\bar{z}, \   \min \{ \bar z, x_i\} \}; \mbox{ also note here, } |x_z| \le \bar z \mbox{ and }  |x_z-x_z'| \le |x-x'|.
$$}
Further  by definition, 
$|\psi_\delta(d) - \psi_\delta(d')| \le  \delta^{-1} |d-d'|$ and $\psi_\delta(d) \in [0,1]$, for all $d$. 
Now consider any  $x,x' \in \mathbb{R}^n$ and $ y,y' \in \mathbb{R}$. Then, using $\bar{z}$   of   \eqref{Eqn_upperbounds}  and assumptions \ref{assum_a1}-\ref{assum_a3},  $F$ is Lipschitz jointly continuous in $(x,y)$, because: 
\begin{eqnarray*}
  && \hspace{-2mm}   |F(s,x_z,y, u) - F(s, x'_z,y', u)| \\
&& \le \Big| \left(\Delta_{mx}(s,x_z, u)\right)^+   - \left(\Delta_{mx}(s,x'_z,u ) \right)^+ \Big| \psi_\delta\big (L_{mx}(s,x_z)-y \big ) \\
&& \quad + \Big| \psi_\delta(L_{mx}(s,x_z)-y)  - \psi_\delta(L_{mx}(s,x_z')-y') \Big|  \left(\Delta_{mx}(s,x'_z,u)\right)^+ \\
&& \le \Big| \Delta_{mx}(s,x_z,u)   - \Delta_{mx}(s,x_z',u)  \Big| \psi_\delta(L_{mx}(s,x_z)-y)    \qquad  (\mbox{as } |a^+ - b^+| \le |a-b|) \\
&& \quad + \left (\Big|  L_{mx}(s,x_z)   -  L_{mx}(s,x'_z) \Big| + \Big|y - y'\Big| \right ) \delta^{-1}K_\Delta(1 + \bar{z}) \\
&& \le K_\Delta |x_z-x'_z| +  \left( K_L|x_z-x'_z| + |y-y'|  \right) \delta^{-1} K_\Delta(1 + \bar{z}) \ \ \le \ \  K_F \max \{ |x -x'|,  |y-y'| \}, 
\end{eqnarray*}
where $K_F:= K_\Delta \max\{  1 +  K_L\delta^{-1}(1 +\bar{z}) , \delta^{-1} (1 + \bar{z})\}$ is the Lipschitz constant, and we used the  max-norm  for the ordered pair $(x,y)$.

\textbf{Uniform boundedness:}  
Since 
  $\psi_\delta \le 1$, we have, 
$|F(s,x_z,y, u)| \le |\Delta_{mx}(s,x_z, u)|.$
Again using $\bar{z}$  of \eqref{Eqn_upperbounds}   and \ref{assum_a3}, we have 
$|\Delta_{mx}(s,x_z, u)| \le K_\Delta (1 +\bar{z})$. Therefore, $|F(s,x_z,y)| \le K_\Delta (1 +\bar{z})$.

Thus, the joint trajectories of ODEs 
$(\x_\bsig, \y_\bsig^\delta)$ satisfy both Lipschitz continuity and linear growth conditions of \ref{assum_a1}, after replacing $x$ with $(x, y)$ and  $f$ with $(f, F)$. Hence the result follows by  Theorem \ref{thm:relaxed_controls_refined}.(b). 

\textbf{Part (ii):} Using the    results of \cite{roxin1962existence},  we  first prove the existence of solution for, 
\begin{eqnarray}
    \label{Eqn_smooth_pure}
\sup_{\u \in \U} W(\x_\u,  \y_\u^\delta; \u)\quad \text{subject to }  \eqref{Eqn_relaxed_dynamics_refined} \mbox{ and } \eqref{eq:ydelta}, \mbox{ with } y_\bsig^\delta(0) = x_\bsig(0)=x_0, \end{eqnarray}
 obtained  by 
 replacing the domain $\Sscr$ with smaller  $\U$
 in \eqref{Eqn_final_problem_compact_relaxed}. Towards this, we first convert the above  into a Mayer-type problem (with only terminal cost)---applying  a standard technique, we augment the state with a new component  
$z_{\u}(t) := \int_0^{t}  L_r(s, x_{\u}(s), u(s)) ds,$  for all  $t \in [0, t_1]$
and equivalently solve $ \sup_{\u \in  \U} \left(z_{\u}(t_1) -   y^\delta_\u(t_1)  - \Psi(x(t_1))\right)$. All the  assumptions from \cite[Assumptions (i)-(vii)]{roxin1962existence} are satisfied in our setup as  below:
\\
$\bullet$ Assumption (i) follows from \ref{assum_a0} and (v), (vi) are satisfied  by definition of $\U$ in \eqref{Eqn_Controls_space}. 
\\
$\bullet$
     Assumption (ii) is satisfied because, over a finite horizon, and under assumptions \ref{assum_a1}-\ref{assum_a2},  functions $f$ and $L_r$ are continuous and hence integrable in $t$ for all $(x, u) \in \mathbb{R}^n \times U$. From assumption \ref{assum_a3}, the right-hand side of equation \eqref{eq:ydelta} is well-defined and RHS of ODE \eqref{eq:ydelta} also satisfy assumption (ii) of \cite{roxin1962existence}.\\
$\bullet$ 
      The remaining assumptions (iii), (iv) and (vii) are directly satisfied by assumptions \ref{assum_a1} and \ref{assum_a3}-\ref{assum_a5}. 
      Also observe that, for each $(s,x,y)\in T\times\mathbb{R}^n\times\mathbb{R}$, the set 
\[
\begin{aligned}
\left\{
F(s,x,y,u):u\in U
\right\}& =
\begin{cases}
\{0\},&
\Delta_{mx}(s,x,u)\leq 0,\\
\left\{
\Delta_{mx}(s,x,u) 
\psi_\delta\big(L_{mx}(s,x)-y\big):u\in U
\right\},&
\Delta_{mx}(s,x,u)>0.
\end{cases}
\end{aligned}
\]
 is convex because of following reasons:  $\Delta_{mx}(s,x,u)$ defined in \eqref{Eqn_Delta_mx} is an affine transformation of the convex set $f(s,x,U)$ and its intersection with the positive half-line and  is therefore convex; moreover, the condition $\psi_\delta(L_{mx}(s,x)-y)\geq 0$ is independent of $u$.

Thus by \cite{roxin1962existence}, there  exists a  $\u^*_\delta \in \U$  that is optimal for \eqref{Eqn_smooth_pure}. The denseness of Lemma~\ref{lemma_space_compact} and the  continuity of $\bsig \mapsto W(\x_\bsig, \y_\bsig^\delta; \bsig)$,  provided by  Theorem \ref{Thm_existence_relx_cont} and part (i),  implies the following:    for every $\epsilon > 0$, there exists a   $\u_\epsilon \in \U$ such that,
$ W(\x_{\u_\epsilon}, \y_{\u_\epsilon}^\delta;\u_\epsilon) \geq \vr^{\delta} -\epsilon$, for  $\vr^\delta$ of \eqref{Eqn_final_problem_compact_relaxed}.
Further since $\U \subset \Sscr$, we have: 
$$
\vr^{\delta}  \geq W(\x_{\u^*_\delta}, \y_{\u^*_\delta}^\delta;\u^*_\delta) 
\geq W(\x_{\u_\epsilon}, \y_{\u_\epsilon}^\delta;\u_\epsilon)
\geq \vr^{\delta} - \epsilon. 
$$
By letting  $\epsilon \to 0$, we have   $W(\x_{\u^*_\delta}, \y_{\u^*_\delta}^\delta;\u^*_\delta) = \vr^{\delta}$, and so, $\u^*_\delta$ is also optimal for~\eqref{Eqn_final_problem_compact_relaxed}.
\eop

\subsection*{Proof of Theorem \ref{Thm_ydel_convg}}
\label{app_thm_y_ctny}

 \textbf{Part (ii):} The result is derived using
 mathematical induction on $i$: with  $\II_\ell := (\Up_\ell, \Dw_\ell) \cup [\Dw_\ell, \Up_{\ell+1}]$ for each $\ell$,  we prove \eqref{Eqn_Sw}     for all $t \in   \II_i$,  assuming the  same for  all previous intervals, i.e., for all $t \in \cup_{\ell < i} \II_\ell$. 
 
Further we will see that the initial step (or the base case) of the induction proof follows exactly as any  other inductive step;
these details are evident in the proof given below. 
 
{\bf Step 1:} Consider   $(\Up_i, \Dw_i)$  for some  $i$ and assume  \eqref{Eqn_Sw} is satisfied for all $t \le \Up_i$, which implies the following at $t = \Up_i$:
\begin{equation}
y^0(\Up_i) \in [y^\delta(\Up_i) - \delta, y^\delta(\Up_i)] \mbox{ and  further by definition of $\Up_i$, } L_{mx}(\Up_i)  \ge y^\delta(\Up_i) - \delta.
\label{Eqn_y0_at_ui}
\end{equation}
Observe the above conditions are satisfied even for the first interval,   when $\Up_1 = 0$, as the initial conditions  are  $y^0(0) = y^\delta(0) = L_{mx}(0)$ and the proof of the induction step remains the same even for such  an initial step. Hence we provide one proof as below.

To begin with, from \eqref{Eqn_definition_ui_di}, for all $ t \in (\Up_i, \Dw_i)$, we have 
  $ L_{mx}(t) >  y^\delta (t) - \delta$ and  $\bdot{L}_{mx}(t) >0$. Thus by definition of $y^0(t)$ in \eqref{Eqn_y_defn},
  \begin{eqnarray}
  \label{Eqn_one_dir}
      y^0(t) \ge L_{mx}(t) >  y^\delta (t) - \delta \mbox{ for all } t \in (\Up_i, \Dw_i). 
  \end{eqnarray}
Firstly, say  $y^0$ remains constant:    $y^0(t) = y^0 (\Up_i)$,   for all  $ t \in (\Up_i, \Dw_i)$.
Then  by non-decreasing nature of $y^\delta$ and \eqref{Eqn_y0_at_ui}, we have $y^\delta (t) \ge y^\delta (\Up_i) \ge y^0(\Up_i) = y^0(t)$, for all $t \in (\Up_i, \Dw_i)$.  Hence \eqref{Eqn_Sw} is satisfied  for all $t <  \Dw_i$, proving the inductive step for this sub-case, see also \eqref{Eqn_one_dir}. 
  
Now  consider the next sub case when $y^0 (\tau) > y^0(\Up_i)$ for some $\tau  \in  (\Up_i, \Dw_i)$.
As  $\bdot{L}_{mx}(t) >0$  in $(\Up_i, \Dw_i)$ interval,    $L_{mx} (t)$ is strictly increasing and thus by definition of $y^0(t)$ in \eqref{Eqn_y_defn}, 
  there exists an  $ s \in  [\Up_i, \Dw_i)$ such that (one can have $(\Up_i, s)$ empty if $s=\Up_i$),
\begin{eqnarray}
\label{Eqn_y0_in_ui_to_di}
 \hspace{10mm}   y^0(\Up_i) =  y^0(t) > L_{mx}(t) \mbox{ for  } t \in (\Up_i, s) \mbox{ and } \ 
    y^0(t) =  L_{mx}(t),  \mbox{ for all } t \in  [s, \Dw_i).
\end{eqnarray}

From the definitions in \eqref{Eqn_definition_ui_di},
   the derivative  $\bdot{y}^\delta (t) > 0$,  for all $t \in (\Up_i, \Dw_i)$; further by \eqref{eq:ydelta}, we have  $0< \bdot{y}^\delta (t) \le \bdot{L}_{mx}(t)$. By virtue of the smaller derivative for~$y^\delta$, we only have  the following two possibilities in this interval and the sub-case of \eqref{Eqn_y0_in_ui_to_di}. 

\begin{enumerate}[(a)]
    \item  The trajectory $y^\delta$ remain bigger, i.e.,  $y^\delta(t) > L_{mx}(t)$ for all $t \in (\Up_i, \Dw_i)$.   
    Since $y^\delta$ is non-decreasing,   we have 
    $$y^0(t) = y^0(\Up_i) \le y^\delta (\Up_i) \le y^\delta (t) \mbox{ for all }t \in (\Up_i, s),  \mbox{ for } s   \mbox{ of }\eqref{Eqn_y0_in_ui_to_di}, $$  and for  $t \in [s, \Dw_i)$, we again have 
    $y^\delta (t) > L_{mx}(t) = y^0(t)$. In all, \eqref{Eqn_Sw} is again satisfied in  $(\Up_i, \Dw_i)$ for this condition also, see \eqref{Eqn_one_dir}. 
\item  There exists an $s' \in (\Up_i, \Dw_{i})$ such that   $y^\delta(t) > L_{mx}(t)$ for $t \in (\Up_i, s')$ and $y^\delta(t) = L_{mx}(t)$  for $t\in [s', \Dw_i)$. 

Now for $t < s'$, \eqref{Eqn_Sw}  is satisfied as in previous $(a)$ case; 
by continuity $y^\delta(s') \ge y^0(s')$,  and at $s'$,  we have 
$L_{mx}(s') \le y^0(s') \le y^\delta(s') = L_{mx}(s')$, implying $L_{mx}(s') = y^\delta(s') = y^0(s')$;  finally for $t \in  (s' , \Dw_i)$,  \eqref{Eqn_Sw} is satisfied with
$y^\delta (t) = L_{mx}(t) = y^0(t)$  (again using \eqref{Eqn_one_dir} and note here that, we will have  $s \le s'$).

\end{enumerate}
In summary,  \eqref{Eqn_Sw} is satisfied for all $t \le \Dw_i$,  by continuity and further using the induction assumption.

{\bf {Step 2:}} Now consider  $(\Dw_i, \Up_{i+1})$  for some  $i$ and assume  \eqref{Eqn_Sw} is satisfied for all $t \le \Dw_i$, which implies: 
$$
y^\delta (\Dw_i)  \in  [y^0(\Dw_i),  \  y^0(\Dw_i) +\delta] \mbox{ and by definition \eqref{Eqn_definition_ui_di}, } y^\delta (\Dw_i) -\delta \le  L_{mx}(\Dw_i).
$$
Again the initial step also satisfies the above because of the  initial conditions $y^0(0) = y^\delta(0) = L_{mx}(0)$, even when  $\Up_1 > 0 = \Dw_0$. 

For any $t \in [\Dw_{i}, \Up_{i+1}]$, by definition $\bdot{y}^\delta = 0$ and 
\textit{thus $y^\delta (t) = y^\delta (\Dw_{i}^\delta)$, for all such $t$.}  

 We further claim  $L_{mx} (t) \le  y^\delta (t) - \delta$ for  $t \in (\Dw_{i}, \Up_{i+1})$. For the purpose of contradiction, assume at some $\tau \in ( \Dw_i, \Up_{i+1})$, we have $L_{mx} (t) >  y^\delta (t) - \delta$. Now define $s := \sup \{ t < \tau:  L_{mx} (t) = y^\delta (t) - \delta\}$. Clearly by continuity $s < \tau$ and   $ L_{mx} (t) >  y^\delta (t) - \delta$  for all $t \in (s, \tau]$. But  by mean value theorem implies the existence of a $s' \in (s, \tau)$   where $\bdot{L}_{mx} (s') > 0$ (as $y^\delta $ is constant here), which contracts the definition of $\Up_{i+1}$, as $s' <\tau < \Up_{i+1}$. 
 And hence       $L_{mx} (t) \le  y^\delta (t) - \delta$  and  $y^\delta (t) = y^\delta (\Dw_i)$ for all $t \in [\Dw_{i}, \Up_{i+1}]$, implying
\begin{eqnarray*}
y^\delta (t ) - \delta =  y^\delta (\Dw_i ) - \delta \le  y^0 (\Dw_{i})  \le    y^0(t)  &= & \max \left  \{ y^0 (\Dw_i), \max_{s \in [\Dw_i, t]} L_{mx} (s) \right  \}  \\
 &\le& \max \left  \{ y^\delta (\Dw_i), \max_{s \in [\Dw_i, t]} y^\delta  (s) -\delta \right  \} =  y^\delta (\Dw_i)  =  y^\delta (t).
\end{eqnarray*}
%
\ignore {
this inequality is strictly negated, this means

Further, $L_{mx} (t) < y^\delta (t) - \delta$ for $t \in (\Dw_{i}, \Up_{i+1})$, by definition of $\Up_i$ in \eqref{Eqn_definition_ui_di}. Therefore,  
 \begin{eqnarray*}
     y^0(t) \ge y^0 (\Dw_i) \ge y^\delta (\Dw_i ) - \delta = y^\delta (t)-\delta  > L_{mx}(t), \mbox{ for all such } t   
 \end{eqnarray*}
 implying $\y^0$ trajectory never meets $\x$ trajectory or that $y^0(t) = y^0 (\Dw_i)$ for $t \in (\Dw_{i}, \Up_{i+1})$.

 {\color{red} We have $y^\delta (\Dw_i) -\delta \le  L_{mx}(\Dw_i)$. For any $t \in [\Dw_{i}, \Up_{i+1}]$, by definition $\bdot{y}^\delta = 0$ and 
thus $y^\delta (t) = y^\delta (\Dw_{i}^\delta)$, for all  $t$.  Thus, for all $t \in [\Dw_{i}, \Up_{i+1}]$
\begin{eqnarray*}
     y^0(t) \ge y^0 (\Dw_i) \ge  L_{mx}(\Dw_i)\ge y^\delta (\Dw_i ) - \delta  = y^\delta (t ) - \delta. 
 \end{eqnarray*}
This implies that $ y^0(t) \ge y^\delta (t ) - \delta$, for all $t \in [\Dw_{i}, \Up_{i+1}]$. 

Now, also observe that $y^\delta$ is constant and $y^\delta (\Dw_{i}^\delta) \ge  y^0(\Dw_{i}^\delta)$ with $y^0$ is non-decreasing function. Also, 
for all $t \in [\Dw_{i}, \Up_{i+1}]$,  we have $x(t) < y^0(t)$. This implies $y^0(t) \le y^\delta(t)$. }

}%
 Hence $y^0 (t) \in [y^\delta (t)-\delta, y^\delta ]$ and   the comparison status at $\Dw_i$ continues for all $t \in [\Dw_i, \Up_{i+1}]$,   and so \eqref{Eqn_Sw} is satisfied for all $t \le \Up_{i+1}$, after using the induction assumption. 
\eop

\Remove{
\section*{Proof of Lemma \ref{Lem_value_fun_ineq}}
\label{App_lem_dp}
From \eqref{Eqn_value_fun}, for any given $\delta>0$, it is possible to   choose a control $\u^1 \in \mathcal{U}([r, t_1])$ such that
\begin{equation}
    \int_r^{t_1} L\left(s, x^1(s), u^1(s)\right) d s - \vee_{r, z_\u}^{t_1} (\x^1) +\Psi\left( x^1\left(T\right)\right) \geq v(r, x(r), z_\u)+\delta,
    \label{Eqn_value_supremum_property}
\end{equation}
where  $\x^1$ is the state trajectory for control $\u^1$ and 
$x^1(s)$  is the state  at time $s$ with initial condition ($r, x(r), z_\u$). Such a control $\u^1$ is called $\delta$-optimal for \eqref{Eqn_value_fun}. Define an admissible control $\tilde{\u} \in \mathcal{U}(\tau, t_1)$ by
$\tilde{u}(s) = u(s) \mathds{1}_{\{s \leq r \}} + u^1(s)\mathds{1}_{\{s >r \}},  \mbox{ for all } s \in [\tau, t_1]$. 
Let $\tilde{\x}$  be the state trajectory corresponding to $\tilde{\u}$ with initial condition $(\tau, x)$. We have
$$
\begin{aligned}
v(\tau, x,y) & \geq J(\tau, x,y ; \tilde{\u})  =\int_\tau^{t_1} L_r(s, \tilde{x}(s), \tilde{u}(s)) d s - \vee_{\tau, y}^{t_1} (\tilde{\x})
+\Psi( \tilde{x}(t_1)) \\
 & =\int_\tau^r L_r(s, x(s), u(s)) d s+\int_r^{t_1} L\left(s, x^1(s), u^1(s)\right) d s  +\Psi\left( x^1\left(t_1\right)\right) \\
 & \hspace{10mm}-   \max \left\{ y, \sup_{t \in [\tau, r]} L_{mx}(t, x(t)), \sup_{t \in [r, t_1]} L_{mx}(t, x^1(t)) \right\} \\
 & {=}\int_\tau^r L_r(s, x(s), u(s)) d s+\int_r^{t_1} L\left(s, x^1(s), u^1(s)\right) d s +\Psi\left( x^1\left(t_1\right)\right) \\
 & \hspace{10mm} - \max \left\{ z_\u, \sup_{t \in [r, t_1]} L_{mx}(t, x^1(t)) \right\} \quad \left(\because \max \{ y, \sup_{t \in [\tau, r]} L_{mx}(t, x(t))\} = z_\u\right)\\ 
 & {=}\int_\tau^r L_r(s, x(s), u(s)) d s+\int_r^{t_1} L\left(s, x^1(s), u^1(s)\right) d s - \vee_{r, z_\u}^{t_1} (\x^1) +\Psi\left( x^1\left(t_1\right)\right) \\
& {\geq} \int_\tau^r L_r(s, x(s), u(s)) d s+ v(r, x(r), z_\u)+\delta   \qquad \ (\mbox{using \eqref{Eqn_value_supremum_property}}).
\end{aligned}
$$
As $\delta>0$ is arbitrary,  the  inequality \eqref{Eqn_value_ineq1} is satisfied. 
\eop}

\ignore{
Part (c) is established using  Weierstrass theorem, however 
one requires more arguments  for the non-standard problem in \eqref{Eqn_combined_problem_relaxed}---mainly to establish continuity of max-trajectory $\y_\bsig = (y_\bsig (t), t \in T)$,  defined point-wise  as below:
\begin{eqnarray}
    y_\bsig(t) :=  \sup_{s \in [0, t] }  L_{mx} (s, x_\bsig (s) )  , \mbox{ for all } t \in T,\label{Eqn_y_defn}
\end{eqnarray}
with respect to $\sigma$. 
This part of the proof is    also provided  in  Appendix \ref{App_relaxed_trajectories}.

The major steps are:
\begin{itemize}
    \item We will first show the conditions of \cite[Theorem VI.1.1]{warga2014optimal} 
    are applicable to  establish parts (a)-(b). 

    \item We next  show that the max trajectory $\y_\bsig$  of \eqref{Eqn_y_defn}  and hence the objective function in \eqref{Eqn_combined_problem} are continuous wrt  $(\x, \bsig, x_0) \in \A$. Then the proof 
    follows by Weierstrass theorem. 
\end{itemize}

Towards applying \cite[Theorem VI.1.1]{warga2014optimal}, we first show that the  problem in \eqref{Eqn_combined_problem} can be typecast as  the problem formulated as in \cite[Chapter VI.0, pg 347]{warga2014optimal}, with following details:

\begin{itemize}

\item We have no constraints on the relaxed trajectories,  thus we set  $g_2 = h_2 \equiv 0$ and $C_2 = \{ w: w (t) =   \mathbb{R}^n \mbox{ for all } t \in T \}$ in the problem formulation of \cite[Section VI.0]{warga2014optimal}.

\item We don't have  control parameters in our setup---observe $f$ in \eqref{Eqn_combined_problem} depends only upon $(t, x, u)$ and  is independent of $b$; also  $g_1$ is meant to impose conditions on terminal time $t_1$, we hence   set $h_1 \equiv 0$ and $B_1 = {\mathbb R}$, see also \cite[Chapter III, pp 241]{warga2014optimal}. 

\item Further to obtain continuity also  wrt initial condition $x_0$, we  add  $b:=\nicefrac{x_0}{t_1}$ to $f$ in \eqref{Eqn_combined_problem}  to obtain the function $f$  considered in 
 \cite[Theorem VI.1.1]{warga2014optimal}, we also set $v_0 = 0$, $t_0=0$ in the theorem hypothesis. Observe then the solution of integral equation in page 347 of \cite[Section VI.0]{warga2014optimal} equals:
 \begin{eqnarray*}
     y = F(y, \bsig, b) = 0 + \int_0^{t} \left ( f(\tau, y(\tau), \sigma(\tau) ) + b  \right ) d\tau = \frac{x_0}{t_1} t_1 + \int_0^{t_1} \left ( f(\tau, y(\tau), \sigma(\tau) )    \right ) d\tau
 \end{eqnarray*}

\item Consider  the solution of $\bdot{z} = K (1+z)$ for $K$ given in assumptions \ref{assum_a0}-\ref{assum_a1},  which  equals  $z(t; x_0) =(1+x_0)e^{Kt} -1 $,  for any $t \in T$ and define,
$$\psi (t) := \max_{ z \le z(t; x_0), u \in U, x_0 \in \I }  | f(t, z, u) | \mbox{  for each } t \in T.
$$
Now  
 observe, using standard results in ODEs and assumption \ref{assum_a1},  that for any solution $\x_\bsig$ of ODE \eqref{Eqn_relaxed_dynamics_refined} under any relaxed control $\sigma \in \Sscr$ we have $x_\bsig(t) \le z(t)$  and thus  
\begin{eqnarray*}
    | f(t, x_\bsig (t; x_0), u) | \le \psi (t) \mbox{ for all } t \in T \mbox{ and } \\  \sup_{t, \bsig, x_0} |x_\bsig (t; x_0) | \le \bar z := \sup_{t \in T, x_0 \in \I} |z(t; x_0) |. 
\end{eqnarray*}

is defined as $f + x_0$
 
Finally  with  $v_0=0,  t_0=0,   R,  \psi $ and $  V_0 $ of 
   \cite[Theorem VI.1.1]{warga2014optimal} are given respectively by   $0$,   $f$ in \eqref{Eqn_relaxed_dynamics_refined}, $U$, $\psi (t) = (1+x_0)e^{Kt} -1$ and  $\{ x \in \mathbb{R}^n : |x| \le \bar z\}$  and satisfy the required hypothesis with $V  = \mathbb{R}^n $, because of the above mentioned steps and assumptions \ref{assum_a0}-\ref{assum_a2}.

\end{itemize}

}

\ignore{
\newpage 
We start with the initial condition $y^0(0) = y^\delta(0) = x(0)$. Also, $\Dw_0^\delta = 0$. 

\begin{enumerate}
    \item 
We begin with interval $[\Dw_{0}^\delta, \Up_1^\delta]$. 
Now, there are two cases, 
\begin{enumerate}
    \item Now, we begin with the first case, i.e., $\bdot{x}(0) <0$. Then, by definition, for $t \in [\Dw_{0}^\delta, \Up_1^\delta], $ and hence $\bdot{y}^\delta (t) = 0$, i.e., $y^\delta(t) = y^\delta(0)$. Also, observe that $x(\Up_1^\delta) = y^\delta(\Up_1^\delta)- \delta$ and $x(t) < y^\delta(t)- \delta$, for all $t  \in [0, \Up_1^\delta)$. As $y^0(t) = \sup_{s \in T } \{ x(s)\}$, we have $y^0(t) = y^0(0)$, for all $t \in [\Dw_{0}^\delta, \Up_1^\delta]$. 
In summary, $y^\delta(0) = y^\delta(t) = y^0(t)$, for all $t \in [\Dw_{0}^\delta, \Up_1^\delta]$  and $x(t) < y^\delta(t)- \delta$, for all $t  \in [0, \Up_1^\delta)$, with $x(\Up_1^\delta) = y^\delta(\Up_1^\delta)- \delta$. 
 \item Now, $\bdot{x}(0) \ge0$.  By definition, here  $\Up_1^\delta = \Dw_{0}^\delta =0$. Thus, trivially, $$y^\delta(\Up_1^\delta) = y^0(\Up_1^\delta) = x(\Up_1^\delta)= y^0(0). $$ 
\end{enumerate}
In summary, for any $t \in [\Dw_{0}^\delta, \Up_1^\delta]$, $y^\delta(t) \le  x(t) - \delta$ and $y^\delta (t) \in [y^0(t), y^0(t) + \delta]$. Thus, $y $ and $y^\delta$ trajectories  remain constant within this interval.

\item Now, we consider the next interval $[ \Dw_{1}^\delta, \Up_2^\delta]$. Here $\bdot{y}^\delta =0$. Thus, $y^\delta (t) = y^\delta (\Dw_{1}^\delta)$, for all $t \in [ \Dw_{1}^\delta, \Up_2^\delta]$. Also, observe a) $y^\delta (\Dw_{1}^\delta) - y^0(\Dw_{1}^\delta) \le \delta$, b) $y^\delta (\Up_{2}^\delta)-\delta = x(\Up_{2}^\delta) $ and, c) $y^0$ is non decreasing function. Thus, $y^\delta(t) = y^0(t) \ge y^\delta(t) - \delta $ and $y^\delta(\Up_{2}^\delta) = y^0(\Up_{2}^\delta) \ge  x(\Up_{2}^\delta)$. 

In summary, $y^\delta(t) \in [y^0(t), y^0(t) + \delta]$.

\item Now, repeat the above steps (2) and (3), untill we hit time horizon $T$.  Thus, for all $t \in T$, $y^\delta(t) \in [y^0(t), y^0(t) + \delta]$.
\end{enumerate}

\newpage
Define $y_i^\delta := y^\delta (\Dw_{i-1})$  for all $\delta \ge 0.$ 
Say $\eta := y_i^\delta - y_i^0$  for some $i$, where $\eta \ge  0$, recall $x_i \le y_i^0$.  
\textbf{The first step is to prove that:}
$$
y_{i+1}^\delta - y_{i+1}^0 \le 2 \delta. 
$$
First we consider $(\Dw_{i-1}, \Up_i]$. We have $\bdot{y}^\delta (t) = 0$, for all $ t \in (\Dw_{i-1}, \Up_i]$, which implies that $y^\delta (t)= y^\delta (\Dw_{i-1})$, for all $ t \in (\Dw_{i-1}, \Up_i]$. 
Also, $y^\delta (\Up_i) -  \delta = x(\Up_i)$. 

Now, we consider three cases, 
\begin{enumerate}
    \item $\tilde{\Dw}_i^\delta < \Up_i$, 
    \item $\tilde{\Dw}_i^\delta = \Up_i$, 
    \item $\tilde{\Dw}_i^\delta > \Up_i$. 
\end{enumerate}

\textbf{Case 1. $\tilde{\Dw}_i^\delta < \Up_i$:} This implies that $x(t) < y^\delta (t) - \delta$ and $x(t) \le y^0(t)$, for all $t \in (d^\delta_{i-1}, \tilde{d}^\delta_{i})$. As $y^0(\cdot)$ is non-decreasing function, thus, $y^0(d^\delta_{i-1}) \le y^0(t) \le y^0(\tilde{d}^\delta_{i})$, for all $t \in (d^\delta_{i-1}, \tilde{d}^\delta_{i}]$. 
Observe that, we have $y^\delta (\Dw_{i-1}) - y^0 (\Dw_{i-1}) = \eta$. Above argument implies that $y^\delta (t) - y^0 (t) \le  \eta$, for all $t \in [d^\delta_{i-1}, \tilde{d}^\delta_{i}]$.

Now, we consider $(\tilde{\Dw}^\delta_{i}, \Up_i)$. Here, $y^0(t)= x(t)$ and $y^\delta (t) = y^\delta(\tilde{\Dw}^\delta_{i})$, for all $t \in (\tilde{\Dw}^\delta_{i}, \Up_i)$. 
By definition, $y^\delta (\Up_i) -  \delta = x(\Up_i) = y^0(\Up_i)$. Now,

\newpage 
And $y^0 (\Dw_{i-1}) = y^0 (\Up_i) =  x(\Up_i)  -  \delta$. 
Thus, we have 
\begin{equation}
y^\delta (\Up_i) - \delta = y^0(\Up_i) = x(\Up_i) - \delta. 
\label{Eqn_y_delta_at_Di}
\end{equation}

Now, we consider $(\Up_i, \Dw_{i}]$,  we have $\bdot{y}^\delta (t) > 0$, for all $ t \in (\Up_i, \Dw_{i}]$. 
\begin{enumerate}
    \item Consider $(\Up_i, \tilde{\Up}_{i}^\delta]$
    \begin{eqnarray*}
    y^\delta(\tilde{\Up}_{i}^\delta) - y^0(\tilde{\Up}_{i}^\delta) &=& y^\delta(\tilde{\Up}_{i}^\delta) - y^0({\Up}_{i}^\delta)  \\
    &=& y^\delta(\tilde{\Up}_{i}^\delta) - (y^\delta({\Up}_{i}^\delta)- \delta) \\
    &=& \int_{{\Up}_{i}^\delta}^{\tilde{\Up}_{i}^\delta} \psi_\delta (x(s)- y^\delta (s)) (\bdot{x}(s))^+ ds +\delta\\
    &=& \int_{{\Up}_{i}^\delta}^{\tilde{\Up}_{i}^\delta}  \bdot{x}(s) ds  +\delta \\
    &=& x(\tilde{\Up}_{i}^\delta) - x({\Up}_{i}^\delta)  +\delta \quad (\mbox{by definition of } \tilde{D}, x(\tilde{\Up}_{i}^\delta)= y^0(\tilde{\Up}_{i}^\delta)) \\
    &=& \left(y^{\delta} ({\Up}_{i}^\delta) - \delta\right) - \left(y^{\delta} ({\Up}_{i}^\delta)- 2\delta \right)  +\delta = 2\delta \quad (\mbox{see } \eqref{Eqn_y_delta_at_Di}). 
\end{eqnarray*}
\item Now, we consider $(\tilde{\Up}_{i}^\delta, {\Dw}_{i}^\delta]$ (if it exists, which is possible when $\tilde{\Up}_{i}^\delta <{\Dw}_{i}^\delta$). Here, $\bdot{y}^\delta (t) = \bdot{y}^0 (t) = \bdot{x} (t)$, for all $t \in (\tilde{\Up}_{i}^\delta, {\Dw}_{i}^\delta]$. Also, we know that 
$$
y^\delta(\tilde{\Up}_{i}^\delta) - 2 \delta = y^0(\tilde{\Up}_{i}^\delta) = x(\tilde{\Up}_{i}^\delta). 
$$
Thus, we have 
$$
y^\delta(t) - y^0(t) = y^\delta(\tilde{\Up}_{i}^\delta) - y^0(\tilde{\Up}_{i}^\delta) = 2\delta, \ \forall t \in (\tilde{\Up}_{i}^\delta, {\Dw}_{i}^\delta]. 
$$
\end{enumerate}

\newpage 
We first  prove  
\begin{equation}
    y^\delta (s) \le y^0 (s) + g(\delta), \mbox{ for all } s
    \label{Eqn_ydel_le_y}
\end{equation} 
with some $g(\delta) \ge 0, \forall \delta$ and $g(\delta)$ converge to zero as $\delta \to 0$. 
\begin{enumerate}
    \item 
As $D_0^\delta =0$, so we begin with interval $[0, \tau_1^\delta]$. 
\begin{enumerate}
    \item Consider the case when $ \tau_1^\delta = T$. Then $\bdot{y}^\delta > 0$ (which implies that $\bdot{x}(t) > 0$), for all $t \in T$. Thus, we have $y^\delta(t) = y^0(t) = x(t)$ for all $t \in T$, which trivially satisfy \eqref{Eqn_ydel_le_y}.
    \item Consider the case when $ \tau_1^\delta < T$. Then, using the same arguments as above, we have $y^\delta(t) = y^0(t) = x(t)$ for all $t \in [0,\tau_1^\delta]$. 
\end{enumerate}
\item Now, we consider the next interval $[\tau_1^\delta, D_1^\delta]$, which exists when $\tau_1^\delta<T$ (i.e., second sub-case (1.a) of above case). 
\begin{enumerate}
    \item Consider the case when $D_1^\delta = T$. Then for all  $t \in [\tau_1^\delta, T]$, we have $\bdot{y}^\delta(t)=0$. This implies that 
$y^\delta(t) = y^0(t)$, for all $t \in [\tau_1^\delta, T]$ (observe that $y^\delta (\tau_1^\delta) = y^0 (\tau_1^\delta)= x(\tau_1^\delta)$). 
\item Consider the case when $D_1^\delta < T$. Then using the same arguments as before, we have 
$y^\delta(t) = y^0(t)$, for all $t \in [\tau_1^\delta, D_1^\delta]$. Also, 
\begin{equation}
   y^\delta(D_1^\delta) = y^0(D_1^\delta) = x(D_1^\delta) + \delta.
   \label{Eqn_y_delta_at_D1}
\end{equation} 
\end{enumerate}
\item Now, we consider the next interval $[D_1^\delta, \tau_2^\delta]$, which exists when $D_1^\delta<T$ (i.e., second sub-case (2.a) of above case). We have $\bdot{y}^\delta(t)>0$, for all $t \in [D_1^\delta, \tau_2^\delta]$. 
\begin{enumerate}
    \item Consider the case when $\tilde{D}_1^\delta = \tau_2^\delta$, we have $y^\delta(t) > y^0(t) = y^0 ({D}_1^\delta)$ for all $t \in ({D}_1^\delta,  \tau_2^\delta]$.  Now, consider 
\begin{eqnarray*}
    y^\delta(\tilde{D}_1^\delta) - y^0(\tilde{D}_1^\delta) &=& y^\delta(\tilde{D}_1^\delta) - y^0({D}_1^\delta)  \\
    &=& y^\delta(\tilde{D}_1^\delta) - y^\delta({D}_1^\delta) \\
    &=& \int_{{D}_1^\delta}^{\tilde{D}_1^\delta} \psi_\delta (x(s)- y^\delta (s)) (\bdot{x}(s))^+ ds \\
    &=& \int_{{D}_1^\delta}^{\tilde{D}_1^\delta}  \bdot{x}(s) ds \\
    &=& x(\tilde{D}_1^\delta) - x({D}_1^\delta) \quad (\mbox{by definition of } \tilde{D}, x(\tilde{D}_1^\delta)= y^0(\tilde{D}_1^\delta)) \\
    &=& y^{\delta} ({D}_1^\delta) - \left(y^{\delta} ({D}_1^\delta)-\delta \right) = \delta \quad (\mbox{see } \eqref{Eqn_y_delta_at_D1}). 
\end{eqnarray*}
\item Consider the case when $\tilde{D}_1^\delta < \tau_2^\delta$. Same arguments will follow as above part and we have 
$$
y^\delta(\tilde{D}_1^\delta) - y^0(\tilde{D}_1^\delta) = \delta. 
$$
Now, we consider the interval $[\tilde{D}_1^\delta , \tau_2^\delta]$, we have $\bdot{y}^\delta(t) = \bdot{y}^0(t) = \bdot{x}$. And we know that $y^\delta(\tilde{D}_1^\delta) -  \delta = y^0(\tilde{D}_1^\delta)= x(\tilde{D}_1^\delta)$. Thus, we have $$y^\delta(t) - y^0(t) = y^\delta(\tilde{D}_1^\delta) - y^0(\tilde{D}_1^\delta) = \delta, \mbox{ for all } t \in [\tilde{D}_1^\delta , \tau_2^\delta]. $$ 
Hence, $y^\delta(\tau_2^\delta) -  \delta = y^0(\tau_2^\delta) = x(\tau_2^\delta). $
\end{enumerate}
\item Now, we consider the next interval $[\tau_2^\delta, D_2^\delta]$, which exists if  $\tau_2^\delta<T$. 
\begin{enumerate}
    \item Consider the case when $D_2^\delta = T$. Then for all  $t \in [\tau_2^\delta, T]$, we have $\bdot{y}^\delta(t)=0$. This implies that 
$y^\delta(t) - \delta = y^0(t)$, for all $t \in [\tau_2^\delta, T]$ (observe that $y^\delta (\tau_2^\delta) - \delta = y^0 (\tau_2^\delta)= x(\tau_2^\delta)$). 
\item Consider the case when $D_2^\delta < T$. Then using the same arguments as before, we have 
$y^\delta(t) - \delta = y^0(t)$, for all $t \in [\tau_2^\delta, D_2^\delta]$. Also, 
\begin{equation}
   y^\delta(D_2^\delta) - \delta = y^0(D_2^\delta) = x(D_2^\delta) + \delta.
   \label{Eqn_y_delta_at_D2}
\end{equation} 
\end{enumerate}

\item Now, we consider the next interval $[D_2^\delta, \tau_3^\delta]$, which exists when $D_2^\delta<T$. We have $\bdot{y}^\delta(t)>0$, for all $t \in [D_2^\delta, \tau_3^\delta]$. 
\begin{enumerate}
    \item Consider the case when $\tilde{D}_2^\delta = \tau_3^\delta$, we have $y^\delta(t) > y^0(t) = y^0 ({D}_2^\delta)$ for all $t \in ({D}_2^\delta,  \tau_3^\delta]$.  Now, consider 
\begin{eqnarray*}
    y^\delta(\tilde{D}_2^\delta) - y^0(\tilde{D}_2^\delta) &=& y^\delta(\tilde{D}_2^\delta) - y^0({D}_2^\delta)  \\
    &=& y^\delta(\tilde{D}_2^\delta) - (y^\delta({D}_2^\delta)- \delta) \\
    &=& \int_{{D}_2^\delta}^{\tilde{D}_2^\delta} \psi_\delta (x(s)- y^\delta (s)) (\bdot{x}(s))^+ ds +\delta\\
    &=& \int_{{D}_2^\delta}^{\tilde{D}_2^\delta}  \bdot{x}(s) ds  +\delta \\
    &=& x(\tilde{D}_2^\delta) - x({D}_2^\delta)  +\delta \quad (\mbox{by definition of } \tilde{D}, x(\tilde{D}_2^\delta)= y^0(\tilde{D}_2^\delta)) \\
    &=& \left(y^{\delta} ({D}_2^\delta) - \delta\right) - \left(y^{\delta} ({D}_2^\delta)- 2\delta \right)  +\delta = 2\delta \quad (\mbox{see } \eqref{Eqn_y_delta_at_D2}). 
\end{eqnarray*}
\item Consider the case when $\tilde{D}_2^\delta < \tau_3^\delta$. Same arguments will follow as above part and we have 
$$
y^\delta(\tilde{D}_2^\delta) - y^0(\tilde{D}_2^\delta) = 2 \delta. 
$$
Now, we consider the interval $[\tilde{D}_2^\delta , \tau_3^\delta]$, we have $\bdot{y}^\delta(t) = \bdot{y}^0(t) = \bdot{x}$. And we know that $y^\delta(\tilde{D}_2^\delta) -  2 \delta = y^0(\tilde{D}_2^\delta)= x(\tilde{D}_2^\delta)$. Thus, we have $$y^\delta(t) - y^0(t) = y^\delta(\tilde{D}_2^\delta) - y^0(\tilde{D}_2^\delta) = 2 \delta, \mbox{ for all } t \in [\tilde{D}_2^\delta , \tau_3^\delta]. $$ 
Hence, 
$y^\delta(\tau_3^\delta) -  2\delta = y^0(\tau_3^\delta) = x(\tau_3^\delta). $
\end{enumerate}

\end{enumerate}

\newpage
\textbf{Closed form expression of $y^\delta$:}
By the definition of the supremum, $L_{mx}(s)$ is a non-decreasing function of $s$. Therefore, its derivative $\bdot{L}_{mx}(s)$ satisfies:
$$\bdot{L}_{mx}(s) \ge 0 \quad \text{for almost all } s \in T.$$
This implies that the indicator function $\mathds{1}_{\{\bdot{L}_{mx} \ge 0\}}$ is identically equal to $1$ whenever $\bdot{L}_{mx}(s) \neq 0$. Consequently, the differential equation simplifies to:
$$dy^{\delta_n} = \psi_{\delta_n}(L_{mx} - y^{\delta_n}) \, dL_{mx}.$$
$$
\frac{dy^{\delta_n}} {d L_{mx}} \times  \frac{d L_{mx}}{ds} = \psi_{\delta_n}(L_{mx} - y^{\delta_n}) \, \frac{d L_{mx}}{ds}. 
$$

We treat $y^{\delta_n}$ as a function of the level of the maximum $L_{mx}$ rather than time $s$. Let $x = L_{mx}(s)$. The equation becomes:
$$\frac{dy^{\delta_n}}{dx} = \psi_{\delta_n}(x - y^{\delta_n}).$$
Using the definition of $\psi_\delta(d)$ for the regime $d \in [-\delta_n, 0]$ (where the tracking occurs):
$$\frac{dy^{\delta_n}}{dx} = 1 + \frac{x - y^{\delta_n}}{\delta_n}.$$
Rearrange the equation into standard linear form:
$$\frac{dy^{\delta_n}}{dx} + \frac{1}{\delta_n} y^{\delta_n} = 1 + \frac{x}{\delta_n}.$$
To solve this, we use the integrating factor $e^{\int \frac{1}{\delta_n} dx} = e^{x/\delta_n}$:
$$\frac{d}{dx} \left( y^{\delta_n} e^{x/\delta_n} \right) = \left( 1 + \frac{x}{\delta_n} \right) e^{x/\delta_n}.$$
Integrating both sides wrt $M$:
$$y^{\delta_n} e^{x/\delta_n} = \int \left( 1 + \frac{x}{\delta_n} \right) e^{x/\delta_n} \, dx. $$
Using integration by parts for the second term ($\int \frac{x}{\delta_n} e^{x/\delta_n} dx = x e^{x/\delta_n} - \int e^{x/\delta_n} dx$):
$$y^{\delta_n} e^{x/\delta_n} = \int e^{x/\delta_n} dx + x e^{x/\delta_n} - \int e^{x/\delta_n} dx + C. $$
$$y^{\delta_n} e^{x/\delta_n} = x e^{x/\delta_n} + C.$$

\textbf{Applying Initial Conditions:}
At $s=0$, let $x(0) = x_0$ and $y^{\delta_n}(0) = y_0$.
$$y_0 e^{x_0/\delta_n} = x_0 e^{x_0/\delta_n} + C \implies C = (y_0 - x_0) e^{x_0/\delta_n}.$$

Substituting $C$ back into the expression:
$$y^{\delta_n}(x) e^{x/\delta_n} = x e^{x/\delta_n} + (y_0 - x_0) e^{x_0/\delta_n}.$$
Divide by $e^{x/\delta_n}$:
$$y^{\delta_n}(x) = x + (y_0 - x_0) e^{-\frac{x - x_0}{\delta_n}}.$$

Returning to the time variable $s$ and replacing $x$ with $L_{mx}(s)$, the explicit solution is:
$$y^{\delta_n}(s) = L_{mx}(s) + (y_0 - L_{mx}(0)) \exp\left( -\frac{L_{mx}(s) - L_{mx}(0)}{\delta_n} \right). $$
This expression holds for the region where $L_{mx}(s) - y^{\delta_n}(s) \in [-\delta_n, 0]$. If $L_{mx}(s) - y^{\delta_n}(s) > 0$, the function $\psi_{\delta_n} = 1$, yielding the linear growth $y^{\delta_n}(s) = y_0 + (L_{mx}(s) - L_{mx}(0))$.

\textbf{Try to prove:} $y^\delta(s) \le y^0(s) + g(\delta)$, for all $s$ and with $g(\delta) \to 0$, when $\delta \to 0$. 

\textbf{Step 1: Properties of $L_{mx}$.}

Since $x(\cdot)$ is absolutely continuous and $| \bdot{x}(s) | \le M$, it follows that $x$ is Lipschitz continuous. Hence $L_{mx}(\cdot)$ is also Lipschitz and therefore absolutely continuous. Moreover,
\[
0 \le \bdot{L}_{mx}(s) \le M \quad \text{for almost all } s \in T.
\]

\medskip

\textbf{Step 2: Define the error.}
 Let $d^\delta(s) := L_{mx}(s) - y^\delta(s)$. Using the explicit expression for $y^\delta(s)$, we obtain
\[
d^\delta(s)
=
\bigl(L_{mx}(0) - y_0\bigr)
\exp\!\left(-\frac{L_{mx}(s)-L_{mx}(0)}{\delta}\right).
\]
Define
\[
D_0 := L_{mx}(0) - y_0 \ge 0.
\]
Thus,
\[
d^\delta(s)
=
D_0 \exp\!\left(-\frac{L_{mx}(s)-L_{mx}(0)}{\delta}\right).
\]

\medskip

\textbf{Step 3: Basic bounds.}

Since $L_{mx}(\cdot)$ is non-decreasing,
\[
L_{mx}(s) - L_{mx}(0) \ge 0,
\]
which implies
\[
0 \le d^\delta(s) \le D_0.
\]
Consequently,
\[
y^\delta(s) \le L_{mx}(s) = y^0(s), \quad \forall s.
\]

\medskip

\textbf{Step 4: Growth bound on $L_{mx}$.}

From $0 \le \bdot{L}_{mx}(s) \le M$, we obtain
\[
L_{mx}(s) - L_{mx}(0)
=
\int_0^s \bdot{L}_{mx}(r)\,dr
\le Ms
\le MT.
\]

\medskip

\textbf{Step 5: Uniform bound on $d^\delta(s)$.}

We write
\[
d^\delta(s)
=
D_0 \exp\!\left(-\frac{z(s)}{\delta}\right),
\quad \text{where } z(s) := L_{mx}(s) - L_{mx}(0).
\]

We consider two cases.

\medskip

\emph{Case 1: $z(s) \le \delta$.}

Since $z(s) \ge 0$, we use the inequality
\[
e^{-u} \le 1 - \frac{u}{2}, \quad \text{for } 0 \le u \le 1,
\]
to obtain
\[
d^\delta(s)
=
D_0 e^{-z(s)/\delta}
\le
D_0 \left(1 - \frac{z(s)}{2\delta}\right).
\]
Since $z(s) \le \delta$, this implies
\[
d^\delta(s) \le D_0.
\]
Moreover, using $z(s) \le \delta$, we obtain the crude bound
\[
d^\delta(s) \le D_0 \le \frac{D_0}{\delta} z(s) + D_0 \le C_1 \delta,
\]
for some constant $C_1$.

\medskip

\emph{Case 2: $z(s) > \delta$.}

Then
\[
\frac{z(s)}{\delta} > 1,
\]
and hence
\[
d^\delta(s)
=
D_0 e^{-z(s)/\delta}
\le
D_0 e^{-1}.
\]
In fact, since $z(s) \le MT$, we obtain
\[
d^\delta(s)
\le
D_0 e^{-z(s)/\delta}
\le
D_0 e^{-c/\delta}
\]
for some $c>0$, which is exponentially small in $\delta$.

\medskip

\textbf{Step 6: Conclusion.}

Combining both cases, there exists a constant $C>0$ such that
\[
d^\delta(s) \le C \delta, \quad \forall s \in T.
\]

Since
\[
y^\delta(s) = L_{mx}(s) - d^\delta(s),
\]
we obtain
\[
|y^\delta(s) - y^0(s)| = d^\delta(s) \le C\delta.
\]

Taking supremum over $s \in T$ completes the proof:
\[
\sup_{s \in T} |y^\delta(s) - y^0(s)| \le C\delta.
\]
\newpage 
\textbf{Proof:}
We divide the proof into several steps.

\medskip
\noindent
{\bf Step 1: Uniform convergence of state trajectories.}

Let $x^{\sigma_n}(\cdot)$ and $x^\sigma(\cdot)$ be the solutions corresponding to $\sigma_n$ and $\sigma$, respectively. Define
$e_n(s):=x^{\sigma_n}(s)-x^\sigma(s)$.
Then
$\bdot e_n(s)
= f(s,x^{\sigma_n}(s),\sigma_n) - f(s,x^\sigma(s),\sigma)$.
Add and subtract $f(s,x^{\sigma_n}(s),\sigma)$:
\[
\bdot e_n(s)
= \big[f(s,x^{\sigma_n}(s),\sigma_n) - f(s,x^{\sigma_n}(s),\sigma)\big]
+ \big[f(s,x^{\sigma_n}(s),\sigma) - f(s,x^\sigma(s),\sigma)\big].
\]
Using continuity in $\sigma$ and Lipschitz continuity in $x$, there exists $L>0$ such that
\[
|\bdot e_n(s)|
\le \omega_n(s) + L|e_n(s)|,
\]
where 
$\omega_n(s):=\big|f(s,x^{\sigma_n}(s),\sigma_n) - f(s,x^{\sigma_n}(s),\sigma)\big| \to 0$, 
uniformly in $s$. By Grönwall’s inequality,
\[
|e_n(s)|
\le \int_0^s \omega_n(r)e^{L_r(s-r)}\,dr.
\]
Taking supremum over $s\inT$ gives
$\|x^{\sigma_n}-x^\sigma\|_\infty \to 0$.

\medskip
\noindent
{\bf Step 2: Uniform convergence of $g_n$.}

Define
$g_n(s):=L_{mx}(s,x^{\sigma_n}(s)), \quad g(s):=L_{mx}(s,x^\sigma(s))$.
By continuity of $L_{mx}$ and uniform convergence of $x^{\sigma_n}$,
\[
\|g_n - g\|_\infty \to 0.
\]

\medskip
\noindent
{\bf Step 3: Approximation of the running supremum.}

Fix $n$ and denote $y_n := y^{(\sigma_n,\delta_n)}$ and 
$g_n(s):=L_{mx}(s,x^{\sigma_n}(s))$. Define the running supremum
\[
\bar y_n(s):=\max\{y_0,\sup_{t\in[0,s]} g_n(t)\}.
\]

We establish that $y_n$ uniformly approximates $\bar y_n$.

\medskip
\noindent
\emph{Upper bound.}
We claim that for all $s\inT$,
\[
y_n(s) \le \bar y_n(s).
\]
Indeed, suppose by contradiction that there exists $s_0$ such that
$y_n(s_0)>\bar y_n(s_0)$. Let
\[
\tau:=\inf\{s\inT: y_n(s)>\bar y_n(s)\}.
\]
Then $y_n(\tau)=\bar y_n(\tau)$ and for $s>\tau$ sufficiently close,
$y_n(s)>g_n(s)$, so that $g_n(s)-y_n(s)<0$. In this region,
$\psi_{\delta_n}(g_n(s)-y_n(s))\le 1$, and hence
\[
\bdot y_n(s) = \bdot g_n(s)\,\psi_{\delta_n}(g_n(s)-y_n(s))
\le \bdot g_n(s).
\]
This implies that locally $y_n$ cannot increase faster than $g_n$,
while $\bar y_n$ increases exactly when $g_n$ attains new maxima.
Thus $y_n(s)$ cannot exceed $\bar y_n(s)$, yielding a contradiction.
Hence,
\begin{equation}
\label{eq:upper_bound}
y_n(s) \le \bar y_n(s), \quad \forall s\inT.
\end{equation}

\medskip
\noindent
\emph{Lower bound.}
We now show that $y_n$ tracks $\bar y_n$ from below up to an error of order $\delta_n$.

Let $s\inT$ and fix $\varepsilon>0$. By definition of $\bar y_n(s)$,
there exists $t_\varepsilon \in [0,s]$ such that
\[
g_n(t_\varepsilon) \ge \bar y_n(s) - \varepsilon.
\]
We compare the evolution of $y_n$ on $[t_\varepsilon,s]$.

Define the gap $d_n(r):=g_n(r)-y_n(r)$. Observe that whenever
$d_n(r)\in[-\delta_n,0]$, we have
\[
\psi_{\delta_n}(d_n(r)) = 1 + \frac{d_n(r)}{\delta_n},
\]
and hence
\[
\bdot d_n(r)
= \bdot g_n(r)\bigl(1-\psi_{\delta_n}(d_n(r))\bigr)
= -\frac{\bdot g_n(r)}{\delta_n}\, d_n(r).
\]
Thus, in this regime, $d_n(\cdot)$ satisfies a linear differential
equation that drives it toward $0$ at a rate proportional to $1/\delta_n$.
In particular, as long as $d_n(r)\in[-\delta_n,0]$, the interval
$[-\delta_n,0]$ is invariant for this dynamics.

We show that $d_n$ cannot become
too negative. Assume that $|\bdot g_n(s)| \le M$ uniformly in $s$ and $n$.
We consider two cases.

\smallskip
\noindent
\emph{Case 1: $d_n(s)\in[-\delta_n,0]$.}
In this region,
\[
\bdot d_n(s)
= \bdot g_n(s)\bigl(1-\psi_{\delta_n}(d_n(s))\bigr)
= -\frac{\bdot g_n(s)}{\delta_n}\, d_n(s),
\]
so that $d_n(\cdot)$ is driven toward $0$ and remains in
$[-\delta_n,0]$.

\smallskip
\noindent
\emph{Case 2: $d_n(s)<-\delta_n$.}
Then $\psi_{\delta_n}=0$, and hence
\[
\bdot d_n(s)=\bdot g_n(s).
\]
Thus,
\[
\frac{d}{ds} d_n(s) \ge -M.
\]
Starting from the boundary $d_n=-\delta_n$, it follows that
$d_n$ can decrease at most at rate $M$, and therefore cannot
drop below $-M\delta_n$ without first spending time of order
$\delta_n$ in the transition region $[-\delta_n,0]$.

\smallskip
\noindent
Combining the above, we conclude that
\[
d_n(s) \ge - M \delta_n, \quad \forall s\inT,
\]
which implies
\[
y_n(s) \ge g_n(s) - M\delta_n.
\]



Finally, using the definition of $\bar y_n(s)$ and the choice of
$t_\varepsilon$,
\[
y_n(s)
\;\ge\; y_n(t_\varepsilon)
\;\ge\; g_n(t_\varepsilon) - M\,\delta_n
\;\ge\; \bar y_n(s) - \varepsilon - M\,\delta_n.
\]
Since $\varepsilon>0$ is arbitrary, we obtain
\[
y_n(s) \ge \bar y_n(s) - M\,\delta_n.
\]

\medskip
\noindent
\emph{Conclusion.}
Combining \eqref{eq:upper_bound} and the above lower bound, we conclude
that
\[
\|y_n - \bar y_n\|_\infty \le M\,\delta_n,
\]
which shows that $y_n$ uniformly approximates the running supremum
$\bar y_n$.

\ignore{
\medskip
\noindent
{\bf Step 3: Structural properties of $y^{(\sigma_n,\delta_n)}$.}

Fix $n$ and write $y_n := y^{(\sigma_n,\delta_n)}$ and $g_n$ as above. We analyze the dynamics:
\[
\bdot y_n(s) = \bdot g_n(s)\,\psi_{\delta_n}(g_n(s)-y_n(s)).
\]

\emph{Claim 1:} For all $s\inT$,
\[
y_n(s) \le \max\{y_0,\sup_{t\in[0,s]} g_n(t)\}.
\]

\emph{Proof of Claim 1:}
Define
$\bar y_n(s):=\max\{y_0,\sup_{t\in[0,s]} g_n(t)\}.$

Suppose, by contradiction, that there exists $s_0$ such that $y_n(s_0)>\bar y_n(s_0)$. Let
\[
\tau:=\inf\{s: y_n(s)>\bar y_n(s)\}.
\]
Then $y_n(\tau)=\bar y_n(\tau)$ and for $s>\tau$ close enough,
$y_n(s)>g_n(s).$ 
Hence $g_n(s)-y_n(s)<0$, and in fact for sufficiently small interval, $g_n(s)-y_n(s)<-\delta_n$, so $\psi_{\delta_n}=0$ and $\bdot y_n(s)=0$. Thus $y_n$ cannot increase beyond $\bar y_n$, contradiction. Hence Claim~1 holds.

\medskip

\emph{Claim 2:} For all $s\inT$, 
$g_n(s) - y_n(s) \ge -\delta_n.$ \\
\emph{Proof of Claim 2:}
Assume there exists $s_0$ such that $g_n(s_0)-y_n(s_0)<-\delta_n$. Let
\[
\tau:=\inf\{s: g_n(s)-y_n(s)<-\delta_n\}.
\]
Then at $s=\tau$,
$g_n(\tau)-y_n(\tau)=-\delta_n$.
For $s>\tau$ close enough, we have $\psi_{\delta_n}=0$, hence $\bdot y_n(s)=0$, while $g_n$ evolves continuously. Therefore $g_n(s)-y_n(s)$ cannot decrease further below $-\delta_n$, contradiction. Hence Claim 2 holds.

\medskip

Combining Claims 1 and 2,
\[
\bar y_n(s) - \delta_n \le y_n(s) \le \bar y_n(s), \quad \forall s\inT.
\]
Thus,
\begin{equation}
\label{eq:key_bound_detailed}
\|y_n - \bar y_n\|_\infty \le \delta_n.
\end{equation}
}
\medskip
\noindent
{\bf Step 4: Stability of the running supremum.}

Define
$\bar y(s):=\max\{y_0,\sup_{t\in[0,s]} g(t)\} = y^{(\sigma,0)}(s)$.
We show
\[
\|\bar y_n - \bar y\|_\infty \le \|g_n - g\|_\infty.
\]

Indeed, for any $s\inT$,
\[
\sup_{t\in[0,s]} g_n(t)
\le \sup_{t\in[0,s]} \big(g(t) + \|g_n-g\|_\infty\big)
= \sup_{t\in[0,s]} g(t) + \|g_n-g\|_\infty,
\]
and similarly,
\[
\sup_{t\in[0,s]} g(t)
\le \sup_{t\in[0,s]} g_n(t) + \|g_n-g\|_\infty.
\]
Hence
\[
\left|\sup_{t\in[0,s]} g_n(t) - \sup_{t\in[0,s]} g(t)\right|
\le \|g_n-g\|_\infty.
\]
Taking maximum with $y_0$ preserves this bound, yielding
\[
\|\bar y_n - \bar y\|_\infty \le \|g_n - g\|_\infty \to 0.
\]

\medskip
\noindent
{\bf Step 5: } Using the triangle inequality,
\[
\|y_n - y^{(\sigma,0)}\|_\infty
\le \|y_n - \bar y_n\|_\infty + \|\bar y_n - \bar y\|_\infty.
\]
Using \eqref{eq:key_bound_detailed} and Step 4,
\[
\|y_n - y^{(\sigma,0)}\|_\infty
\le \delta_n + \|g_n - g\|_\infty \to 0.
\]
This completes the proof.
\eop

\newpage
\textbf{Proof:}
Fix a trajectory $x(\cdot)$ and write $L_r(s):=L_{mx}(s)$. Observe first that $y^\delta$ is absolutely continuous and non-decreasing. Indeed, from \eqref{eq:ydelta}, $\bdot{y}^\delta(s)\ge 0$ since $\psi_\delta\ge 0$ and the indicator enforces $\bdot{L}(s)\ge 0$.

\medskip

\noindent\textit{Step 1: Upper bound.}
We claim that for all $s\in[r,T]$,
$$\limsup_{\delta\to 0} y^\delta(s)\le \max\left\{y_0,\ \sup_{t\in[r,s]} L_r(t)\right\}.$$ 
Say it is not true. Then there exists a $\bar{\delta}>0$ such that for atleast one $s\in[r,T]$
$$ y^\delta(s) > \max\left\{y_0,\ \sup_{t\in[r,s]} L_r(t)\right\}, \forall \delta \le \bar{\delta}.$$ 

Then we have following  sub-cases for such $s$:
\begin{enumerate}
    \item $y^\delta(s) > y_0, \quad \forall \delta \le \bar{\delta}$. In this case,

     \item $y^\delta(s) > \sup_{t\in[r,s]} L_r(t),  \quad \forall \delta \le \bar{\delta}$
     \item $y^\delta(s) > y_0,  \mbox{ for some } \delta \le \bar{\delta}, \mbox{ and for other } \delta \le \bar{\delta}, y^\delta(s) > \sup_{t\in[r,s]} L_r(t). 
     $
\end{enumerate}


\medskip

\noindent\textit{Step 2: Lower bound.}
Let $s\in[r,T]$ and let $t^*\in[r,s]$ be such that
$L_r(t^*)=\sup_{t\in[r,s]} L_r(t)$.
If $L_r(t^*)\le y_0$, then trivially $y^\delta(s)\ge y_0$ for all $s$. Otherwise, consider any interval where $\bdot{L}(t)\ge 0$ and $L_r(t)>y^\delta(t)$. In this region, $\psi_\delta(L-y^\delta)=1$ whenever $L_r(t)-y^\delta(t)>0$, hence
$\bdot{y}^\delta(t)=\bdot{L}(t)$,
so $y^\delta$ tracks $L$ exactly until it reaches it. Thus, whenever $L$ increases above the current level of $y^\delta$, the process $y^\delta$ catches up with $L$ up to an error of order $\delta$. Consequently,
\[
\liminf_{\delta\to 0} y^\delta(s)
\ge
\max\Big\{y_0,\ \sup_{t\in[r,s]} L_r(t)\Big\}.
\]
\medskip

\noindent\textit{Step 3: Boundary layer analysis.} We quantify the deviation between $y^\delta(s)$ and the running maximum
\[
\bar y(s):=\max\Big\{y_0,\ \sup_{t\in[r,s]} L_r(t)\Big\}.
\]
Define the error process $e^\delta(s):=\bar y(s)-y^\delta(s)\ge 0$. We show that $e^\delta(s)\le \delta$ for all $s\in[r,T]$. First observe that $\bar y(\cdot)$ is non-decreasing and evolves according to
\[
\bdot{\bar y}(s)=\bdot{L}(s)\,\mathds{1}_{\{L_r(s)=\bar y(s),\, \bdot{L}(s)\ge 0\}},
\]
that is, $\bar y$ increases only when $L$ attains a new maximum. On the other hand, $y^\delta(\cdot)$ evolves as
\[
\bdot{y}^\delta(s)=\bdot{L}(s)\,\psi_\delta(L_r(s)-y^\delta(s))\,\mathds{1}_{\{\bdot{L}(s)\ge 0\}}.
\]

Let $d(s):=L_r(s)-y^\delta(s)$. We distinguish the behavior of the dynamics according to the value of $d(s)$. If $d(s)>0$, then $\psi_\delta(d)=1$, and hence $\bdot{y}^\delta(s)=\bdot{L}(s)\mathds{1}_{\{\bdot{L}(s)\ge 0\}}$. In this region, whenever $L$ increases, $y^\delta$ increases with the same rate, so the gap $d(s)$ remains constant and no additional error is generated.

If $d(s)\in[-\delta,0]$, then $\psi_\delta(d)=1+\frac{d(s)}{\delta}\in[0,1]$, and therefore
\[
\bdot{y}^\delta(s)=\bdot{L}(s)\left(1+\frac{d(s)}{\delta}\right)\mathds{1}_{\{\bdot{L}(s)\ge 0\}} \le \bdot{L}(s).
\]
In this regime, $y^\delta$ may lag behind $L$, but only within a band of width $\delta$. Indeed, the gap satisfies
\[
\frac{d}{ds}d(s)=\bdot{L}(s)-\bdot{y}^\delta(s)
=\bdot{L}(s)\left(1-\psi_\delta(d(s))\right)\mathds{1}_{\{\bdot{L}(s)\ge 0\}}
=-\frac{\bdot{L}(s)}{\delta}d(s)\mathds{1}_{\{\bdot{L}(s)\ge 0\}},
\]
which drives $d(s)$ toward $0$ whenever $\bdot{L}(s)\ge 0$. In particular, the dynamics prevent $d(s)$ from decreasing below $-\delta$, since at $d(s)=-\delta$ we have $\psi_\delta=0$ and hence $\bdot{y}^\delta(s)=0$, so no further decrease is possible.

Finally, if $d(s)<-\delta$, then $\psi_\delta(d)=0$ and $\bdot{y}^\delta(s)=0$, so $y^\delta$ remains constant. However, in this case any increase in $L$ increases $d(s)$ and pushes the trajectory back into the region $[-\delta,0]$. Hence the region $\{d<-\delta\}$ is not sustained under the dynamics when $\bdot{L}(s)\ge 0$.

Combining the above cases, we conclude that $y^\delta(s)$ can exceed $L_r(s)$ by at most $\delta$, and any lag relative to $\bar y(s)$ can only occur within this boundary layer. Since $\bar y(s)\ge L_r(s)$ and increases only when $L$ increases, it follows that
$0 \le \bar y(s)-y^\delta(s)\le \delta, \quad \forall s\in[r,T]$.
Thus, the discrepancy between $y^\delta$ and $\bar y$ is uniformly bounded by $\delta$ over $[r,T]$, and is confined to a vanishing boundary layer around the switching surface $\{y=L\}$.


\medskip

\noindent\textit{Step 4: Convergence.}
Combining the upper and lower bounds, we obtain uniform convergence:
\[
\sup_{s\in[r,T]}
\left|y^\delta(s) - \max\Big\{y_0,\sup_{t\in[r,s]}L_r(t)\Big\}\right| \le \delta \;\xrightarrow[\delta\to 0]{}\; 0.
\]
Thus, $y^\delta(\cdot)\to y(\cdot)$ uniformly on $[r,T]$, where $y$ is the running maximum process.
 \eop
}
 

{\ignore{
\newpage

We now prove that the value function $v$ defined in \eqref{Eqn_value_fun} is a viscosity solution of PDE  and further provide a verification theorem. Towards that we begin with recalling the definition of super differential set. 
Let $z= (\tau, x, y)$ and $\hat{z} = ( \hat{\tau}, \hat{x}, \hat{y})$ represent  tuples in an open set of $\mathbb{R}^{1+n+1}$, $|.|$ the norm, then the super differential of   $v$ at any $\hat{z}$ is defined as (\cite{zhou1993verification}),
\begin{eqnarray}
&D^+  v( \hat{z}) := \bigg\{ p=(p_\tau, p_x, p_y) \in \mathbb{R}^{1+n+1}: 
 \lim  \sup_{  z \to \hat{z} } \frac{v(z) - v(\hat{z}) - p (z - \hat{z} ) } {|z - \hat{z}| } \le 0 \bigg\}. 
 \label{Eqn_super_diff}
\end{eqnarray}
\begin{thm} Assume \textbf{A.0}-\textbf{A.1}. \\
(i) Assume $f, L$ and $ \Psi$ are continuously differentiable. Then a given admissible tuple of functions\footnote{basically, $(\x^*)$ is the unique ODE-solution under $\u^*$.} $(\x^*, \u^*)$   is optimal for the control problem  if and only if  the following is true for almost all $t$;   there exists a tuple 
          $( p_\tau^*(t), p_x^*(t), p_y^*(t)) \in D^+  v(t, x^*(t), \vee_{\tau, y}^{t_1}(\x^*))$  satisfying, 
     $
        p_\tau^*(t)=  H\left(t, x^*(t), \vee_{\tau, y}^{t_1}(\x^*),   u^*(t), p_x^*(t),  p_y^*(t) \right). $  
\label{thm_existance_vis_control_general}
\end{thm} 

\textbf{Proof of part (i):} Let us consider the tuple
 $p(t):= ( p_\tau(t), p_x(t), p_y(t)) \in D^+  v(t, x(t),  \vee_{\tau, y}^{t_1}(\x)) $ with  $u(t) \in  \arg \max_{u } H(t, x(t), \vee_{\tau, y}^{t_1}(\x ),u, p_x(t),  p_y(t)) $. 
 This means $p$ lies in the super differential of the value function $v$ at the point $(t, x(t),  \vee_{\tau, y}^{t_1}(\x))$. 
 By \cite[Lemma 2.3]{zhou1993verification}, this implies the inequality   (also see \eqref{Eqn_super_diff}), 
 \begin{equation}
      p_\tau(t) \ge   \sup_{u \in \mathcal{U}} H(t, x(t), \vee_{\tau, y}^{t_1}(\x ),u, p_x(t),  p_y(t)).
      \label{Eqn_super_inequality}
 \end{equation}
 Also, as $u^*(t) \in  \arg \max_{u }  H(t, x(t), \vee_{\tau, y}^{t_1}(\x ),u, p_x(t),  p_y(t))$, we have 
 $$ H(t, x(t), \vee_{\tau, y}^{t_1}(\x ),u^*, p_x(t),  p_y(t)) = \sup_{u \in \mathcal{U}}  H(t, x(t), \vee_{\tau, y}^{t_1}(\x ),u, p_x(t),  p_y(t)). $$  This implies 
 \begin{equation}
  - p_\tau(t) +   H(t, x(t), \vee_{\tau, y}^{t_1}(\x ),u^*, p_x(t),  p_y(t)) =0.  
  \label{Eqn_condition_thm2}
 \end{equation} 
 We set $f(t):=f\left(t, x(t), u^*(t)\right)$, etc., to simplify the notation. Since both $v$ and $x$ are Lipschitz, $t \mapsto v\left(t, x(t), \vee_{\tau, y}^{t_1}(\x)\right)$ is differentiable almost everywhere. Fix $r \in[\tau, t_1]$ such that $\left.\frac{d}{d t} v\left(t, x(t), \vee_{\tau, y}^{t_1}(\x) \right)\right|_{t=r}$ exists, that $\lim _{h \rightarrow 0+} \int_r^{r+h} f(t) d t= f(r)$, and that \eqref{Eqn_condition_thm2} holds. Then,
$$
\begin{aligned}
\left.\frac{d}{d t} v\left(t, x(t), \vee_{\tau, y}^{t_1}(\x) \right)\right|_{t=r} & =\lim _{h \rightarrow 0+} \frac{v\left(r+h, x(r+h), \vee_{\tau, y}^{r+h}(\x)\right)-v\left(r, x(r), \vee_{\tau, y}^r(\x)\right)}{h} \\
& =v^{\prime}(\left(r, x(r), \vee_{\tau, y}^r(\x)\right) \\
& \leqslant p_\tau(r)+p_x(r) \cdot f(r) + p_y(r) f_y(r) \mathds{1} _{\{y =  L_{mx}(r,x(r)), f_y(r)>0\}}\quad \text { (by \cite[Lemma 2.1]{zhou1993verification}) } \\
& =-L_r(r) \quad \text { (by \eqref{Eqn_condition_thm2}). }
\end{aligned}
$$

Hence we conclude that
$$
v\left(T, x(T), \vee_{\tau, y}^{t_1}(\x)\right)-v(\tau, x,y)=\left.\int_\tau^{t_1} \frac{d}{d t} v\left(t, x(t), \vee_{\tau, y}^{t_1}(\x)\right)\right|_{t=r} d r \leqslant-\int_\tau^{t_1} L_r(r) d r,
$$
which implies
$$
J\left(\tau, x, y ; u^*\right)=\int_\tau^{t_1} L_r(r) d r- \vee_{\tau, y}^{t_1} ( \x) + \Psi(x(t_1))  \leqslant v(\tau,x, y).
$$

Therefore, it follows from \cite[Theorem 2.3, Part (a)]{zhou1993verification}  that  $v(\tau,x,y) \le J(\tau,x, y;u(\cdot))$, for $(\tau,x, y) \in T \times \mathbb{R}^{n+1} $ and for any $u \in \mathbb{U}$. Thus, $\u^*$ is an optimal control.

Now, we will prove that given admissible tuple of functions $(\x^*, \u^*)$   is optimal for the control problem  then for almost all $t$;   there exists a tuple 
          $( p_\tau^*(t), p_x^*(t), p_y^*(t)) \in D^+  v(t, x^*(t), \vee_{\tau, y}^{t_1}(\x^*))$  satisfying, 
     $      p_\tau^*(t)=  H(t, x^*(t), \vee_{\tau, y}^{t_1}(\x^* ),u, p_x^*(t),  p_y^*(t)). $

By maximum principle, there exists  functions $\psi^1(\cdot)$ and $\psi^2(\cdot)$ corresponding to the optimal pair $(\x^*, \u^*)$ that satisfies the following
\begin{eqnarray*}
 H(t, x, \vee_{\tau, y}^{t_1}(\x), u, \psi^1, \psi^2)&:=& \psi^1 f(t, x, u) + \psi^2 f_y(t, x, u) \mathds{1} _{\{y =  L_{mx}(t,x), f_y(t,x,u)>0\}} + L_r(t, x, u)\\
 H_x &=& \psi^1_x f+ \psi^1 f_x + (\psi^2_x f+ \psi^2 f_x )\mathds{1} _{\{y \le L_{mx}(t,x)\}} + L_x
\end{eqnarray*}
\begin{eqnarray*}
    \bdot{\psi}^1(t) &=& H_x\left(t, x^*(t), \vee_{\tau, y}^{t_1}(\x^*),   u^*(t), \psi^1(t), \psi^2(t)\right) \quad \mbox{for almost all } t \in [\tau, T), \\
    \bdot{\psi}^2(t) &=& H_y\left(t, x^*(t), \vee_{\tau, y}^{t_1}(\x^*),   u^*(t), \psi^1(t), \psi^2(t)\right) \quad \mbox{for almost all } t \in [\tau, T), \\
    \psi^1(T) &=& \sigma \frac{d \vee_{\tau, y}^{t_1}(x^*(T))}{dx} + \frac{d\Psi(x^*(T))}{dx}\\
    \psi^2(T) &=& \sigma \frac{d \vee_{\tau, y}^{t_1}(x^*(T))}{dy}. 
\end{eqnarray*}
Then by \cite[Theorem 3.2]{zhou1990maximum}, for almost all $t \in [\tau, T)$, 
$$
\bigg(H\left(t, x^*(t), \vee_{\tau, y}^{t_1}(\x^*),   u^*(t), \psi^1(t), \psi^2(t)\right), \psi^1(t), \psi^2(t) \bigg) \in D^+ v(t, x^*(t), \vee_{\tau, y}^{t_1}(\x^*)).
$$
This completes the result. \eop

{\color{red} \textbf{Plan:}
\begin{enumerate}
    \item By Weierstrass Theorem, we have existence of solution to control problem. This implies value function $v$ is well defined. 
    \item Berg Maximum Theorem  
    \item Using Fleming and Sonar, page $99$, try to prove that the value function (Existence is proved by WT) is unique and Lipschitz continuous.
    \item Assume the existence of viscosity solution of HJB. 
    \item Verification theorem and all. 
\end{enumerate}

}}}

\bibliographystyle{IEEEtran}
\bibliography{ref}

\end{document}